\documentclass[twoside,11pt]{article}

\usepackage{amsmath,amsthm,amsfonts,dsfont}
\usepackage[hidelinks]{hyperref} 
\usepackage[noabbrev,capitalise]{cleveref}
\usepackage{mathtools}

\usepackage[abbrvbib,nohyperref,preprint]{jmlr2e}
\renewenvironment{proof}[1][\proofname]{\par\noindent{\bf #1\ }}{\hfill\BlackBox\\[2mm]}

\usepackage{lastpage}
\jmlrheading{xx}{xxxx}{1-\pageref{LastPage}}{12/xx}{x/xx}{xx-xxxx}{Gabriel Rioux, Ziv Goldfeld, and Kengo Kato}

\ShortHeadings{Entropic Gromov-Wasserstein Distances:
Stability and Algorithms}{Rioux, Goldfeld, and Kato}
\firstpageno{1}

\usepackage{float}
\usepackage[utf8]{inputenc} %
\usepackage[T1]{fontenc}    %

\usepackage{comment}
\usepackage{subcaption}
\usepackage{wrapfig}
\usepackage{bm}
\usepackage{url}            %
\usepackage{booktabs}       %
\usepackage{nicefrac}       %
\usepackage{microtype}      %

\usepackage{bbm}
\usepackage{cite}
\usepackage{psfrag}
\usepackage{cite}
\usepackage{color,soul}
\usepackage{breakcites}     %
\usepackage{bigints}

\usepackage{enumitem}
\setlist{leftmargin=1.6em}

\usepackage{xcolor}
\usepackage{svg}
\usepackage{algorithm}
\usepackage[noend]{algpseudocode}

\newcommand{\Var}[0]{\mathrm{Var}}

\newcommand{\diag}[0]{\mathrm{diag}}
\newcommand{\sign}[0]{\mathrm{sign}}
\newcommand{\Tr}[0]{\mathrm{tr}}

\newcommand{\calC}[0]{\mathcal{C}}

\newcommand{\calP}[0]{\mathcal{P}}

\newcommand{\supp}{\mathrm{spt}}

\newcommand{\R}[0]{\mathbb{R}}

\newcommand{\vertiii}[1]{{\left\vert\kern-0.25ex\left\vert\kern-0.25ex\left\vert #1 
    \right\vert\kern-0.25ex\right\vert\kern-0.25ex\right\vert}}

\newcommand{\sS}{\mathsf{S}}

\newcommand{\cD}{\mathcal{D}}

\newcommand{\cP}{\mathcal{P}}

\newcommand{\cR}{\mathcal{R}}

\newcommand{\cX}{\mathcal{X}}

\newcommand{\EE}{\mathbb{E}}

\newcommand{\RR}{\mathbb{R}}

\newcommand{\vasti}{\bBigg@{3.5 }}
\newcommand{\vast}{\bBigg@{4}}
\newcommand{\Vast}{\bBigg@{5}}
\newcommand{\Vastt}{\bBigg@{7}}

\makeatother

\newcommand{\be}{\begin{equation}}
\newcommand{\ee}{\end{equation}}
\newcommand{\ba}{\begin{align}}
\newcommand{\ea}{\end{align}}
\newcommand{\baa}{\begin{align*}}
\newcommand{\eaa}{\end{align*}}

\newcommand{\argmin}{\mathop{\mathrm{argmin}}}
\newcommand{\argmax}{\mathop{\mathrm{argmax}}}

\newcommand{\KL}{\mathsf{D}_{\mathsf{KL}}}
\newcommand{\OT}{\mathsf{OT}}

\newcommand{\edit}[1]{{#1}}%

\DeclareMathOperator{\interior}{int}
\DeclareMathOperator{\diam}{diam}

\DeclareMathOperator{\Id}{Id}

\DeclareMathOperator{\conv}{conv}

\begin{document}

\title{Entropic Gromov-Wasserstein Distances:\\
Stability and Algorithms}

\author{\name Gabriel Rioux
\email ger84@cornell.edu \\
       \addr Center for Applied Mathematics\\
       Cornell University\\
       Ithaca, NY 14853, USA
       \AND
       \name  Ziv Goldfeld
\email goldfeld@cornell.edu \\
       \addr School of Electrical and Computer Engineering\\
       Cornell University\\
       Ithaca, NY 14853, USA
       \AND
       \name  Kengo Kato
\email kk976@cornell.edu \\
       \addr Department of Statistics and Data Science\\
       Cornell University\\
       Ithaca, NY 14853, USA
       }

\editor{}

\maketitle

\begin{abstract}%
\edit{
The Gromov-Wasserstein (GW) distance quantifies discrepancy between metric measure spaces and provides a natural framework for aligning heterogeneous datasets. Alas, as exact computation of GW alignment is NP hard, entropic regularization provides an avenue towards a computationally tractable proxy. Leveraging a recently derived variational representation for the quadratic entropic GW (EGW) distance, this work derives the first efficient algorithms for solving the EGW problem subject to formal, non-asymptotic convergence guarantees. To that end, we derive smoothness and convexity properties of the objective in this variational problem, which enables its resolution by the accelerated gradient method. Our algorithms employs Sinkhorn's fixed point iterations to compute an approximate gradient, which we model as an inexact oracle. We furnish convergence rates towards local and even global solutions (the latter holds under a precise quantitative condition on the regularization parameter), characterize the effects of gradient inexactness, and prove that stationary points of the EGW problem converge towards a stationary point of the unregularized GW problem, in the limit of vanishing regularization. We provide numerical experiments that validate our theory and empirically demonstrate the state-of-the-art empirical performance of our~algorithm.} 

\end{abstract}


\section{Introduction}

The Gromov-Wasserstein (GW) distance compares probability distributions that are supported on possibly distinct metric spaces by aligning them with one another. Given two metric measure (mm) spaces $(\mathcal X_0,\mathsf d_0,\mu_0)$ and $(\mathcal X_1,\mathsf d_1,\mu_1)$, the $(p,q)$-GW distance between them is 
\begin{equation}
\label{eq:GW_intro}
    \mathsf D_{p,q}(\mu_0,\mu_1)\coloneqq\inf_{\pi\in\Pi(\mu_0,\mu_1)} \left(\int_{\cX_0 \times \cX_1} \int_{\cX_0 \times \cX_1} \left|\mathsf d_0^q(x,x')-\mathsf d_1^q(y,y')\right|^p d\pi\otimes \pi(x,y,x',y')\right)^{\frac{1}{p}},
\end{equation}
where $\Pi(\mu_0,\mu_1)$ is the set of couplings between $\mu_0$ and $\mu_1$. This approach, proposed in \citet{Memoli11}, is an optimal transport (OT) based $L^p$-relaxation of the classical Gromov-Hausdorff distance between metric spaces. The GW distance defines a metric on the space of all mm spaces modulo measure preserving isometries.\footnote{Two mm spaces $(\mathcal X_0,\mathsf d_0,\mu_0)$ and $(\mathcal X_1,\mathsf d_1,\mu_1)$ are isomorphic if there exists an isometry $T:\mathcal X_0\to\mathcal{X}_1$ for which $\mu_0\circ T^{-1}=\mu_1$ as measures. The quotient space is then the one induced by this equivalence relation.} From an applied standpoint, a solution to the GW problem between two heterogeneous datasets yields not only a quantification of discrepancy, but also an optimal alignment $\pi^\star$ between them. As such, alignment methods inspired by the GW problem have been proposed for many applications, encompassing single-cell genomics \citep{blumberg2020mrec,demetci2020gromov}, alignment of language models \citep{alvarez2018gromov}, shape matching \citep{memoli2009spectral,koehl2023computing}, graph matching \citep{xu2019gromov,xu2019scalable}, heterogeneous domain adaptation \citep{yan2018semi}, and generative modeling \citep{bunne2019learning}.

Exact computation of the GW distance is a quadratic assignment problem, which is known to be NP-complete \citep{Commander2005}. To remedy this, various computationally tractable reformulations of the distance have been proposed. We postpone full discussion of such methods to \cref{subsec:literature} and focus here on the entropic GW (EGW) distance \citep{peyre2016gromov,solomon2016entropic}
\[
\sS_{p,q}^\varepsilon(\mu_0,\mu_1)\coloneqq\inf_{\pi\in\Pi(\mu_0,\mu_1)} \iint \left|\mathsf d_0^q(x,x')-\mathsf d_1^q(y,y')\right|^p d\pi\otimes \pi(x,y,x',y')+ \varepsilon \mathsf{D}_{\mathsf {KL}}(\pi\|\mu_0\otimes \mu_1),
\]
which is at the center of this work. Entropic regularization by means of the Kullback-Leibler (KL) divergence above transforms the linear Kantorovich OT  problem \citep{kantorovich1942translocation} into a strictly convex one and enables directly solving it using Sinkhorn's algorithm \citep{cuturi2013sinkhorn}. Although the EGW problem is, in general, not convex, \citet{solomon2016entropic} propose to solve it via an iterative approach with Sinkhorn iterations. This method is known to converge to a stationary point of a tight (albeit non-convex) relaxation of the EGW problem, but this is an asymptotic statement and the overall computational complexity is unknown. Similar limitations apply for the popular mirror descent-based approach from \citet{peyre2016gromov}. To the best of our knowledge, there is currently no known algorithm for computing the EGW distance subject to non-asymptotic convergence rate bounds, let alone global optimality guarantees. \edit{Further, these methods do not account for the error incurred by using Sinkhorn iterations, nor do they address the behaviour of the solutions they obtain in the limit of vanishing regularization.}

The goal of this work is to close the aforementioned computational gaps, \edit{targeting algorithms with non-asymptotic guarantees, accounting for inexactness in Sinkhorn's algorithm, characterizing convexity regimes of the EGW problem, and establishing convergence of stationary points to the EGW problem to stationary points of the GW problem as $\varepsilon\downarrow 0$}. All of these will be achieved as a consequence of a new stability analysis of the EGW variational representation from \citet{Zhang22Gromov}. %

\subsection{Contributions}

Theorem 1 in \citet{Zhang22Gromov} shows that the EGW distance with quadratic cost between the Euclidean mm spaces $(\RR^{d_0},\|\cdot\|,\mu_0)$ and $(\RR^{d_1},\|\cdot\|,\mu_1)$ can be written as:
\begin{equation}
\sS_{2,2}^\varepsilon(\mu_0,\mu_1)=C_{\mu_0,\mu_1}+\inf_{\bm A\in \cD_M} \big\{32\|\bm A\|_F^2+\mathsf{OT}_{\bm A,\varepsilon}(\mu_0,\mu_1)\big\},\label{eq:dual_intro}
\end{equation}
where $C_{\mu_0,\mu_1}$ is a constant that depends only on the moments of the marginals, $\cD_M \subset \RR^{d_0\times d_1}$ is a compact rectangle, and $\mathsf{OT}_{\bm A,\varepsilon}(\mu_0,\mu_1)$ is an EOT problem with a particular cost function that depends on the auxiliary variable $\bm A$. This representation connects EGW to the well-understood EOT problem, thereby unlocking powerful tools for analysis. To exploit this connection for new computational advancements, we begin by analyzing the stability of the objective in \eqref{eq:dual_intro} in $\bm A$ and derive its first and second-order Fr{\'e}chet derivatives. This, in turn, enables us to derive its convexity and smoothness properties and devise new algorithms for solving the EGW problem, subject to formal convergence guarantees.

The Fr{\'e}chet derivatives of the objective function from  \eqref{eq:dual_intro} in $\bm{A}$ reveal that this problem falls under the paradigm of smooth constrained optimization. Indeed, the derivatives imply, respectively, upper and lower bounds on the top and bottom eigenvalues of the Hessian matrix of the objective; $L$-smoothness then follows from the mean value inequality. By further requiring positive semidefiniteness of the Hessian we obtain a sharp and primitive sufficient condition on $\varepsilon$ under which \eqref{eq:dual_intro} becomes convex. 
These regularity properties are used to lift the accelerated first-order methods for smooth (non-convex)  optimization \citep{ghadimi2016accelerated} and convex programming with an inexact oracle \citep{Aspremont2008smooth} to the EGW problem.
This yields the first algorithms with non-asymptotic convergence guarantees toward global or local EGW solutions, depending on whether the problem is convex or not (e.g., under the aforementioned condition). Our algorithms compute not only the EGW cost, but also provide an approximate optimizer---namely a coupling, which serves as the alignment scheme that achieves the said cost.

Specifically, our method iteratively solves the optimization problem in the space of auxiliary matrices $\bm A\in\RR^{d_0\times d_1}$, with each iterate calling the Sinkhorn algorithm to obtain an approximate solution (viz. the inexact oracle) to the corresponding EOT problem. The time complexity of the Sinkhorn algorithm governs the overall runtime, which is therefore $O(N^2)$, for $\mu_0$ and $\mu_1$ as distributions on $N$ points. This presents a significant speedup to the $O(N^3)$ runtime of popular iterative algorithms from \citet{peyre2016gromov,solomon2016entropic}. Under certain low-rank assumptions on the cost matrix, \citet{scetbon2022linear} recently showed the mirror descent approach from \citet{peyre2016gromov} can be sped up to run in $O(N^2)$ time, which is comparable to our method. Nevertheless, our algorithms are coupled with~formal convergence guarantees, non-asymptotic error bounds, and global optimality claims under the said convexity condition, while no such assurances are available for other~methods. 

\edit{As the derivative of the objective from \eqref{eq:dual_intro} depends on the optimal coupling for a particular EOT problem, we also account for the error incurred by solving this problem numerically (e.g. via Sinkhorn iterations). In particular, we characterize how well the output of a standard implementation of Sinkhorn's algorithm approximates the true EOT coupling. This effect was not analyzed before, as existing literature focused on approximating the unregularized OT cost, while treating the KL divergence term as a bias. }

\edit{The EGW distance serves as a proxy to unregularized GW, which renders the vanishing regularization parameter regime, namely $\varepsilon\downarrow 0$, of central importance. %
We show that stationary points of the variational EGW problem converge (possibly along a subsequence) to a stationary point of the variational GW problem %
 as $\varepsilon\downarrow 0$. Only convergence to a stationary point can be guaranteed in the limit since the underlying variational problem  may fail to be convex when $\varepsilon$ is small. Nonetheless, this is the first result that provides stationarity guarantees for limiting solutions, which clarifies how the local solutions obtained using our algorithm approximate local solutions to the unregularized problem. }

\subsection{Literature review}
\label{subsec:literature} 

The computational intractability of the GW problem in \eqref{eq:GW_intro} has inspired several reformulations that aim to alleviate this issue.  %
The sliced GW distance \citep{vayer2020sliced} attempts to reduce the computational burden by considering the average of GW distances between one-dimensional projections of the marginals. However, unlike OT in one dimension, the GW problem does not have a known simple solution even in one dimension \citep{beinert2022assignment}. Another approach is to relax the strict marginal constraints by optimizing over the weights of one of the marginals as in semi-relaxed GW \citep{vincent2022semi} or by using $f$-divergence penalties; this leads to the unbalanced GW distance \citep{sejourne2021unbalanced}, which lends itself well for convex/conic relaxations. A variant that directly optimizes over bi-directional Monge maps between the mm spaces was considered in \citet{zhang2022cycle}. \edit{The fused GW distance \citep{vayer2018optimal} enables comparing both feature and structural properties of structured data.} \edit{Although most of these relaxations can be shown to converge to the GW distance (in terms of function values) under certain regimes, they involve solving non-convex problems, which limits their utility for numerical resolution of the GW problem. The recent work of \citet{chen2023semidefinite} proposes a semidefinite relaxation of the GW problem along with a certificate of optimality which, upon obtaining a solution to the relaxed problem,  establishes if it is optimal for the original problem. } 

While these methods offer certain advantages, it is the approach based on entropic regularization \citep{peyre2016gromov,solomon2016entropic} that is most frequently used in application. A low-rank variant of the EGW problem was proposed in \citet{scetbon2022linear}, where the distance distortion cost is only optimized over coupling admitting a certain low-rank structure. They arrive at a linear time algorithm for this problem by adapting the mirror descent method of \citet{peyre2016gromov}.
As an intermediate step of their analysis, they show that if $\mu_0$ and $\mu_1$ are supported on $N$ distinct points, then the $O(N^3)$ complexity of mirror descent (see, e.g., Remark 1 in \citealt{peyre2016gromov}) can be reduced to $O(N^2)$ by assuming that the matrices of pairwise costs admit a low-rank decomposition (without imposing any structure on the couplings). This decomposition holds, for instance, when the cost is the squared Euclidean distance and the sample size dominates the ambient dimension. Although mirror descent seems to solve the EGW problem quite well in practice, formal guarantees concerning convergence rates or local optimality are lacking.%

Other related work explores structural properties of GW distances, on which some of our findings also reflect. The existence of Monge maps for the GW problem was studied in \citet{dumont2022existence} and they show that optimal couplings are induced by a bimap (viz. two-way map) under general conditions. \citet{delon2022gromov} focused on the GW distance between Gaussian distributions, deriving upper and lower bounds on the optimal cost. A closed-form expression under the Gaussian setting was derived in \citet{le2022entropic} for the EGW distance with inner product cost.

\edit{
As we establish stability results for the EGW problem by utilizing its connection to the EOT problem with a parametrized cost, we mention that different notions of stability for EOT have been studied. For instance, \citet{ghosal2022stability} concerns stability of the EOT cost for varying marginals, cost function, and regularization parameter. The related works \citep{Carlier2020Differential,eckstein2022quantitative,nutz2023stability} concern stability of the EOT cost and related objects for weakly convergent marginal distributions.
}

\subsection{Organization} 

\label{sec: organization}

This paper is organized as follows. In \cref{sec: background} we compile background material on EOT and the EGW problems. In 
\cref{sec:asymptoticNormalityCompact}, we describe the smoothness and convexity of the variational EGW problem. 
In \cref{sec:Computation}, we analyze and test two algorithms for solving the EGW problem. 
We compile the proofs for Sections \ref{sec:asymptoticNormalityCompact} and  \ref{sec:Computation} %
in \cref{sec:Proofs}. \cref{sec: summary} contains some concluding remarks.

\subsection{Notation}\label{subsec:Notation} Denote by  $\calP(\mathbb R^d)$ the collection of all Borel probability measures on $\mathbb R^d$, %
and by %
$\mathcal P_p(\mathbb R^d)$ the set of all $\mu \in \calP(\R^d)$ with finite $p$-th moment ($p>0$). 
The pushforward of $\mu\in\mathcal P(\mathbb R^{d_0})$ through a measurable map $T:\mathbb R^{d_0}\to \mathbb R^{d_1}$ is denoted by $T_{\sharp}\mu := \mu \circ T^{-1}$. 
The Frobenius inner product on $\mathbb R^{d_0\times d_1}$ is defined by $\langle \bm A,\bm B\rangle_F=\Tr\left(\bm A^{\intercal}\bm B\right)$; the associated norm is denoted by $\|\cdot\|_F$. For a nonemtpy set $S\subset\mathbb R^d$, $\mathcal C(S)$ is the set of continuous functions on $S$. We adopt the shorthands $a\wedge b=\min\{a,b\}$ and $a\vee b=\max\{a,b\}$.

\section{Background and Preliminaries}\label{sec: background}

We first establish notation %
and review standard definitions and results underpinning our analysis of the EGW distance.

\subsection{Entropic Optimal Transport}\label{subsec:EOT}
Entropic regularization transforms the linear OT problem into a strictly convex one. Given distributions $\mu_i\in\cP(\RR^{d_i})$, $i=0,1$, and a Borel cost function $c:\RR^{d_0}\times \RR^{d_1}\to\RR$ that is bounded from below on $\supp(\mu_0) \times \supp(\mu_1)$, the primal EOT problem is obtained by regularizing the standard OT problem via the Kullback-Leibler (KL) divergence,
\[
    \OT_{\varepsilon}(\mu_0,\mu_1)=\inf_{\pi\in\Pi(\mu_0,\mu_1)} \int c\,d\pi+\varepsilon \KL(\pi\|\mu_0\otimes \mu_1), 
\]
where $\varepsilon>0$ is a regularization parameter and 
\[
    \KL(\mu_0\|\mu_1)=\begin{cases} \int \log \left(\frac{d\mu_0}{d\mu_1}\right)d\mu_0,&\text{if }\mu_0\ll\mu_1,\\+\infty,&\text{otherwise.}
    \end{cases}
\]
Classical OT is obtained from the above by setting $\varepsilon=0$. When $c\in L^1(\mu_0\otimes \mu_1)$,
EOT admits the following dual formulation, 
\[
    \OT_{\varepsilon}(\mu_0,\mu_1)=\sup_{(\varphi_0,\varphi_1)\in L^1(\mu_0)\times L^1(\mu_1)}\int \varphi_0d\mu_0+\int\varphi_1d\mu_1-\varepsilon \int e^{\frac{\varphi_0\oplus\varphi_1-c}{\varepsilon}}d\mu_0\otimes \mu_1+\varepsilon,
\]
where $\varphi_0\oplus \varphi_1(x,y)=\varphi_0(x)+\varphi_1(y)$. For $\varepsilon>0$, the set of solutions to the dual problem coincides with the set of solutions to the so-called Schr{\"o}dinger system, 
\begin{equation}
\label{eq:SchrodingerSystem}
\begin{aligned}
\int e^{\frac{\varphi_0(x)+\varphi_1(\cdot)-c(x,\cdot)}{\varepsilon}}d\mu_1=1,\quad\text{$\mu_0$-a.e. $x\in\RR^{d_0}$},\\
\int e^{\frac{\varphi_0(\cdot)+\varphi_1(y)-c(\cdot,y)}{\varepsilon}}d\mu_0=1,\quad\text{$\mu_1$-a.e. $y\in\RR^{d_1}$},
\end{aligned}
\end{equation}
for $(\varphi_0,\varphi_1)\in L^1(\mu_0)\times L^1(\mu_1)$. 
A pair $(\varphi_0,\varphi_1)\in L^1(\mu_0)\times L^1(\mu_1)$ solving \eqref{eq:SchrodingerSystem} is known to be a.s. unique up to additive constants in the sense that any other solution $(\bar\varphi_0,\bar\varphi_1)$ satisfies $\bar\varphi_0=\varphi_0+a$ $\mu_0$-a.s. and $\bar\varphi_1=\varphi_1-a$ $\mu_1$-a.s. for some $a \in \R$. 
Moreover, the {unique} EOT coupling $\pi_{\varepsilon}$ is characterized by
\begin{equation}
\label{eqn:EOTcoupling}   
   \frac{d\pi_{\varepsilon}}{d\mu_0\otimes \mu_1}(x,y)=e^{\frac{\varphi_0(x)+\varphi_1(y)-c(x,y)}{\varepsilon}},
\end{equation}
and, under some additional conditions on the cost and marginals which hold throughout this paper, \eqref{eq:SchrodingerSystem} admits a pair of continuous solutions which is unique up to additive constants and satisfies the system at all points $(x,y)\in \RR^{d_0}\times \RR^{d_1}$. We call such continuous solutions EOT potentials. The reader is referred to \citet{nutz2021introduction} for a comprehensive overview of EOT. %

\subsection{Entropic Gromov-Wasserstein Distance}\label{subsec:EGW}
This work studies stability and computational aspects of the entropically regularized GW distance under the quadratic cost. By analogy to OT, EGW serves as a proxy of the standard $(p,q)$-GW distance, which quantifies discrepancy between complete and separable mm spaces $(\cX_0,\mathsf d_0,\mu_0)$ and $(\cX_1,\mathsf d_1,\mu_1)$ as \citep{Memoli11,sturm2012space}
\[
    \mathsf{D}_{p,q}(\mu_0,\mu_1)\coloneqq \inf_{\pi\in\Pi(\mu_0,\mu_1)} \|\Gamma_q\|_{L^p(\pi\otimes\pi)},
\]
where $\Gamma_{q}(x,y,x',y')=\big| \mathsf d_0^q(x,x')-\mathsf d_1^q(y,y')\big|$ is the distance distortion cost. This definition is the $L^p$-relaxation of the Gromov-Hausdorff distance between metric spaces,\footnote{The Gromov-Hausdorff distance between $(\cX_0,\mathsf d_0)$ and $(\cX_1,\mathsf d_1)$ is given by $\frac 12 \inf_{R\in\cR(\cX_0,\cX_1)}\|\Gamma_1\|_{L^\infty(R)}$, where $\cR(\cX_0,\cX_1)$ is the collection of all correspondence sets of $\cX_0$ and $\cX_1$, i.e., subsets $R\subset\cX_0\times\cX_1$ such that the coordinate projection maps
are surjective when restricted to $R$. The correspondence set can be thought of as $\supp(\pi)$ in the GW formulation.} and gives rise to a metric on the collection of all isomorphism classes of mm spaces %
with finite $pq$-size, i.e., such that $\int \mathsf d(x,x')^{pq}\,d\mu\otimes \mu(x,x')<\infty$.

\medskip
From here on out, we instantiate the mm spaces as the Euclidean spaces $(\RR^{d_i},\|\cdot\|,\mu_i)$, for $i=0,1$, and {focus on the EGW distance for the quadratic cost}.

\paragraph{Quadratic Cost}
\label{subsubsec:EGWQuadCost}
The quadratic EGW distance, which corresponds to the $p=q=2$ case, is defined as
\begin{equation}
\label{eq:QuadEGW}
    \mathsf{S}_{\varepsilon}(\mu_0,\mu_1)=\inf_{\pi\in\Pi(\mu_0,\mu_1)}\int \left|\|x-x'\|^2-\|y-y'\|^2\right|^2d\pi\otimes\pi(x,y,x',y')+\varepsilon \KL(\pi\|\mu_0\otimes \mu_1).   
\end{equation}
One readily verifies that, like the standard GW distance, EGW is invariant to isometric actions on the marginal spaces such as orthogonal rotations and translations. In addition,  note that $\sS_{\varepsilon}(\mu_0,\mu_1)=\varepsilon\sS_1(\mu_0^{\varepsilon},\mu_1^{\varepsilon})$, where $\mu_i^{\varepsilon}=(\varepsilon^{-1/4}\Id)_{\sharp}\mu_i$. 
\edit{In general, \eqref{eq:QuadEGW} is a non-convex quadratic program. Non-convexity can easily be discerned from the  representation \eqref{eq:EGWDecomp}.}

\medskip
When $\mu_0,\mu_1$ are centered, which we may assume without loss of generality, the EGW distance decomposes\footnote{A similar decomposition holds for the inner product cost, obtained by replacing the squared Euclidean distances in \cref{eq:QuadEGW} by inner products. As such, all results derived in this manuscript apply to the inner product cost with minor modifications.} as (cf. Section 5.3 in \citealt{Zhang22Gromov})%
\begin{equation}
    \sS_{\varepsilon}(\mu_0,\mu_1)=\sS^1(\mu_0,\mu_1)+\sS^{2}_{\varepsilon}(\mu_0,\mu_1),\label{eq:EGWDecomp}
\end{equation}
where %
\begin{align*}
&\sS^1(\mu_0,\mu_1)\mspace{-3mu}=\mspace{-3mu}\int\mspace{-3mu}\|x\mspace{-2mu}-\mspace{-2mu}x'\|^4d\mu_0\mspace{-2mu}\otimes\mspace{-2mu}\mu_0(x,x')\mspace{-3mu}+\mspace{-3mu}\int\mspace{-3mu}\|y\mspace{-2mu}-\mspace{-2mu}y'\|^4d\mu_1\mspace{-2mu}\otimes\mspace{-2mu}\mu_1(y,y')\mspace{-3mu}-\mspace{-3mu}4\mspace{-4mu}\int\mspace{-3mu}\|x\|^2\|y\|^2d\mu_0\mspace{-2mu}\otimes\mspace{-2mu}\mu_1(x,y),\\
&\sS^{2}_{\varepsilon}(\mu_0,\mu_1)\mspace{-3mu}=\mspace{-3mu}\inf_{\pi\in\Pi(\mu_0,\mu_1)}\int\mspace{-3mu}-4\|x\|^2\|y\|^2d\pi(x,y)\mspace{-3mu}-\mspace{-3mu}8\mspace{-3.5mu}\sum_{\substack{1\leq i\leq d_0\\1\leq j\leq d_1}}\mspace{-4mu}\left(\int\mspace{-3mu} x_iy_j\,d\pi(x,y)\right)^2\mspace{-3mu}+\mspace{-3mu}\varepsilon\KL(\pi\|\mu_0\mspace{-2mu}\otimes\mspace{-2mu} \mu_1).
\end{align*}
Evidently, $\sS^1$ depends only on the moments of the marginal distributions $\mu_0,\mu_1$, while $\sS^{2}_{\varepsilon}$ captures the dependence on the coupling. A key observation in \citet{Zhang22Gromov} is that $\sS^{2}_{\varepsilon}$ admits a variational form that ties it to the well understood EOT problem. %

\begin{lemma}[EGW duality; Theorem 1 in \citealt{Zhang22Gromov}]
\label{lem:EGWDual}
    Fix $\varepsilon>0$, $(\mu_0,\mu_1)\in\cP_4(\RR^{d_0})\times\cP_4(\RR^{d_1})$, and let $M_{\mu_0,\mu_1}\coloneqq\sqrt{M_2(\mu_0)M_2(\mu_1)}$. Then, for any $M\geq M_{\mu_0,\mu_1}$, 
    \begin{equation}
    \label{eq:S2Decomp}
       \sS^{2}_{\varepsilon}(\mu_0,\mu_1)=\inf_{\bm A\in\cD_M}32\|\bm A\|_F^2+\OT_{\bm A,\varepsilon}(\mu_0,\mu_1),
     \end{equation}
    where $\cD_M\coloneqq\{\bm A\in\mathbb R^{d_0\times d_1}:\|\bm A\|_F\leq M/2\}$ and $\OT_{\bm A,\varepsilon}(\mu_0,\mu_1)$ is the EOT problem with the cost function $c_{\bm A}:(x,y)\in\RR^{d_0}\times \RR^{d_1}\mapsto-4\|x\|^2\|y\|^2-32x^{\intercal}\bm Ay$ and regularization parameter $\varepsilon$. 
    Moreover, the infimum is achieved at some $\bm A^{\star}\in\cD_{M_{\mu_0,\mu_1}}$. %
\end{lemma}

\edit{An analogous result holds in the unregularized ($\varepsilon=0$) case, see Corollary 1 in \citet{Zhang22Gromov}.}
The proof of Theorem 1 in \citet{Zhang22Gromov} demonstrates that if $\mu_0$ and $\mu_1$ are centered and $\pi_{\star}$ is optimal for the original EGW formulation, then $\bm A^{\star}=\frac{1}{2}\int xy^{\intercal}\,d\pi_{\star}(x,y)$ is optimal for \eqref{eq:S2Decomp} %
{and $\pi_{\star} = \pi_{\bm A{^\star}}$}, where 
{$\pi_{\bm A^{\star}}$ is the unique EOT coupling for $\OT_{\bm A^{\star},\varepsilon}(\mu_0,\mu_1)$}. It follows from Jensen's inequality and the Cauchy-Schwarz inequality that $\bm A^{\star}\in\mathcal D_M$.  \cref{cor:StationaryPoint} ahead expands on this connection by establishing a one-to-one correspondence between solutions of $\sS_{\varepsilon}$ and $\sS^{2}_{\varepsilon}$ and shows that all solutions of \eqref{eq:S2Decomp} lie in $\mathcal D_M$.

Although \eqref{eq:S2Decomp} illustrates a connection between the EGW and EOT problems, the outer minimization over $\cD_M$ necessitates studying EOT with a parametrized cost function $c_{\bm A}$. 

\section{Stability of Entropic Gromov-Wasserstein Distances}
\label{sec:asymptoticNormalityCompact}
\label{subsec:QuadraticDerivative}
We analyze the stability of the EGW problems with respect to (w.r.t.) the matrix $\bm A$ 
appearing in the dual formulation \eqref{eq:S2Decomp}. %
Specifically, we characterize the first and second derivatives of the objective function whose optimization defines $\sS^{2}_{\varepsilon}$. These, in turn, elucidates its smoothness and convexity properties. %
Our stability analysis is later used to (i) gain insight into the structure of optimal couplings for the EGW problem; and (ii) devise novel approaches for computing the EGW distance with formal convergence guarantees.

\medskip
Throughout this section, we restrict attention to compactly supported distributions, as some of the technical details do not directly translate to the unbounded setting (e.g.,~the proof of \cref{lem:BanachSpaceIsomorphism}). For a Fr{\'e}chet differentiable map $F:U\to V$ between normed vector spaces $U$ and $V$,\footnote{ A map $F:U\to V$ is Fr{\' e}chet differentiable at $u\in U$ if there exists a bounded linear operator $A:U\to V$ for which $F(u+h)=F(u)+Ah+o(h)$ as $h\to 0$. If such an $A$ exists, it is called the Fr{\'e}chet derivative to $F$ at $u$.} we denote the derivative of $F$ at the point $u\in U$ evaluated at $v\in V$ by $D F_{[u]}(v)$.

 Fix compactly supported distributions $(\mu_0,\mu_1)\in\cP(\RR^{d_0})\times\cP(\RR^{d_1})$ and some $\varepsilon>0$. Let 
\[
\Phi:\bm A \in \mathbb R^{d_0\times d_1}\mapsto 32\|\bm A\|^2_{F}+\OT_{\bm A,\varepsilon}(\mu_0,\mu_1)
\] 
denote the objective in \eqref{eq:S2Decomp}. 
We first characterize the derivatives of $\Phi$ and then prove that this map is weakly convex and $L$-smooth.\footnote{\label{footnote:weakConvLSmooth}A function $f:\RR^d\to \RR$ is $\rho$-weakly convex if $f+\frac{\rho}{2}\|\cdot\|^2$ is convex; $f$ is $L$-smooth if its gradient is $L$-Lipschitz, i.e., $\|\nabla f(x)-\nabla f(y)\|\leq L\|x-y\|$, for all $x,y\in\RR^d$.}

\begin{proposition}[First and second derivatives]
    \label{prop:PhiDerivative}
   The map $\Phi:\bm A\in\RR^{d_0\times d_1}\mapsto 32\|\bm A\|_F^2+\OT_{\bm A,\varepsilon}(\mu_0,\mu_1)$ is smooth, coercive, and has first and second-order Fr{\'e}chet derivatives at $\bm A\in \RR^{d_0\times d_1}$ given by
   \begin{align*}
        D\Phi_{[\bm A]}(\bm B)&=64\,\Tr(\bm A^{\intercal}\bm B)-32\int x^{\intercal}\bm B y\,d\pi_{\bm A}(x,y), \\
        D^2\Phi_{[\bm A]}(\bm B,\bm C)&=64\,\Tr(\bm B^{\intercal}\bm C)+32\varepsilon^{-1}\int x^{\intercal}\bm B y \left(h_0^{\bm A,\bm C}(x)+h_1^{\bm A,\bm C}(y)-32x^{\intercal}\bm Cy\right)d\pi_{\bm A}(x,y),
   \end{align*}
   where  $\bm B,\bm C\in\RR^{d_0\times d_1}$, $\pi_{\bm A}$ is the unique EOT coupling for $\OT_{\bm A,\varepsilon}(\mu_0,\mu_1)$, and $\big(h_0^{\bm A,\bm C},h_1^{\bm A,\bm C}\big)$ is the unique (up to additive constants) pair of functions in $\calC(\supp(\mu_0))\times\calC(\supp(\mu_1))$ satisfying 
    \begin{equation}
    \label{eq:potentialDerivativeSystem}
   \begin{aligned}
      \int\left(h_0^{\bm A,\bm C}(x)+h_1^{\bm A,\bm C}(y)-32x^{\intercal}\bm Cy\right)e^{\frac{\varphi_0^{\bm A}(x)+\varphi_1^{\bm A}(y)-c_{\bm A}(x,y)}{\varepsilon}}d\mu_1(y)=0, \quad \forall\,x\in\supp(\mu_0),\\
    \int\left(h_0^{\bm A,\bm C}(x)+h_1^{\bm A,\bm C}(y)-32x^{\intercal}\bm Cy\right)e^{\frac{\varphi_0^{\bm A}(x)+\varphi_1^{\bm A}(y)-c_{\bm A}(x,y)}{\varepsilon}}d\mu_0(x)=0,  \quad\forall \,y\in\supp(\mu_1).
   \end{aligned}
   \end{equation} 
   Here, $(\varphi_0^{\bm A},\varphi_1^{\bm A})$ is any pair of EOT potentials for $\OT_{\bm A,\varepsilon}(\mu_0,\mu_1)$. %
\end{proposition}

\cref{prop:PhiDerivative} essentially follows from the implicit mapping theorem, which enables us to compute the Fr{\'e}chet derivative of the EOT potentials for $\OT_{(\cdot),\varepsilon}(\mu_0,\mu_1)$ using the Schr{\"o}dinger system \eqref{eq:SchrodingerSystem}. The derivative of $\OT_{(\cdot),\varepsilon}(\mu_0,\mu_1)$, which is a simple function of the EOT potentials, is then readily obtained. By differentiating the Frobenius norm, this further yields the derivative of $\Phi$. See \cref{sec:PhiDerivativeProof} for details. 

\medskip
The following remark clarifies that the first and second Fr{\'e}chet derivatives of $\Phi$ can be identified with the gradient and Hessian of a function on $\mathbb R^{d_0d_1}$. Recall that the Frobenius inner product, $\langle \bm A,\bm B\rangle_F=\Tr\left(\bm A^{\intercal}\bm B\right)=\sum_{\substack{1\leq i\leq d_0\\1\leq j\leq d_1}} \bm A_{ij}\bm B_{ij}$, is simply the Euclidean inner product between 
the vectorized matrices $\bm A,\bm B\in\mathbb R^{d_0\times d_1}$. 

\begin{remark}[Interpreting derivatives as gradient/Hessian]
\label{rmk:GradientPhi}
   The first derivative of $\Phi$ from \cref{prop:PhiDerivative} can be written as 
   \[
D\Phi_{[\bm A]}(\bm B)=
   \left\langle64\bm A-32\int xy^{\intercal}\,d\pi_{\bm A}(x,y),\bm B\right\rangle_F.
   \] 
Recall that if $f$ is a continuously differentiable function $f$ on $\mathbb R^d$, its directional derivative at $x$ along the direction $y$ is $D f_{[x]}(y)=\langle\nabla f(x),y\rangle$, for $x,y\in\mathbb R^d$. 
By analogy, we may think of $D\Phi_{[\bm A]}$ as $64 \bm A-32\int xy^{\intercal}\,d\pi_{\bm A}(x,y)$. This perspective is utilized in Section \ref{sec:Computation} when studying computational guarantees for the EGW distance, as it is simpler to view iterates as matrices rather than abstract linear operators. By the same token, the second derivative of $\Phi$ at $\bm A\in\RR^{d_0\times d_1}$ is a bilinear form on $\mathbb R^{d_0\times d_1}$ and hence can be identified with a $d_0d_1\times d_0d_1$ matrix by analogy with the Hessian. %
    
\end{remark}

As a direct corollary to \cref{prop:PhiDerivative}, we provide an (implicit) characterization of the stationary points of $\Phi$ and connect its minimizers to solutions of $\sS_{\varepsilon}$. Details are provided in \cref{subsec:proofCorStationary}.

\begin{corollary}[Stationary points and correspondence between $\sS_{\varepsilon}$ and $\sS^{2}_{\varepsilon}$]
    \label{cor:StationaryPoint}
    \quad 
    \begin{enumerate}
    \item[(i)] A matrix $\bm A\mspace{-3mu}\in\mspace{-2mu}\RR^{d_0\mspace{-1mu}\times\mspace{-1mu} d_1}$ is a stationary point of $\Phi$ if and only if 
   $
        \bm A=\frac{1}{2} \int xy^{\intercal}\,d\pi_{\bm A}(x,y).
    $
     As $\Phi$ is coercive, all minimizers of $\Phi$ are stationary points and hence contained in $\mathcal D_{M_{\mu_0,\mu_1}}$.  
    \item[(ii)] If $\mu_0$ and $\mu_1$ are centered, then a given matrix $\bm A$ minimizes $\Phi$ if and only if $\pi_{\bm A}$ is optimal  for $\sS_{\varepsilon}$ and satisfies $\frac{1}{2}\int xy^\intercal \,d\pi_{\bm{A}} (x,y) = \bm{A}$.
    \item[(iii)] Suppose  $\mu_0$ and $\mu_1$ are centered. If $\sS_{\varepsilon}$ admits a unique optimal coupling $\pi_{\star}$, then $\Phi$ admits a unique minimizer $\bm{A}^\star$ and $\pi_{\star} = \pi_{\bm{A}^{\star}}$. Conversely, if $\Phi$ admits a unique minimizer $\bm{A}^\star$, then $\pi_{\bm{A}^\star}$ is a unique optimal coupling for $\sS_{\varepsilon}$. 
\end{enumerate}
\end{corollary}

Although the second derivative of $\Phi$ involves the implicitly defined functions $(h_0^{\bm A,\bm C},h_1^{\bm A,\bm C})$, its maximal and minimal eigenvalues admit the following explicit bounds. 

\begin{corollary}[Hessian eigenvalue bounds]\label{cor:HessianMinimalEigenvalue} 
    The following hold for any $\bm A\in\RR^{d_0\times d_1}$:
    \begin{enumerate}
        \item[(i)] The minimal eigenvalue of $D^2\Phi_{[\bm A]}$ satisfies %
    \begin{align*}
        \lambda_{\min}\left(D^2\Phi_{[\bm A]}\right)&=64+\varepsilon^{-1}\inf_{\|\bm C\|_F=1}\int\bigg[ \left(h_0^{\bm A,\bm C}(x)+h_1^{\bm A,\bm C}(y)\right)^2-32^2(x^{\intercal}\bm C y)^2\bigg]d\pi_{\bm A}(x,y)\\
        &\geq 64-32^2\varepsilon^{-1}\sup_{\|\bm C\|_F=1}\Var_{\pi_{\bm A}}(X^{\intercal}\bm C Y), \label{eq:MinimalEigenvalue}
    \end{align*} 
    where the variance term admits the uniform bound 
   $\sup\limits_{\|\bm C\|_F=1}\mspace{-6mu}\Var_{\pi_{\bm A}}(X^{\intercal}\bm C Y)\mspace{-3mu}\leq\mspace{-3mu} \sqrt{M_4(\mu_0)M_4(\mu_1)}$.
   \vspace{1mm}
   \item[(ii)]  The maximal eigenvalue of $D^2\Phi_{[\bm A]}$ satisfies $\lambda_{\max}\left(D^2\Phi_{[\bm A]}\right)\leq 64$.
    \end{enumerate}
\end{corollary}

\cref{cor:HessianMinimalEigenvalue} follows from \cref{prop:PhiDerivative} by considering the variational form of the maximal and minimal eigenvalues; see \cref{subsec:ProofCorMinEig} for details. 
We note that, in general, the variance bound in Item (i) is sharp up to constants in arbitrary dimensions. For example, it is attained up to a factor of 2 by $\mu_0=\frac{1}{2}\left(\delta_{0}+\delta_{a}\right)$ and $\mu_1=\frac{1}{2}\left(\delta_{0}+\delta_{b}\right)$ for $a\in\RR^{d_0}$ and $b\in\RR^{d_1}$; see \cref{app:VarianceBound} for the full variance computation.%

\medskip
Armed with the eigenvalue bounds, we now state the main result of this section addressing convexity and smoothness of $\Phi$. %

\begin{theorem}[Convexity and $L$-smoothness]\label{thm:ConvexitySmoothness} 
        The map $\Phi$ is weakly convex with parameter at most $32^2\varepsilon^{-1}\sqrt{M_4(\mu_0)M_4(\mu_1)}-64$ and, if $\sqrt{M_4(\mu_0)M_4(\mu_1)}<\frac{\varepsilon}{16}$, then it is strictly convex and admits a unique minimizer. Moreover,  for any $M>0$, $\Phi$ is $L$-smooth on $\mathcal D_M$ with
    \[
    L= \max_{\bm A\in\cD_M}\lambda_{\max}\left( D^2\Phi_{[\bm A]}\right)\vee \big(-\lambda_{\min}\left(D^2\Phi_{[\bm A]}\right)\big)\leq 64\vee \left(32^2\varepsilon^{-1}\sqrt{M_4(\mu_0)M_4(\mu_1)}-64\right).
    \] 
\end{theorem}

\edit{ \cref{thm:ConvexitySmoothness} shows that $\Phi$ is amenable to optimization by accelerated gradient methods with step size $L$ and establishes sufficient conditions to guarantee convergence of these algorithms to a global minimizer (i.e.,  convexity of $\Phi$). %
}
In general, optimal EGW couplings may not be unique.
\cref{thm:ConvexitySmoothness} provides sufficient conditions for uniqueness of solutions to both \eqref{eq:S2Decomp} and the EGW problem by the connection discussed in \cref{cor:StationaryPoint} when the marginals are centered.
When the optimal EGW coupling is unique, symmetries in the marginal spaces result in certain structural properties for the optimal $\bm A^\star$ in \eqref{eq:S2Decomp}. The following remark expands on these connections.
\begin{remark}[Symmetries and uniqueness of couplings]
Fix $\varepsilon>0$ and a pair of centered distributions $(\mu_0,\mu_1)\in\mathcal P_{4}(\mathbb R^{d_0})\times\mathcal P_{4}(\mathbb R^{d_1})$. Assume that $\Phi$ admits a unique minimizer $\bm A^{\star}$ and let $\pi_{\bm A^{\star}}$ be the associated EOT/EGW coupling (e.g., under the conditions of \cref{thm:ConvexitySmoothness}, given that $\mu_0,\mu_1$ are compactly supported). If, for $i=0,1$, $\mu_i$ is invariant under the action of the orthogonal transformation $U_i:\mathbb R^{d_i}\to \mathbb R^{d_i}$ in the sense that $(U_i)_{\sharp}\mu_i=\mu_i$ it follows that $(U_0,U_1)_{\sharp}\pi_{\bm A^{\star}}$ is also an optimal EGW coupling, whence $(U_0,U_1)_{\sharp}\pi_{\bm A^{\star}}=\pi_{\bm A^{\star}}$ by uniqueness. Thus, by \cref{cor:StationaryPoint}, $\bm U_0^{\intercal}\bm A^{\star}\bm U_1=\bm A^{\star}$ where $\bm U_i$ is the matrix associated with $U_i$, for $i=0,1$. The previous equality holds for any pair of orthogonal transformations leaving the marginals invariant, so the rows of $\bm A^{\star}$ are left eigenvectors of $\bm U_1$ with eigenvalue $1$ and its columns are right eigenvectors of $\bm U_0^{\intercal}$ with eigenvalue $1$ for every $\bm U_i$ such that $(U_i)_{\sharp}\mu_i=\mu_i$. Thus, we see that symmetries of the marginals dictate the structure of $\bm A^{\star}$. For example, if $\mu_1=(-\Id)_{\sharp}\mu_1$, we have that $\bm A^{\star}=-\bm A^{\star}$, so $\bm A^{\star}=\bm 0$.   
\end{remark}

\section{Computational Guarantees}
\label{sec:Computation}

Building on the stability theory from \cref{sec:asymptoticNormalityCompact}, we now study computation of the EGW problem. %
The goal is to compute the distance between two discrete distributions $\mu_0\in\cP(\RR^{d_0})$ and $\mu_1\in\cP(\RR^{d_1})$ supported on $N_0$ and $N_1$ atoms $(x^{(i)})_{i=1}^{N_0}$ and $(y^{(j)})_{j=1}^{N_1}$, respectively. In light of the decomposition \eqref{eq:EGWDecomp}, we focus on $\sS^{2}_{\varepsilon}$, which is given by a smooth optimization problem whose convexity depends on the value of $\varepsilon$ (cf. 
\cref{subsec:QuadraticDerivative}). %
Throughout, we adopt the perspective of \cref{rmk:GradientPhi} and treat  $D\Phi_{[\bm A]}$, for $\bm A\in\mathbb R^{d_0\times d_1}$, as the matrix $64 \bm A-32\int xy^{\intercal}\,d\pi_{\bm A}(x,y)$.

\subsection{Inexact Oracle Methods}\label{subsec:InexactOracle}
As these problems are already $d_0d_1$-dimensional and computing the second Fr{\'e}chet derivative of $\Phi$ may be infeasible (in particular, it requires solving \eqref{eq:potentialDerivativeSystem}), we focus on first-order methods. Given the regularity of the $\sS^{2}_{\varepsilon}$ optimization problem, standard out-of-the-box numerical routines are likely to yield good results in practice. %
However, to provide meaningful formal guarantees one must account for the fact that evaluation of $\Phi$ and its gradient, for $\bm A\in\cD_M$, %
requires computing the corresponding EOT plan, which often entails an approximation. Indeed, an explicit characterizations of the EOT plan between arbitrary distributions is unknown and algorithms typically rely on a fast numerical proxy of the coupling. We model this under the scope of gradient methods with inexact gradient oracles \citep{Aspremont2008smooth,devolder2014first,dvurechensky2017gradient}.  %

\medskip
For a fixed $\varepsilon>0$ and $\mu_0,\mu_1$ as above, our goal is thus to solve
\[
\min_{\bm A \in \cD_M} 32\|\bm A\|^2_{F}+\OT_{\bm A,\varepsilon}(\mu_0,\mu_1),
\] 
where $M>M_{\mu_0,\mu_1}$, which guarantees that all the optimizers are within the optimization domain (cf. \cref{cor:StationaryPoint}). As we are in the discrete setting, the EOT coupling $\pi^{\bm A}$ for $\OT_{\bm A,\varepsilon}(\mu_0,\mu_1)$, $\bm A\in\cD_M$, is represented by $\bm \Pi^{\bm A}\in\RR^{N_0\times N_1}$, where $ {\bm \Pi}_{ij}^{\bm A}=\pi^{\bm A}(x^{(i)},y^{(j)})$. %
The inexact oracle paradigm assumes that, for any $\bm A\in\cD_M$, we have access to a $\delta$-oracle approximation $\widetilde{\bm \Pi}^{\bm A}$ of ${\bm \Pi}^{\bm A}$ with $\|\widetilde {\bm \Pi}^{\bm A}-{\bm \Pi}^{\bm A}\|_{\infty}<\delta$. Such oracles can be obtained, for instance, by numerical resolution of the EOT problem. To this end, Sinkhorn's algorithm \citep{sinkhorn1967diagonal,cuturi2013sinkhorn} serves as the canonical approach.  

\begin{proposition}[Inexact oracle via Sinkhorn iterations]
\label{rmk:InexactSinkhorn}
Fix $\delta>0$. Then,
Sinkhorn's algorithm (\cref{alg:SinkhornAlg}) returns an $(e^{\delta}-1)$-oracle approximation $\widetilde{\bm \Pi}^{\bm A}$ of $\bm \Pi^{\bm A}$ in at most $\tilde k$ iterations, where $\tilde k$ depends only on $\mu_0,\mu_1,\bm A,\delta,$ and $\varepsilon$, and is given explicitly in  \eqref{eq:maxIters}. 
\end{proposition}
The proof of \cref{rmk:InexactSinkhorn} follows by combining a number of known results. Complete details can be found in \cref{app:SinkhornInexactAlg}. \edit{To our knowledge, the majority of the literature concerning the use of Sinkhorn's algorithm for EOT focuses on approximating unregularized OT and treats the KL divergence term as a bias. Here we quantify the accuracy of estimating the true EOT plan, which may be of an independent interest.}

With these preparations, we first discuss the case where $\Phi$ is known to be convex on $\cD_M$. 

\subsection{Convex Case}

Assume that $\Phi$ is convex on $\cD_M$, e.g., under the setting of \cref{thm:ConvexitySmoothness}. As convexity implies that the minimal eigenvalue of $D^2\Phi_{[\bm A]}$ is positive for any $\bm A\in \cD_M$, \cref{thm:ConvexitySmoothness} further yields that $\Phi$ is $64$-smooth. %
With that, we can the apply inexact oracle~first-order method from \citet{Aspremont2008smooth}. To describe the approach, assume that we are given a $\delta$-oracle $\widetilde{\bm \Pi}^{\bm A}$ for the EOT plan $\bm \Pi^{\bm A}$ for $\OT_{\bm A,\varepsilon }(\mu_0,\mu_1)$, and define the corresponding gradient approximation
\begin{equation}
\label{eq:appGradient}
\widetilde D\Phi_{[\bm A]}=64\bm A-32\sum_{\substack{1\leq i\leq N_0\\1\leq j\leq N_1}} x^{(i)}(y^{(j)})^{\intercal}\widetilde{\bm \Pi}_{ij}^{\bm A}.
\end{equation}
We now present the algorithm and follow it with formal convergence guarantees.%

\begin{algorithm}
\caption{Fast gradient method with inexact oracle}\label{alg:fgmInexact}
\begin{algorithmic}[1]
\Statex Fix $L= 64$ and let $\alpha_k=\frac{k+1}{2}$, and $\tau_k=\frac{2}{k+3}$
\State $k\gets 0$
\State $\bm A_0\gets\bm 0$
\State $\bm G_0\gets \widetilde D\Phi_{[\bm A_0]}$
\State $\bm W_0\gets \alpha_0\bm G_0$
\While{stopping condition is not met}
\State $\bm B_{k}\gets \min\left(1,\frac{M}{2\|\bm A_k-L^{-1}\bm G_k\|_F}\right)(\bm A_k-L^{-1}\bm G_k)$
\State $\bm C_{k}\gets -
\min\left(1,\frac{M}{2\|L^{-1}\bm W_k\|_F}\right)L^{-1}\bm W_k$
\State $\bm A_{k+1}\gets\tau_k\bm C_k+(1-\tau_k) \bm B_k$
\State $\bm G_{k+1}\gets \widetilde D\Phi_{[\bm A_{k+1}]}$
\State $\bm W_{k+1}\gets \bm W_{k}+\alpha_{k+1}\bm G_{k+1}$
\State $k\gets k+1$
\EndWhile
\State \textbf{return} $\bm B_k$
\end{algorithmic}
\end{algorithm}

The multiplication operations in \cref{alg:fgmInexact}  are applied entrywise and it is understood that $\min(1,{M}/{ 0})=1$. Due to inexactness, stopping conditions based on insufficient progress of functions values or setting a threshold on the norm of the gradient require care. A condition based on the number of iterations is discussed in  \cref{rmk:RatesandStopping}.

\medskip
 We now provide formal convergence guarantees for \cref{alg:fgmInexact}.

\begin{theorem}[Fast convergence rates]%
\label{thm:fgmRates}
    Assume that $\Phi$ is convex and $L$-smooth on $\cD_M$ and that $\widetilde{\bm \Pi}^{\bm A}$ is a $\delta$-oracle for $\bm \Pi^{\bm A}$. Then, the iterates $\bm B_k$ in \cref{alg:fgmInexact} with $\widetilde{D}\Phi_{[\bm A_k]}$ given by \eqref{eq:appGradient}
    satisfy 
\begin{equation}
    \label{eq:fgmRate}
    \Phi(\bm B_k)-\Phi(\bm B^{\star})\leq \frac{2L\|\bm B^{\star}\|_F^2}{(k+1)(k+2)}+3\delta',
\end{equation}
where $\bm B^{\star}$ is a global minimizer of $\Phi$ and $\delta'=32M\delta \sum_{\substack{1\leq i\leq N_0\\1\leq j\leq N_1}}\left\| x^{(i)}\right\|\left\|y^{(j)}\right\|$. Moreover, for any
$\eta>3\delta'$, \cref{alg:fgmInexact} requires at most 
\begin{equation}
\label{eq:maxIterations}
k=\left\lceil -\frac{3}{2}+\frac{1}{2}\sqrt{1+\frac{8L\|\bm B^{\star}\|^2_F}{\eta-3\delta'}} \right\rceil  \leq \left\lceil -\frac{3}{2}+\frac{1}{2}\sqrt{1+\frac{128M^2}{\eta-3\delta'}} \right\rceil
\end{equation}
iterations to achieve an $\eta$-approximate solution. 
\end{theorem}

The proof of \cref{thm:fgmRates}, given in \cref{sec:ProoffgmRates}, follows from Theorem 2.2 in \citet{Aspremont2008smooth} after casting our problem as an instance of their setting. Some implications of \cref{thm:fgmRates} are discussed in the following remark.

\begin{remark}[Optimal rates and stopping conditions]
\label{rmk:RatesandStopping}
    First, consider the convergence rate of the function values in \eqref{eq:fgmRate}. The first term 
 on the right-hand side exhibits the optimal complexity bound for smooth constrained optimization of $O(1/k^2)$ (cf., e.g., \citealt{nesterov2003introductory}). The second term accounts for the underlying oracle error. Notably, the progress of the optimization procedure and the oracle error are completely decoupled in this bound.

\medskip
    Next, we describe a stopping condition based on the number of iterations.
    Observe that all terms involved in the upper bound in \eqref{eq:maxIterations} are explicit as soon as a desired precision $\eta$ is chosen since the oracle error $\delta$ can be fixed according to \cref{rmk:InexactSinkhorn}. %
    Consequently, \eqref{eq:maxIterations} can be used as an explicit stopping condition for \cref{alg:fgmInexact}.%
\end{remark}

\subsection{General Case}

We now discuss an optimization procedure which does not require convexity of the objective. This accounts for the fact that outside the sufficient conditions of \cref{thm:ConvexitySmoothness}, convexity of $\Phi$ is generally unknown. However, the same result shows that $\Phi$ is $L$-smooth with $L=64\vee \left(32^2\varepsilon^{-1}\sqrt{M_4(\mu_0)M_4(\mu_1)}-64\right)$ and $\OT_{(\cdot),\varepsilon}$ is $L'$-smooth with $L'=32^2\varepsilon^{-1}\sqrt{M_4(\mu_0)M_4(\mu_1)}$. Hence, we adapt the smooth non-convex optimization routine of \citet{ghadimi2016accelerated} to account for our inexact oracle. Notably, their method adapts to the convexity of $\Phi$ as described in \cref{thm:adaptRates}.   

\medskip
We now present the algorithm and describe its convergence rate.

\begin{algorithm}
\caption{Adaptive gradient method with inexact oracle}\label{alg:fgmNonConvexInexact}
\begin{algorithmic}[1]
\Statex Given $\bm C_0\in\mathcal D_M$, fix the step sequences $\beta_k=\frac{1}{2L}$, $\gamma_k=\frac{k}{4L}$, and $\tau_k=\frac{2}{k+2}$. 
\State $k\gets 1$
\State $\bm A_1\gets \bm C_0$
\State $\bm G_1\gets \widetilde D\Phi_{[\bm A_1]}$
\While{stopping condition is not met}
\State $\bm B_{k}\gets \min\left(1,\frac{M}{2\|\bm A_k-\beta_k\bm G_k\|_F}\right)(\bm A_k-\beta_k\bm G_k)$
\State $\bm C_{k}\gets 
\min\left(1,\frac{M}{2\|\bm C_{k-1}-\gamma_k\bm G_k\|_F}\right)(\bm C_{k-1}-\gamma_k\bm G_k)$
\State $\bm B_{k}\gets \frac M2 \sign( \bm A_k-\beta_k\bm G_k)\min\left(\frac{2}{M}\left|\bm A_k-\beta_k\bm G_k\right|,1\right)$
\State $\bm C_{k}\gets \frac M 2\sign(\bm C_{k-1}-\gamma_k\bm G_k)\min\left(\frac{2}{M}\left|\bm C_{k-1}-\gamma_k\bm G_k\right|,1\right)$
\State $\bm A_{k+1}\gets\tau_k\bm C_k+(1-\tau_k) \bm B_k$
\State $\bm G_{k+1}\gets \widetilde D\Phi_{[\bm A_{k+1}]}$
\State $k\gets k+1$
\EndWhile
\State \textbf{return} $\bm B_k$
\end{algorithmic}
\end{algorithm}

Unlike \cref{alg:fgmInexact}, which can be initialized at any fixed $\bm A_0$, the starting point in \cref{alg:fgmNonConvexInexact} should be chosen according to some selection rule that avoids initializing at a stationary point (e.g., random initialization). Indeed, if $\bm A_1$ is set to a stationary point of $\Phi$, then $D\Phi_{[\bm A_1]}=\bm 0$ and, consequently $\widetilde D\Phi_{[\bm A_1]}\approx\bm 0$ (given that the approximate gradient is reasonably accurate), which may result in premature and undesirable termination. Clearly, this is not a concern for \cref{alg:fgmInexact} since it assumes convexity of $\Phi$, whereby any stationary point is a global optimum. 

\medskip
The following result follows by adapting the proofs of Theorem 2 and Corollary 2 in \citet{ghadimi2016accelerated}. For completeness, we provide a self-contained argument in \cref{app:adaptRates} along with a discussion of how this problem fits in the framework of \citet{ghadimi2016accelerated}.

\begin{theorem}[Adaptive convergence rate]%
\label{thm:adaptRates}
Assume that $\Phi$ is $L$-smooth on $\cD_M$ and that $\widetilde{\bm \Pi}^{\bm A}$ is a $\delta$-oracle for $\bm \Pi^{\bm A}$. Then, the iterates $\bm A_k,\bm B_k$ in \cref{alg:fgmNonConvexInexact} with $\widetilde{D}\Phi_{[\bm A_k]}$ given by \eqref{eq:appGradient}
    satisfy 
    \begin{enumerate}
        \item  If $\Phi$ is non-convex and $\OT_{(\cdot),\varepsilon}(\mu_0,\mu_1)$ is $L'$-smooth, then 
\[
\min_{1\leq i\leq k}\left\|\beta_i^{-1}\mspace{-3mu}\left(\bm B_i\mspace{-2mu}-\mspace{-2mu}\bm A_i\right)\right\|^2_F\leq \mspace{-2mu}\frac{96L^2}{k(k\mspace{-3mu}+\mspace{-3mu}1)(k\mspace{-3mu}+\mspace{-3mu}2)}\|\bm C_0-\bm B^{\star}\|^2_F +\frac{24 LL'}{k}\mspace{-2mu}\left(\mspace{-3mu}\|\bm B^{\star}\mspace{-1mu}\|^2_F\mspace{-2mu}+\mspace{-2mu}\frac{5M^2}{16}\right)+8L\delta',  
\]  
where $\bm B^{\star}$ is a global minimizer of $\Phi$, and $\delta'=32M\delta \sum_{\substack{1\leq i\leq N_0\\1\leq j\leq N_1}}\left\| x^{(i)}\right\|\left\|y^{(j)}\right\|$.
\item If $\Phi$ is convex, then 
\[
\min_{1\leq i\leq k} \left\|\beta_i^{-1}\left(\bm B_i-\bm A_i\right)\right\|^2_F\leq \frac{96L^2}{k(k+1)(k+2)}\|\bm C_0-\bm B^{\star}\|^2_F +8L\delta'.
\]
\end{enumerate}
\end{theorem}

To better motivate this result, we show that when  $\left\|\beta_k^{-1}\left(\bm B_k-\bm A_k\right)\right\|_F$ is small,  $D{\Phi}_{[\bm A_k]}$ is approximately stationary.

\begin{corollary}[Approximate stationarity]
\label{cor:adaptRatesGradient}
Let $\bm A_k,\bm B_k$ be iterates from \cref{alg:fgmNonConvexInexact} and assume that $\bm B_k\in\interior(\mathcal D_M)$. Then,  
\[
    \|D\Phi_{[\bm A_k]}\|_F<32\delta\sum_{\substack{1\leq i\leq N_0\\1\leq j\leq N_1}} {\|x^{(i)}\|\|y^{(j)}\|}+\left\|\beta_{k}^{-1}(\bm B_{k}-\bm A_{k})\right\|_F. 
\]
\end{corollary}
The proof of \cref{cor:adaptRatesGradient} follows from the $\delta$-oracle assumption and the fact that when $\bm B_k$ is an interior point of $\mathcal D_M$, we have $\bm B_k=\bm A_k-\beta_k\bm G_k$. See \cref{sec:proofCorAdaptRatesGradient} for the full details. When $\bm B_{k}$ is not an interior point of $\cD_M$, the interpretation of $\|\beta_k^{-1}(\bm B_k-\bm A_k)\|_F$ is less straightforward. However, as all stationary points of $\Phi$ are contained in $\mathcal D_{M_{\mu_0,\mu_1}}$, it is expected that  \cref{alg:fgmNonConvexInexact} will converge to an interior point when $M>M_{\mu_0,\mu_1}$. By analogy with \cref{rmk:RatesandStopping}, when all iterates are interior points \cref{alg:fgmNonConvexInexact} yields a bound on the total number of iterations required to achieve an approximate stationary point.

The following remark addresses the distinctions between the convex and non-convex settings in \cref{thm:adaptRates}.

\begin{remark}[Adaptivity of \cref{alg:fgmNonConvexInexact}] 
As in \cref{thm:fgmRates}, the convergence rates in \cref{thm:adaptRates} are decoupled into a term related to the progress of the optimization procedure and a term related to the oracle error.

\medskip
In the case where $\Phi$ is non-convex, the dominant term in the optimization error is $O(1/k)$, which coincides with the best known rates for solving general unconstrained nonlinear programs \citep{ghadimi2016accelerated}. On the other hand, 
when $\Phi$ is convex, the rate of convergence improves to $O(1/k^3)$ which essentially matches the best known rates for the norm of the gradient in the unconstrained accelerated gradient method applied to a convex L-smooth function (see Theorem 6 in \citet{shi2021understanding} and Theorem 3.1 in \citet{chen2022gradient}\footnote{More precisely, \citet{chen2022gradient} show that the iterates $(y_i)_{i=1}^k$ generated by the accelerated gradient method applied to a convex $L$-smooth function $f$ are such that $\min_{0\leq i \leq k}\|\nabla f(y_i)\|^2=o(1/k^3)$.}). This adaptivity is beneficial, as $\Phi$ may be convex beyond the conditions derived in \cref{thm:ConvexitySmoothness}.

\end{remark}

An empirical comparison of Algorithms \ref{alg:fgmInexact} and \ref{alg:fgmNonConvexInexact} in the convex setting is included in \cref{subsec:NumericalExperiments}. In particular,  \cref{alg:fgmInexact} is seen to outperform \cref{alg:fgmNonConvexInexact} in terms of average runtime despite having the same per iteration complexity when the inexact gradient is computed using standard Sinkhorn iterations.

\begin{remark}[Computational complexity of Algorithms \ref{alg:fgmInexact} and \ref{alg:fgmNonConvexInexact}]
\label{rmk:CompComplex}
As Sinkhorn's algorithm is known to have a complexity of $O(N_0N_1)$ (cf. e.g. \citealt{scetbon2022linear}), the gradient approximation \eqref{eq:appGradient} can be computed in $O(N_0N_1)$ time. It follows that Algorithms \ref{alg:fgmInexact} and \ref{alg:fgmNonConvexInexact} admit a computational complexity of $O(N_0N_1)$. 
\end{remark}

\subsection{Approximating Unregularized Gromov-Wasserstein Distances}

\edit{The EGW distance can approximate unregularized GW to an arbitrary precision, with  error $\left|\mathsf S_{\varepsilon}(\mu_0,\mu_1)-\mathsf S_{0}(\mu_0,\mu_1)\right|=O\left(\varepsilon\log\left(\frac{1}{\varepsilon}\right)\right)$ as $\varepsilon\downarrow 0$ (see Proposition 1 in \citet{Zhang22Gromov}). It is therefore natural to ask if the proposed algorithms can be used to approximate the unregularized GW distance between finitely supported marginals. Note, however, that for $\varepsilon>0$ sufficiently small, $\Phi$ may fail to be convex (see \cref{thm:ConvexitySmoothness}), whence \cref{alg:fgmNonConvexInexact} may only converge to an approximate stationary point of $\Phi$. As such, it is not guaranteed that itself $\mathsf S_{\varepsilon}(\mu_0,\mu_1)$ (i.e., the global minimum of the entropic problem) can be approximated to within a desired accuracy. Nevertheless, we show that these approximate stationary points can be used to capture a notion of local optimality for $\mathsf S_{0}^2(\mu_0,\mu_1)$, keeping in mind that the quadratic GW distance decomposes as $\mathsf S_{0}(\mu_0,\mu_1)=\mathsf S^{1}(\mu_0,\mu_1)+\mathsf S_{0}^2(\mu_0,\mu_1)$ for centered distributions (see Corollary 1 in \citet{Zhang22Gromov}).

As opposed to EOT, optimal solutions to unregularized OT may fail to be unique, which entails that $\Phi_0\coloneqq32\|\cdot\|_F^2+\OT_{(\cdot),0}(\mu_0,\mu_1)$ may be non-differentiable. Remarkably, the following analogue of \cref{cor:StationaryPoint} on stationary points still holds.

\begin{proposition}[Optimality conditions for $\Phi_0$]
    \label{prop:UnregularizedSolutions}
    Let $(\mu_0,\mu_1)\in\mathcal P(\mathbb R^{d_0})\times \mathcal P(\mathbb R^{d_1})$ be compactly supported. If $\bar{\bm A}$ is a local minimizer of $\Phi_0$, then there exists a solution $\bar\pi$ to $\OT_{\bar{\bm A},0}(\mu_0,\mu_1)$ with $ \bar{\bm A}=\frac12\int xy^{\intercal}d\bar{\pi}(x,y)\in\mathcal D_M$ for any $M>M_{\mu_0,\mu_1}$. If $\bar{\bm A}$ is globally optimal and $\mu_0,\mu_1$ are centered, then $\bar \pi$ solves $\mathsf S_0(\mu_0,\mu_1)$. 
\end{proposition}
The proof of \cref{prop:UnregularizedSolutions}, which will appear in \cref{sec:proofUnregularized}, follows similar lines to the proofs provided in the regularized case with the caveat that $\Phi_0$ is merely locally Lipschitz. To adapt these results, we characterize the Clarke subdifferential 
 of $\Phi_0$ \citep{clarke1975generalized}. 

\medskip
 \cref{prop:UnregularizedSolutions} shows that any minimizer, $\bar{\bm A}_0\in\mathcal D_M$, of $\Phi_0$ satisfies an analogous criterion to the stationary points, $\bar{\bm A}_\varepsilon\in\mathcal D_M$,  of $\Phi_{\varepsilon}\coloneqq32\|\cdot\|^2_F+\OT_{(\cdot),\varepsilon}(\mu_0,\mu_1)$ for $\varepsilon>0$ (namely, that $\bar{\bm A}_{\varepsilon}=\frac12\int xy^{\intercal}d\bar{\pi}_{\varepsilon}(x,y)$ for some $\bar \pi_{\varepsilon}$ solving $\OT_{\bar{\bm A}_{\varepsilon},0}(\mu_0,\mu_1)$ for $\varepsilon\geq 0$). As noted previously, when $\varepsilon$ is small, we can only guarantee that \cref{alg:fgmNonConvexInexact} will converge to an approximate stationary point, ${\bm A}^{\star}_{\varepsilon}\in\mathcal D_M$, satisfying $\|D(\Phi_{\varepsilon})_{[\bm A^{\star}_{\varepsilon}]}\|_F=\|64{\bm A}^{\star}_{\varepsilon}-32\int xy^{\intercal}d{\pi}^{\star}_{\varepsilon}(x,y)\|_F\leq \delta$, for some $\delta>0$, where ${\pi}^{\star}_{\varepsilon}$ solves $\OT_{{\bm A}^{\star}_{\varepsilon},\varepsilon}(\mu_0,\mu_1)$.  We now show that the limit points of $\bm A^{\star}_{\varepsilon}$, as $\varepsilon\downarrow 0$, are approximately stationary for $\Phi_0$. Furthermore, if the matrices along the sequence are globally/locally optimal (in an appropriate sense), we show that global/local optimality is inherited at the limit. %

\begin{theorem}[Convergence of approximate stationary points]
\label{thm:approxConvergence}
   Let $(\mu_0,\mu_1)\in\mathcal P(\mathbb R^{d_0})\times \mathcal P(\mathbb R^{d_1})$ be compactly supported and fix $\delta\geq 0$. For $\varepsilon>0$, let $\bm A_{\varepsilon}^{\star}\in\mathcal D_M$ be such that $\|D(\Phi_{\varepsilon})_{[\bm A_{\varepsilon}^{\star}]}\|_F\leq\delta$ and let $\bm A^{\star}_0$ be a cluster point of $(\bm A^{\star}_{\varepsilon})_{\varepsilon>0}$ (as $\varepsilon\downarrow 0$). Then, 
  \begin{enumerate}
      \item  $\|64\bm A^{\star}_{0}-32\int xy^{\intercal}d\pi_{\bm A^{\star}_{0}}\|_F\leq \delta$, where $\pi_{\bm A^{\star}_{0}}$ is a solution of $\OT_{\bm A^{\star}_{0},0}(\mu_0,\mu_1)$.
      \item if $\bm A_{\varepsilon}^{\star}$ is a  global minimizer of $\Phi_{\varepsilon}$ for all $\varepsilon>0$ sufficiently small, then $\bm A^{\star}_0$ minimizes $\Phi_0$.
      \item if, up to a subsequence $\varepsilon_n\downarrow 0$ along which $\bm A^{\star}_{\varepsilon_n}\to \bm A^{\star}_0$,  $\bm A_{\varepsilon_n}^{\star}$ minimizes $\Phi_{\varepsilon_n}$ on a ball of fixed radius $r>0$ centred at $\bm A^{\star}_0$, then $\bm A^{\star}_0$ is a local minimizer of $\Phi_0$.
  \end{enumerate} 
   
\end{theorem}

Recall that $\mathcal D_M$ is compact such that a cluster point $\bm A^{\star}_0$ always exists in \cref{thm:approxConvergence}.
In light of \cref{prop:UnregularizedSolutions}, these limit points can be thought of as approximate stationary points for $\Phi_0$. This enables approximating local solutions to the unregularized variational GW problem by stationary points obtained using \cref{alg:fgmNonConvexInexact} for small $\varepsilon$. %
The proof of \cref{thm:approxConvergence} uses the notion of $\Gamma$-convergence (see e.g. \citealt{braides2014local}) to show that the solutions of $\OT_{\bm A^{\star}_{\varepsilon},\varepsilon}(\mu_0,\mu_1)$ converge to a solution of $\OT_{\bm A^{\star}_{0},0}(\mu_0,\mu_1)$ up to a subsequence and that the local/global minimizers of $\Phi_{{\varepsilon}}$ admit local/global minimizers of $\Phi_0$ as cluster points, see \cref{sec:thmapproxConvergence} for details. %
}

\subsection{Numerical Experiments}
\label{subsec:NumericalExperiments}
We conclude this section with some experiments that empirically validate the rates obtained in Theorems \ref{thm:fgmRates} and \ref{thm:adaptRates} and the computational complexity discussed in \cref{rmk:CompComplex}. All experiments were performed on a desktop computer with 16 GB of RAM and an Intel i5-10600k CPU using the Python programming language. The considered marginal distributions described below, $\mu_0,\mu_1$ were randomly generated. 

\medskip

\noindent\textbf{Convergence rates.}
Figure~\ref{fig:CombinedFigure} (a) presents an example of applying \cref{alg:fgmInexact} to a convex $\Phi$, where the marginals are $\mu_0=0.4\delta_{-1.4}+0.6\delta_{1.2}$ and  $\mu_1=0.4\delta_{-1.01}+0.6\delta_{1.31}$, with  $\varepsilon$ chosen large enough to guarantee convexity. The theoretical rate of $O(k^{-2})$ from \cref{thm:fgmRates} on the optimality gap $\Phi(\bm B_k)-\Phi(\bm B^{\star})$ is seen to hold.\footnote{The plot shows the approximate  gap $\Phi(\bm B_k)-\Phi(\bar{\bm B}^{\star})$, where $\bar{\bm B}^{\star}$ is the iterate attaining the minimal objective value.} 
Figure~\ref{fig:CombinedFigure} (b)  illustrates the progress of \cref{alg:fgmNonConvexInexact} 
 applied to a non-convex $\Phi$, for $\mu_0=\frac{1}{3}\left(\delta_{0.3}+\delta_{-0.8}+\delta_{-0.5}\right)$ and $\mu_1=\frac{1}{3}\left(\delta_{(0.1,0.6)}+\delta_{(-0.5,0.3)}+\delta_{(0.4,-0.3)}\right)$, with $\varepsilon=0.07$ which makes $\Phi$ non-convex. The $O(k^{-1})$ rate for $\min_{1\leq i\leq k}\left\|\beta_i^{-1}\mspace{-3mu}\left(\bm B_i\mspace{-2mu}-\mspace{-2mu}\bm A_i\right)\right\|^2_F$ in the non-convex case from \cref{thm:adaptRates} is well reflected in this example.
Figure~\ref{fig:CombinedFigure} (c) shows that \cref{alg:fgmNonConvexInexact} can match the theoretical rate of $O(k^{-3})$ %
in the convex regime  when initialized in a region of local convexity. In~this example, the generated marginals are $\mu_0=\frac{1}{5}\left(\delta_{-0.1}+\delta_{-0.2}+\delta_{0.2}+\delta_{-0.3}+\delta_{0.3}\right)$ and $\mu_1=\frac{1}{5}\left(\delta_{0.2}+\delta_{-0.3}+\delta_{0.3}+\delta_{-0.4}+\delta_{0.4}\right)$ and $\varepsilon=0.03$. %
 The stopping condition used in all these example is $\|\bm G_k\|_F<5\times 10^{-8}$ and the approximate gradient \eqref{eq:appGradient} is computed using the standard implementation of Sinkhorn's algorithm from the Python Optimal Transport package \citep{flamary2021pot}. 

\begin{figure}[!t]
    \centering
    \includegraphics{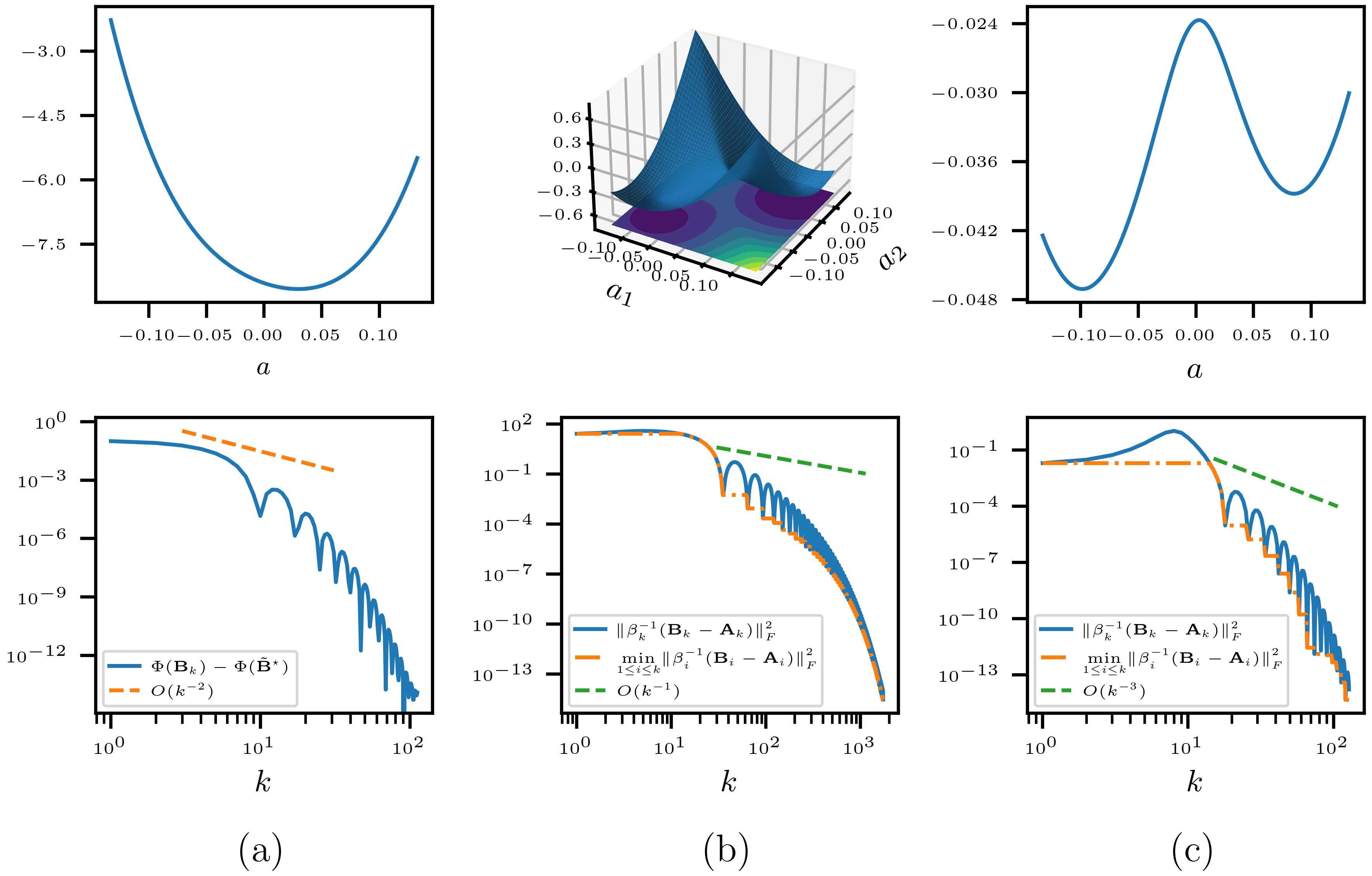}
    \caption{The top row compiles plots of $\Phi$ for the different examples described in the text. The bottom row consists of plots tracking the progress of the iterates. In (b) and (c), \cref{alg:fgmNonConvexInexact} is initialized at $\bm C_0=(1,1)\times 10^{-5}$ and $\bm C_0=1\times 10^{-5}$, respectively. %
    }
    \label{fig:CombinedFigure}
\end{figure}

\medskip

\noindent\textbf{Time complexity.} To study the time complexity of Algorithms \ref{alg:fgmInexact} and \ref{alg:fgmNonConvexInexact}, we first choose the dimension $d\in\{1,16,64,128\}$ and let $\mu_0,\mu_1\in\mathcal P(\mathbb R^{d})$ be supported on $N\in\{16,32,64,128,256,$  $512,1024,2048,$ $4096,8192,16384\}$ samples of a mean-zero normal distribution with standard deviation $0.05$ for $\mu_0$ and $0.1$ for $\mu_1$. The weights are chosen uniformly at random from $[0,1)$ and normalized so as to sum to $1$. This procedure is repeated to generate a collection of pairs of random distributions $\{(\mu_{0,i},\mu_{1,i})\}_{i=1}^{50}$. %
In the sequel, a \emph{single experiment} refers to the process of timing the computation of $\mathsf S_{\varepsilon}(\mu_{0,i},\mu_{1,i})$ for some fixed $d,N$ and all $i=1,\ldots,50$.  For practical reasons, we choose to abort an experiment before all $50$ EGW distances have been computed if %
the total runtime for this experiment exceeds $1$ hour. %
The average runtime is then computed among all completed calculations in a single experiment.

\begin{figure}[t!]
    \centering  
    \includegraphics[width=0.99\textwidth]{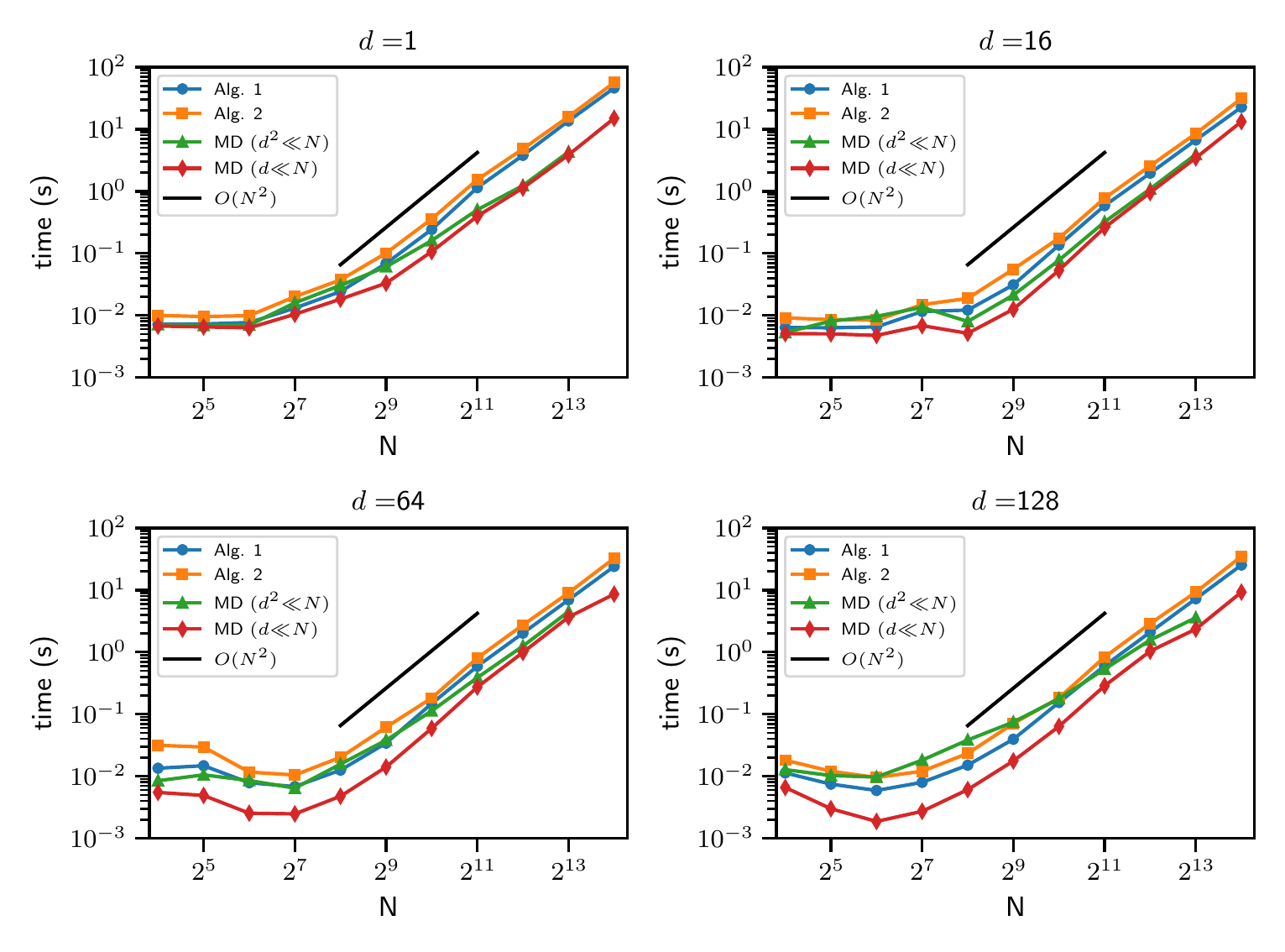}
    \caption{The various plots compile the average runtime of Algorithms \ref{alg:fgmInexact} and \ref{alg:fgmNonConvexInexact}, and two versions of the mirror descent algorithm in the convex regime for different combinations of $d$ and $N$.}%
    \label{fig:runtimeCVX}
\end{figure}

\begin{figure}[t!]
    \centering  
    \includegraphics[width=0.99\textwidth]{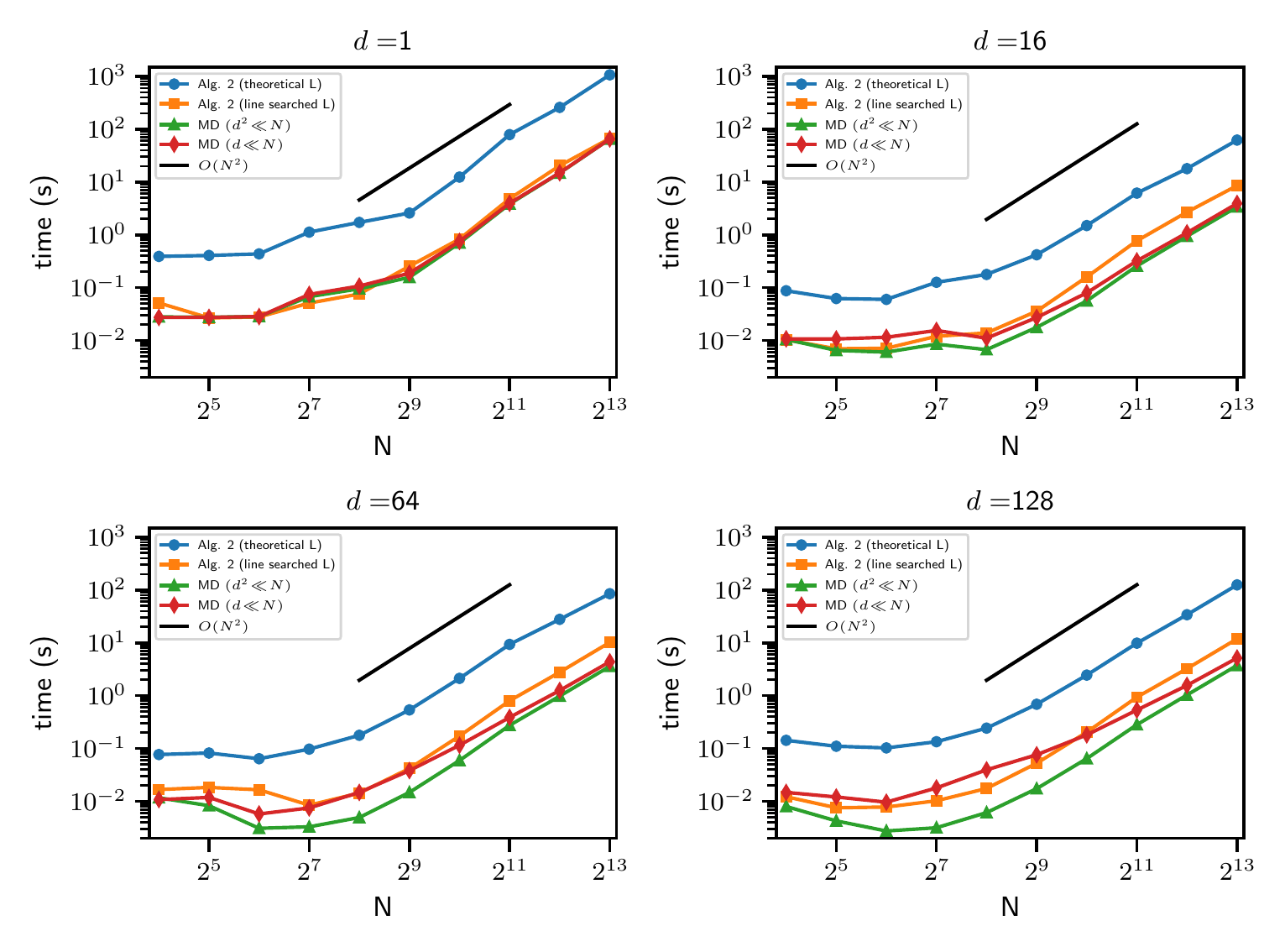}
    \caption{The various plots compile the average runtime of Algorithm  \ref{alg:fgmNonConvexInexact} with the two methods for choosing $L$, and two versions of the mirror descent algorithm in the non-convex regime for different combinations of $d$ and $N$.}%
    \label{fig:runtimeNONCVX}
\end{figure}

\medskip
\underline{The convex case:} First, $\varepsilon$ is chosen as $1.05\times 16\sqrt{M_4(\mu_0)M_4(\mu_1)}$ so as to guarantee convexity of $\Phi$ for each instance by \cref{thm:ConvexitySmoothness} and $M$ is set to $\sqrt{M_2(\mu_0)M_2(\mu_1)}+ 10^{-5}$. Figure \ref{fig:runtimeCVX} presents the average runtime of both algorithms in this setting with the  %
stopping condition %
$\|\bm G_k\|_F<10^{-6}$. We compare the performance of our methods with the two implementations of the $O(N^2)$ mirror descent algorithm provided in \citet{scetbon2022linear}.\footnote{We do not compare with the original implementation of mirror descent \citet{peyre2016gromov} or the iterative algorithm from \citet{solomon2016entropic} due to their much slower $O(N^3)$ runtime.} The first implementation includes certain algorithmic tweaks when $d^2\ll N$, whereas the second only requires $d\ll N$ to achieve the quadratic complexity. Our implementation of the mirror descent algorithm is based on the code provided in \citet{scetbon2022linear} with some small modifications (e.g., EOT couplings are computed using Sinkhorn's algorithm from the Python Optimal Transport package \citep{flamary2021pot}  and some extraneous logging features are removed) the algorithm is run until the generated couplings differ by less than $10^{-6}$ under the Frobenius norm. We note that the first version of the mirror descent algorithm encounters ``out of memory'' errors for $N=16384$. 

The plots show that the four algorithms perform similarly on the considered examples, and empirically validate the %
$O(N^2)$ computational complexity from \cref{rmk:CompComplex}. %
To verify that the algorithms all converge to solutions with similar objective values, we evaluate the relative error\footnote{Relative error is measured by 
$\max_{i\in\mathcal C(d,N)}\big|\mathsf S^{A1}_{\varepsilon}(\mu_{0,i},\mu_{1,i})- \mathsf S^{A2}_{\varepsilon}(\mu_{0,i},\mu_{1,i})\big|/\big(\mathsf S^{A1}_{\varepsilon}(\mu_{0,i},\mu_{1,i})\wedge\mathsf S^{A2}_{\varepsilon}(\mu_{0,i},\mu_{1,i})\big)$,
where $\mathsf S^{A1}_{\varepsilon}(\mu_{0,i},\mu_{1,i})$ and $ \mathsf S^{A2}_{\varepsilon}(\mu_{0,i},\mu_{1,i})$ denote the objective values achieved by the first and second algorithm of the pair, and $\mathcal C(d,N)$ is the collection of completed runs from a given experiment.} between all pairs of algorithms for each $d,N$. The largest relative error we observe is $3.3\times 10^{-6}$ for $d=1$ and, for the other choices of $d$, is at most $7.9\times 10^{-13}$. We conclude that the values obtained are in good agreement.   

\medskip

\underline{The non-convex case:} To evaluate the performance of Algorithm \ref{alg:fgmNonConvexInexact} when convexity is unknown, we set $\varepsilon$ to violate the condition of \cref{thm:ConvexitySmoothness}, but still be large enough so as to avoid numerical errors. %
If errors in running Algorithm \ref{alg:fgmNonConvexInexact} or the mirror descent methods occur, we double $\varepsilon$ until all algorithms converge without errors. The initial point $\bm C_0$ for \cref{alg:fgmNonConvexInexact} is taken as the matrix of all ones scaled by $\min\{M,1\}\times 10^{-5}$. We consider two ways of choosing the smoothness parameter $L$, which effectively dictates the rate of convergence. The first is to set $L$ equals to the theoretical upper bound from \cref{thm:ConvexitySmoothness}, i.e., $L=64\vee \left(32^2\varepsilon^{-1}\sqrt{M_4(\mu_0)M_4(\mu_1)}-64\right)$. As this choice may be too conservative in practice, we also consider setting $L$ via a line search. Namely, we fix a value for $L$ (e.g., the theoretical upper bound or an arbitrary value) and check if the algorithm converges for a given choice of $d,N,\mu_{0,i},\mu_{1,i}$. If so, we multiply $L$ by $0.99$ and repeat this procedure until the algorithm no longer converges. For each $d$ and $N$, we choose the value of $L$ that attains the fastest convergence, and repeat this procedure for 5 pairs of distributions. For Algorithm \ref{alg:fgmNonConvexInexact} with the choice of $L$ that comes from the theoretical bound and the two versions of mirror descent we follow the same methodology as in the convex case, i.e., averaging over $50$ pairs and stopping an experiment after $1$ hour. The average runtimes of all methods are reported in Figure \ref{fig:runtimeNONCVX}. The restriction to 5 runs in the line search case is only out of convenience and we note that all algorithms yield similar results if we restrict to 5 runs throughout.

The plots again validate the $O(N^2)$ time complexity for all four approaches. However, we see that choosing $L$ in Algorithm \ref{alg:fgmNonConvexInexact} according to the theoretical upper bound may indeed be too conservative, as it results in a $10\times$ or larger slowdown compared to the other methods. By setting $L$ via the line search, on the other hand, Algorithm \ref{alg:fgmNonConvexInexact} and mirror descent exhibit similar performance. This suggests that the longer runtime of Algorithm \ref{alg:fgmNonConvexInexact} with the theoretical $L$ value can be attributed to this being an overly conservative choice as opposed to a fundamental limitation of this method. Optimization routines that update $L$ at each iteration have been proposed in \citet{tseng2008accelerated,becker2011templates,nesterov2013gradient}, but require solving an additional EOT problem at each step for our application. As such, these approaches may reduce the number of iterations required for convergence, at the cost of increasing the per iteration complexity.

\medskip

\edit{
\noindent \textbf{Real-world data.} We next assess the performance of our algorithms on real-world data from the {Fashion-MNIST} dataset \citep{xiao2017fashion}. Since the ground truth EGW value is unknown, we test the performance of the algorithm in capturing the invariance of the EGW distance to isometries. As the EGW distance generally does not nullify between isomorphic mm spaces,\footnote{Indeed, %
$\mathsf{D}_{\mathsf{KL}}(\pi\|\mu_0\otimes \mu_1)\geq 0$ with equality if and only if $\pi=\mu_0\otimes \mu_1$ whereas the integral of the distance distortion cost vanishes if and only if the
coupling is induced by some isometry $T:\mathbb R^{d_0}\to \mathbb R^{d_1}$. By \cref{cor:StationaryPoint}, an optimal coupling $\pi^{\star}$ for $\mathsf S_{\varepsilon}(\mu_0,\mu_1)$ is equivalent to $\mu_0\otimes \mu_1$ (in the sense that $\pi^{\star}\ll\mu_0\otimes \mu_1$ and $\mu_0\otimes\mu_1\ll\pi^{\star}$), so $\pi^{\star}$ cannot be induced by an isometry unless $\mu_0,\mu_1$ are supported on one point.} we consider a centered/debiased version, inspired by debiased EOT (also known as Sinkhorn divergence, \citealp{feydy2018interpolating,genevay2018learning}). %
Define the debiased quadratic EGW distance between $(\mu_0,\mu_1)\in\mathcal P(\mathbb R^{d_0})\times\mathcal P(\mathbb R^{d_1})$ as
\[
\bar{\mathsf S}_{\varepsilon}(\mu_0,\mu_1)\coloneqq\mathsf {S}_{\varepsilon}(\mu_0,\mu_1)-\frac{1}{2}\big( \mathsf{S}_{\varepsilon}(\mu_0,\mu_0)+\mathsf{S}_{\varepsilon}(\mu_1,\mu_1)\big).
\]
The recentering guarantees that $\bar{\mathsf S}_{\varepsilon}$ nullifies whenever $(\mathbb R^{d_0},\|\cdot\|,\mu_0)$ and $(\mathbb R^{d_1},\|\cdot\|,\mu_1)$ are isomorphic, as desired. %

Having that, our  experiment consists of comparing a fixed image from the Fashion MNIST dataset to rotated versions of itself and other images from the dataset, by computing the debiased EGW distance between them. By doing so, we empirically validate that computation of $\bar{\mathsf S}_{\varepsilon}$ using \cref{alg:fgmNonConvexInexact} is invariant to isometric transformations on a real-world dataset. Precisely, we pad the $28\times 28$ pixel images from the dataset with zeros on all sides such that the effective image size is $40\times 40$ pixels (this guarantees that no nonzero pixels are lost upon rotation) and treat these padded images as probability distributions on a $40\times 40$ grid of points in
$\mathbb R^2$ with weights proportional to the pixel intensity value. We then compare these distributions using $\bar {\mathsf S}_{\varepsilon}$ upon removing the points with zero mass from each distribution, the values thus obtained are included in \cref{fig:Fashion_MNIST}.  
}
\begin{figure}[!htb]
    \centering
    \includegraphics[width=0.9\textwidth]{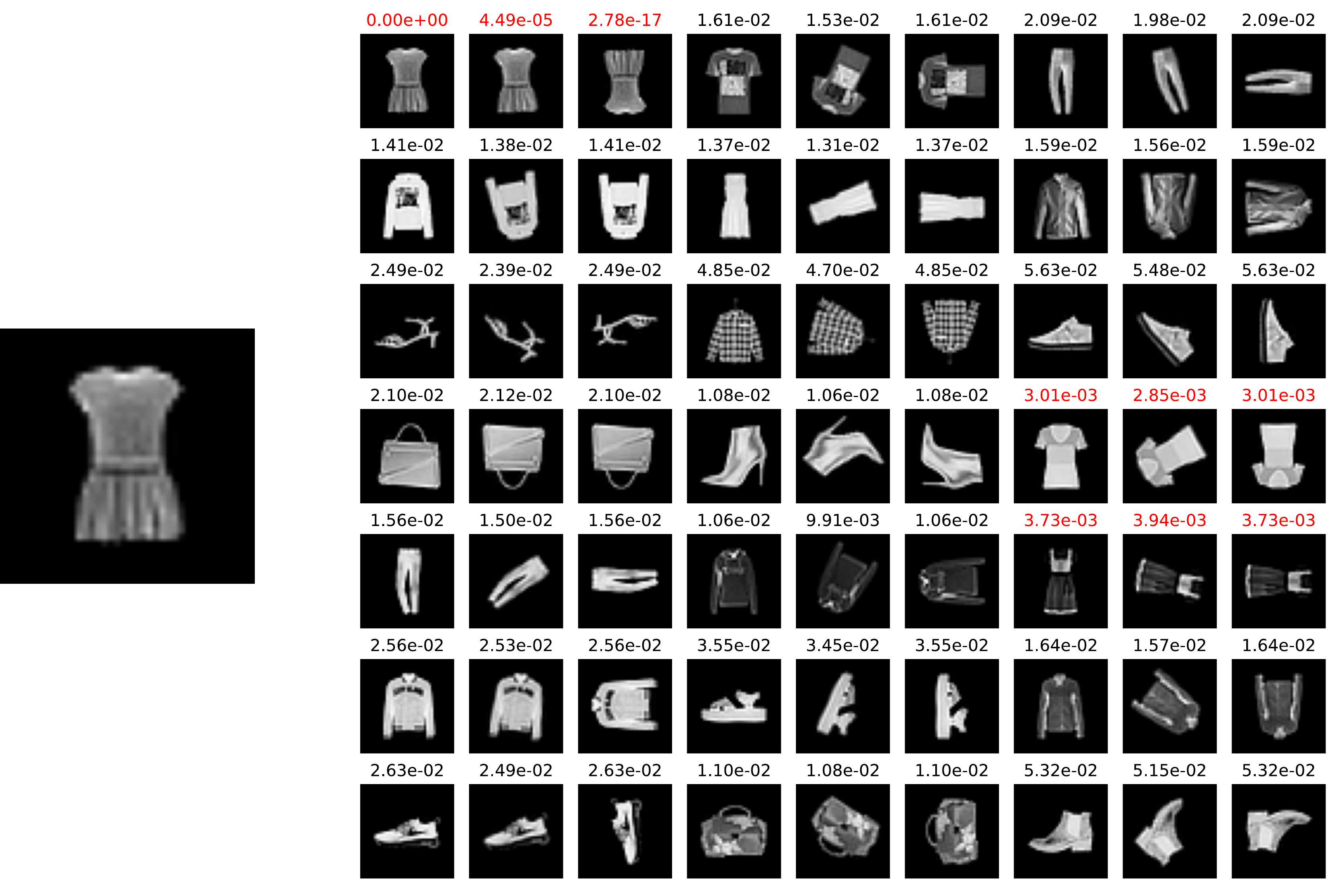}
    \caption{We compare the image on the left to the images on the right using debiased EGW with $\varepsilon=0.1$ as a figure of merit. The corresponding value of $\bar{\mathsf S}_{\varepsilon}$ is included on top of the corresponding images. The images are presented in groups of three, where the leftmost image in the group is the original image, the middle image is obtained via a random rotation of the original image, and the rightmost image is a rotation by a multiple
    of $90^{\circ}$. Distance values smaller than $5\times 10^{-3}$ are in red.}
    \label{fig:Fashion_MNIST}
\end{figure}

\edit{
We note that rotating an image by an angle which is not a multiple of $90^{\circ}$ does not correspond to an isometric action on $\mathbb R^2$, as it requires interpolating the pixel back onto the $40\times 40$ grid of points. Nevertheless, we see from \cref{fig:Fashion_MNIST}, that the images subject to random rotations and the unrotated images achieve similar values of $\bar{\mathsf S}_{\varepsilon}$ relative to the fixed reference image. The discrepancy between these two values can be thought of as a quantification of the distortion to the image structure caused by the interpolation procedure. When the rotation is a multiple of
90$^{\circ}$, we see that the values obtained are identical, as no interpolation is performed on the image.

\cref{fig:Fashion_MNIST}  marks in red values corresponding to images that are closest to the reference image in the debiased EGW distance. These images have a notable structural similarity to the reference in their overall shape and/or features (e.g., pleats on the dress). We also observe that quite naturally the images with largest discrepancy are those of shoes, which have a distinct structure, and the plaid shirt which has a different pattern. Interestingly, between these two extremes, there are images which have comparable values, but are structurally dissimilar (for instance the images in the center column except for the plaid shirt and the sandal). This behaviour demonstrates that the debiased EGW distance does not simply compare the shapes of the images, but rather takes into account the intricate interplay between the intensity values of the images.  

}

\medskip

\section{Proofs}
\label{sec:Proofs}

\subsection{Proof of \texorpdfstring{\cref{prop:PhiDerivative}}{Proposition 1}}
\label{sec:PhiDerivativeProof}

We first fix some notation. Let $S_i=\supp(\mu_i)$ for $i=0,1$ and define the Banach spaces
\begin{gather*}
    \mathfrak E=\left\{(f_0,f_1)\in\calC(S_0)\times \calC(S_1):\int f_0d\mu_0=0\right\},
    \\
    \mathfrak F=\left\{(f_0,f_1)\in\calC(S_0)\times \calC(S_1):\int f_0d\mu_0=\int f_1d\mu_1\right\}.
\end{gather*}
Consider the map $\Upsilon:\RR^{d_0\times d_1}\times \mathfrak{E}\to \calC(S_0)\times\calC(S_1)$ given by
\[
    \Upsilon:(\bm A,\varphi_0,\varphi_1)\mapsto \Bigg(\int e^{\frac{\varphi_0(\cdot)+\varphi_1(y)-c_{\bm A}(\cdot,y)}{\varepsilon}}d\mu_1(y)-1,\int e^{\frac{\varphi_0(x)+\varphi_1(\cdot)-c_{\bm A}(x,\cdot)}{\varepsilon}}d\mu_0(x)-1\Bigg).
\]
For fixed $\bm A\in\RR^{d_0\times d_1}$, the solution to the equation $\Upsilon(\bm A,\cdot,\cdot)=0$ is the unique pair of EOT potentials $(\varphi_0^{\bm A},\varphi_1^{\bm A})$ for $\mu_0,\mu_1$ with the cost $c_{\bm A}$ satisfying the normalization from $\mathfrak E$. Observe that, by compactness of $S_0$ and $S_1$, the potentials are bounded.   

\medskip
The following lemmas verify the conditions to apply the implicit mapping theorem to $\Upsilon$ in order to obtain the Fr{\'e}chet derivative of the map $\bm A\in\mathbb R^{d_0\times d_1}\mapsto (\varphi_0^{\bm A},\varphi_1^{\bm A})$. Given that $\OT_{\bm A,\varepsilon}(\mu_0,\mu_1)=\int \varphi_0^{\bm A}d\mu_0+\int \varphi_1^{\bm A}d\mu_1$, the derivative of the map $\bm A\mapsto\OT_{\bm A,\varepsilon}(\mu_0,\mu_1)$ and that of $\Phi$ itself will readily follow.   %

\begin{lemma}
    \label{lem:UpsilonDerivative}
   The map $\Upsilon$ is smooth with first derivative at $(\bm A,\varphi_0,\varphi_1)\in\RR^{d_0\times d_1}\times \mathfrak E$ given by,
   \begin{align*}
        D\Upsilon_{[\bm A,\varphi_0,\varphi_1]}(\bm B,h_0,h_1)=\varepsilon^{-1}\Bigg( &\int (h_0(\cdot)+h_1(y)+32(\cdot)^{\intercal}\bm B y)e^{\frac{\varphi_0(\cdot)+\varphi_1(y)-c_{\bm A}(\cdot,y)}{\varepsilon}}d\mu_1(y),
        \\&\hspace{1em}\int (h_0(x)+h_1(\cdot)+32x^{\intercal}\bm B (\cdot))e^{\frac{\varphi_0(x)+\varphi_1(\cdot)-c_{\bm A}(x,\cdot)}{\varepsilon}}d\mu_0(x)\Bigg),
   \end{align*} 
   where $(\bm B,h_0,h_1)\in\RR^{d_0\times d_1}\times \mathfrak E$. 
\end{lemma}
The proof of this result is straightforward, but included in \cref{app:UpsilonDerivative} for completeness. 
 Now, define $\xi_{\bm A}\coloneqq\varepsilon D\Upsilon_{[\bm A,\varphi^{\bm A}_0,\varphi^{\bm A}_1]}(0,\cdot,\cdot)$ and let $\pi_{\bm A}$ be the EOT coupling for $\OT_{\bm A,\varepsilon}(\mu_0,\mu_1)$. Note that for any $(h_0,h_1)\in\mathfrak E$, we have $\xi_{\bm A}(h_0,h_1)\in\mathfrak F$, which follows by recalling that $\frac{d\pi_{\bm A}}{d\mu_0\otimes \mu_1}(x,y)=e^{\frac{\varphi^{\bm A}_0(x)+\varphi^{\bm A}_1(y)-c_{\bm A}(x,y)}{\varepsilon}}$ and observing 
\begin{align*}
    \int\big(\xi_{\bm A}(h_0,h_1)\big)_1d\mu_0= \int h_0d\mu_0+\int h_1d\pi_{\bm A}=\int h_0d\mu_0+\int h_1d\mu_1\\
    \int\big(\xi_{\bm A}(h_0,h_1)\big)_2d\mu_1= \int h_0d\pi_{\bm A}+\int h_1d\mu_1=\int h_0d\mu_0+\int h_1d\mu_1.
\end{align*}
We next prove that $\xi_{\bm A}$ is an isomorphism between $\mathfrak E$ and $\mathfrak F$ by following the proof of Proposition 3.1 in \citet{Carlier2020Differential}.

\begin{lemma}
    \label{lem:BanachSpaceIsomorphism}
    The map $\xi_{\bm A}$ is an isomorphism between $\mathfrak E$ and $\mathfrak F$.
\end{lemma}

\begin{proof}
    Observe that $\xi_{\bm A}$ extends naturally to a map on $L^2(\mu_0)\times L^2(\mu_1)$ and admits the representation 
    \[
        \xi_{\bm A}:(f_0,f_1)\in L^2(\mu_0)\times L^2(\mu_1)\mapsto (f_0,f_1)+\mathcal L(f_0,f_1)\in L^2(\mu_0)\times L^2(\mu_1),
    \]
    where 
    \[\mathcal L(f_0,f_1)=\left(\int f_1(y)e^{\frac{\varphi_0^{\bm A}(\cdot)+\varphi_1^{\bm A}(y)-c_{\bm A}(\cdot,y)}{\varepsilon}}d\mu_1(y),\int f_0(x)e^{\frac{\varphi_0^{\bm A}(x)+\varphi_1^{\bm A}(\cdot)-c_{\bm A}(x,\cdot)}{\varepsilon}}d\mu_0(x)\right).
    \]
    \cref{lem:integralHS} in Appendix \ref{sec: compactness} demonstrates that 
    $\mathcal L$ is a compact linear self-map of $L^2(\mu_0)\times L^2(\mu_1)$. 
    
    We first show that $\xi_{\bm A}$ is injective on $E\coloneqq\{(f_0,f_1)\in L^2(\mu_0)\times L^2(\mu_1):\int f_0d\mu_0=0\}$. 
    Suppose that $(\bar f_0,\bar f_1)$ satisfies $\xi_{\bm A}(\bar f_0,\bar f_1)=0$. Multiplying $(\xi_{\bm A}(\bar f_0,\bar f_1))_1$ by $\bar f_0$ and $(\xi_{\bm A}(\bar f_0,\bar f_1))_2$ by $\bar f_1$, we have 
    \begin{gather*}
         \int \left(\bar f_0^2(\cdot)+\bar f_0(\cdot)\bar f_1(y)\right)e^{\frac{\varphi_0^{\bm A}(\cdot)+\varphi_1^{\bm A}(y)-c_{\bm A}(\cdot,y)}{\varepsilon}}d\mu_1(y)=0,\\ 
        \int \left(\bar f_0(x)f_1(\cdot)+\bar f_1^2(\cdot)\right)e^{\frac{\varphi_0^{\bm A}(x)+\varphi_1^{\bm A}(\cdot)-c_{\bm A}(x,\cdot)}{\varepsilon}}d\mu_0(x)=0,
    \end{gather*}
    and summing these equations gives $\int(\bar f_0+\bar f_1)^2d\pi_{\bm A}=0$. As $\pi_{\bm A}$ is equivalent to $\mu_0\otimes \mu_1$, we have $\bar f_0+\bar f_1=0$ $\mu_0\otimes\mu_1$-a.e., which further implies that $(\bar f_0,\bar f_1)=(a,-a)$  $\mu_0\otimes\mu_1$-a.e. for some $a\in\R$. Consequently, $\ker(\xi_{\bm A})$ is $1$-dimensional and $\xi_{\bm A}$ is injective on $E$.

    Next, we show that $\xi_{\bm A}$ is onto $F\coloneqq\{(f_0,f_1)\in L^2(\mu_0)\times L^2(\mu_1):\int f_0d\mu_0=\int f_1d\mu_1\}$.
    As in the lead-up to this lemma, $\xi_{\bm A}(E)\subset F$. By the Fredholm alternative (cf. Theorem 6.6 in \citet{brezis2011functional}), $(\Id+\mathcal L)(L^2(\mu_0)\times L^2(\mu_1))$ has codimension $1$ and, as $F$ has codimension $1$, we must have $\xi_{\bm A}(E)=F$. 
    
    As such, for any $(g_0,g_1)\in \mathfrak F\subset F$, there exists $(h_0,h_1)\in E$ for which 
    \[
        \xi_{\bm A}(h_0,h_1)=(h_0,h_1)+\mathcal L(h_0,h_1)=(g_0,g_1). 
    \]
    As $\mathcal L(h_0,h_1)\in \calC(S_0)\times \calC(S_1)$, $(h_0,h_1)=(g_0,g_1)-\mathcal L(h_0,h_1)\in \calC(S_0)\times \calC(S_1)$ with $\int h_0d\mu_0=0$, and thus $(h_0,h_1)\in \mathfrak E$. This implies that $\xi_{\bm A}(\mathfrak E)\supset \mathfrak F$ and from before we have $\xi_{\bm A}(\mathfrak E)\subset \mathfrak F$, yielding $\xi_{\bm A}(\mathfrak E)=\mathfrak F$. We have shown that $\xi_{A}: \mathfrak E \to \mathfrak F$ is a  continuous linear bijection and hence an isomorphism by the open mapping theorem (cf. Corollary 2.7 in \citealt{brezis2011functional}).
\end{proof}

We now apply the implicit mapping theorem to obtain the Fr{\'e}chet derivative of $(\varphi_0^{(\cdot)},\varphi_1^{(\cdot)})$.

\begin{lemma}\label{lem:PotentialDerivative}
    The map $\bm A\in\RR^{d_0\times d_1}\mapsto(\varphi_0^{\bm A},\varphi_1^{\bm A})\in \mathfrak E$ is smooth with Fr{\'e}chet derivative
    \[
        D\left(\varphi^{(\cdot)}_0,\varphi^{(\cdot)}_1\right)_{[\bm A]}(\bm B)=-\left(h_0^{\bm A,\bm B},h_1^{\bm A,\bm B}\right), 
    \]
    where $\left(h_0^{\bm A,\bm B},h_1^{\bm A,\bm B}\right)\in \mathfrak E$ satisfies
    \begin{equation}
   \label{eq:hSystem}
 \begin{aligned}
      \int\left(h_0^{\bm A,\bm B}(x)+h_1^{\bm A,\bm B}(y)-32x^{\intercal}\bm By\right)e^{\frac{\varphi_0^{\bm A}(x)+\varphi_1^{\bm A}(y)-c_{\bm A}(x,y)}{\varepsilon}}d\mu_1(y)=0, \quad \forall\,x\in\supp(\mu_0),\\
    \int\left(h_0^{\bm A,\bm B}(x)+h_1^{\bm A,\bm B}(y)-32x^{\intercal}\bm By\right)e^{\frac{\varphi_0^{\bm A}(x)+\varphi_1^{\bm A}(y)-c_{\bm A}(x,y)}{\varepsilon}}d\mu_0(x)=0,  \quad\forall \,y\in\supp(\mu_1),
   \end{aligned}
   \end{equation}
   with
   $(\varphi_0^{\bm A},\varphi_1^{\bm A})$ being any pair of EOT potentials for $(\mu_0,\mu_1)$ with the cost $c_{\bm A}$.
\end{lemma}
\begin{proof}
Fix $\bm A\in\RR^{d_0\times d_1}$ with corresponding EOT potentials $\left(\varphi_0^{\bm A},\varphi_1^{\bm A}\right)$. For notational convenience, define the shorthands $D_1\Upsilon_{\bm A}=D\Upsilon_{[\bm A,\varphi^{\bm A}_0,\varphi^{\bm A}_1]}(\cdot,0,0)$ and $D_2\Upsilon_{\bm A}=D\Upsilon_{[\bm A,\varphi^{\bm A}_0,\varphi^{\bm A}_1]}(0,\cdot,\cdot)$ (cf. \cref{lem:UpsilonDerivative}). 
By \cref{lem:BanachSpaceIsomorphism}, $D_2\Upsilon_{\bm A}$ is an isomorphism and we may invoke the implicit mapping theorem (cf. Theorem 5.14 in \citealt{bonnans2013perturbation}). This implies that there exists an open neighborhood $U\subset\RR^{d_0\times d_1}$ of $\bm A$ and a smooth map $g:U\to \mathfrak E$ for which $\Upsilon\big(\bm B,g(\bm B)\big)=0$ for every $\bm B\in U$ and 
\[
    Dg_{[\bm A]}(\bm B)=-(D_2\Upsilon_{\bm A})^{-1}\left(D_1\Upsilon_{\bm A}(\bm B)\right),
\]
i.e., $-Dg_{[\bm A]}(\bm B)$ solves \eqref{eq:hSystem}. By construction, $g(\bm B)=(\varphi_0^{\bm B},\varphi_1^{\bm B})$ and by repeating this~process at any $\bm A\in\RR^{d_0\times d_1}$, differentiability of the potentials is extended to the entire space~$\RR^{d_0\times d_1}$. 
\end{proof}

Given the dual form of the EOT cost, \cref{lem:PotentialDerivative} suffices to prove \cref{prop:PhiDerivative}.

\begin{proof}[Proof of \cref{prop:PhiDerivative}]
   As $\OT_{\bm A,\varepsilon}(\mu_0,\mu_1)=\int \varphi_0^{\bm A}d\mu_0+\int \varphi_1^{\bm A}d\mu_1$, \cref{lem:PotentialDerivative} implies that $\OT_{(\cdot),\varepsilon}(\mu_0,\mu_1)$ is smooth with first derivative at $\bm A\in\mathbb{R}^{d_0\times d_1}$ given by 
   \[
   D\big(\OT_{(\cdot),\varepsilon}(\mu_0,\mu_1)\big)_{[\bm A]}(\bm B)=-\int h_0^{\bm A,\bm B}d\mu_0-\int h_1^{\bm A,\bm B}d\mu_1,  
   \] 
   where $\bm B\in\mathbb R^{d_0\times d_1}$.
Integrating the first equation in \eqref{eq:hSystem} w.r.t. $\mu_0$ while using $\frac{d\pi_{\bm A}}{\mu_0\otimes\mu_1}(x,y)=e^{\frac{\varphi_0^{\bm A}(x)+\varphi_1^{\bm A}(y)-c_{\bm A}(x,y)}{\varepsilon}}$, yields
\begin{equation}
\label{eq:h0h1mean}
\int \left(h_0^{\bm A,\bm B}(x)+h_1^{\bm A,\bm B}(y)\right)d\pi_{\bm A}(x,y)=\int h_0^{\bm A,\bm B}d\mu_0+\int h_1^{\bm A,\bm B}d\mu_1=32\int x^{\intercal} \bm B y \,d\pi_{\bm A}(x,y) ,
\end{equation} 
whence
\[
D\big(\OT_{(\cdot),\varepsilon}(\mu_0,\mu_1)\big)_{[\bm A]}(\bm B)=-32\int x^{\intercal}\bm By \,d\pi_{\bm A}(x,y).
\] 
As $\|\bm A\|_{F}^2=\Tr(\bm A^\intercal \bm A)$, we have $D\big(32\|\bm \cdot\|^2_F\big)_{[\bm A]}(\bm B)=64\Tr(\bm A^\intercal \bm B)$, which together with the display above yields
\[
D\Phi_{[\bm A]}(\bm B)=64\,\Tr(\bm A^{\intercal}\bm B)-32\int x^{\intercal}\bm B y \,d\pi_{\bm A}(x,y),
\]
as desired.

\medskip

For the second-order derivative,  %
recall from \eqref{eqn:EOTcoupling} that $\frac{d\pi_{\bm A}}{d\mu_0\otimes\mu_1}(x,y)=e^{\frac{\varphi_0^{\bm A}(x)+\varphi_1^{\bm A}(y)-c_{\bm A}(x,y)}{\varepsilon}}$. As in the proof of \cref{lem:UpsilonDerivative}, as the map 
\[
\bm A\in\RR^{d_0\times d_1}\mapsto \left( (x,y)\in S_0\times S_1\mapsto\varphi_0^{\bm A}(x)+ \varphi_1^{\bm A}(y)-c_{\bm A}(x,y)\right)\in\mathcal C(S_0\times S_1)
\]
is differentiable at $\bm A\in\RR^{d_0\times d_1}$ with derivative 
\[
\bm C\in\RR^{d_0\times d_1}\mapsto\left((x,y)\in S_0\times S_1\mapsto -\left(h_0^{\bm A,\bm C}(x)+ h_1^{\bm A,\bm C}(y)-32 x^{\intercal}\bm Cy\right)\right)\in\mathcal C(S_0\times S_1), 
\]
the expansion
\[
\frac{d\pi_{\bm A+\bm C}}{d\mu_0\otimes\mu_1}(x,y)-\frac{d\pi_{\bm A}}{d\mu_0\otimes\mu_1}(x,y)=-\varepsilon^{-1}z_{\bm A,\bm C}(x,y)\frac{d\pi_{\bm A}}{d\mu_0\otimes\mu_1}(x,y)+R_{{\bm C}}(x,y),
\]
holds uniformly over $(x,y)\in S_0\times S_1$, where $R_{{\bm C}}(x,y)=o(\bm C)$ as $\|\bm C\|_{F}\to 0$ and $z_{\bm A,\bm C}(x,y)=h_0^{\bm A,\bm C}(x)+ h_1^{\bm A,\bm C}(y)-32 x^{\intercal}\bm Cy$. Thus,  
\begin{align*} 
&\sup_{\|\bm B\|_{F}=1}\frac{\left|
\int x^{\intercal}\bm B y \,d\pi_{\bm A+\bm C}(x,y)-\int x^{\intercal}\bm B y \,d\pi_{\bm A}(x,y)-\varepsilon^{-1}\int x^{\intercal}\bm B yz_{\bm A,\bm C}(x,y)d\pi_{\bm A}(x,y)\right|}{\|\bm C\|_{F}}
\\
=&\sup_{\|\bm B\|_{F}=1}\left|\int x^{\intercal}\bm B y \|\bm C\|_{F}^{-1}R_{\bm C}(x,y)d\mu_0\otimes \mu_1(x,y) \right|
\\
&\leq \sup_{(x,y)\in S_1\times S_2}\|x\|\|y\|\int \|\bm C\|_{F}^{-1}\left|R_{\bm C}(x,y)\right|d\mu_0\otimes \mu_1(x,y).
\end{align*}
As $R_{{\bm C}}(x,y)=o(\bm C)$, this final term converges to $0$ as $\|\bm C\|_{F}\to 0$,
so 
\[
D^2\big(\OT_{(\cdot),\varepsilon}(\mu_0,\mu_1)\big)_{[\bm A]}(\bm B,\bm C)=32\varepsilon^{-1}\int x^{\intercal}\bm By\left(h_0^{\bm A,\bm C}(x)+h_1^{\bm A,\bm C}(y)-32x^{\intercal}\bm Cy\right)d\pi_{\bm A}(x,y).
\]

As $D\big(32\| \cdot\|^2_F\big)_{[\bm A]}(\bm B)=64\Tr(\bm A^\intercal \bm B)$, $D^2\big(32\| \cdot\|^2_F\big)_{[\bm A]}(\bm B,\bm C)=64\Tr(\bm C^\intercal \bm B)$. Altogether, 
\[
     D^2\Phi_{[\bm A]}(\bm B,\bm C)=64\,\Tr(\bm B^{\intercal}\bm C)+32\varepsilon^{-1}\int x^{\intercal}\bm B y \left(h_0^{\bm A,\bm C}(x)+h_1^{\bm A,\bm C}(y)-32x^{\intercal}\bm Cy\right)d\pi_{\bm A}(x,y).
\]
Coercivity of $\Phi$ is due to nonnegativity of the KL divergence along with the Cauchy-Schwarz inequality,
\begin{align*}
    \OT_{\bm A,\varepsilon}(\mu_0,\mu_1)&=\inf_{\pi\in\Pi(\mu_0,\mu_1)}\left\{\int -4\|x\|^2\|y\|^2-32x^{\intercal}\bm A y \,d\pi(x,y)+\varepsilon\KL(\pi\|\mu_0\otimes\mu_1)\right\},
    \\
    &\geq \inf_{\pi\in\Pi(\mu_0,\mu_1)}\left\{\int -4\|x\|^2\|y\|^2-32\|\bm A\|_F\|x\|\|y\|d\pi(x,y)\right\}
    \\
    &\geq -4\sqrt{M_4(\mu_0)M_4(\mu_1)}-32\|\bm A\|_F\sqrt{M_2(\mu_0)M_2(\mu_1)},
\end{align*}
such that $\Phi(\bm A)=32\|\bm A\|^2_{F}+\OT_{\bm A,\varepsilon}(\mu_0,\mu_1)\to +\infty$ as $\|\bm A\|_F\to \infty$.
\end{proof}

\subsection{Proof of \texorpdfstring{\cref{cor:StationaryPoint}}{Corollary 1}}
\label{subsec:proofCorStationary}
\underline{Item (i)}. 
The expression for the stationary points follows immediately from \cref{prop:PhiDerivative}. To see that all stationary points are elements of $\mathcal D_{M_{\mu_0,\mu_1}}$, observe that if $\bm A$ is a stationary point, then
\[
    \|\bm A\|_F=\frac{1}{2}\left\|\int xy^{\intercal} \,d\pi_{\bm A}(x,y)\right\|\leq\frac{1}{2}\int \|x\|\|y\| d\pi_{\bm A}(x,y) \leq \frac{1}{2}\sqrt{M_{2}(\mu_0)M_2(\mu_1)},  
\]
where the first inequality is due to Jensen's inequality, and the second is due to the Cauchy-Schwarz inequality.

\underline{Item (ii)}. As discussed in \cref{subsubsec:EGWQuadCost}, if $\pi_{\star}$ is optimal for $\sS_{\varepsilon}$ then  $\frac{1}{2}\int xy^{\intercal}\,d\pi_{\star}(x,y)$ minimizes $\Phi$. On the other hand, if $\bm{A}$ minimizes $\Phi$, then we have  $\bm{A}=\frac{1}{2}\int xy^{\intercal}\,d\pi_{\bm{A}}$ and hence 
\begin{align*}
    \sS^{2}_{\varepsilon}(\mu_0,\mu_1)&=
    8\left\|\int x y^{\intercal }\,d\pi_{\bm{A}}(x,y)\right\|_F^2
    -4\int \|x\|^2\|y\|^2d\pi_{\bm{A}}(x,y)
    \\
    &
    \quad -32\left\langle\frac{1}{2}\int x y^{\intercal }\,d\pi_{\bm{A}},\int x y^{\intercal }\,d\pi_{\bm{A}}\right\rangle_F+\varepsilon\KL(\pi_{\bm{A}}||\mu_0\otimes \mu_1)
   \\ 
    &=-4\int \|x\|^2\|y\|^2d\pi_{\bm{A}(x,y)}-8\left\|\int xy^{\intercal} \,d\pi_{\bm{A}}(x,y) \right\|^2_F+\varepsilon\KL(\pi_{\bm{A}}||\mu_0\otimes \mu_1). 
\end{align*}
By \eqref{eq:EGWDecomp}, 
\begin{equation}
\label{eq:interproofcorr}
\begin{aligned}
\mathsf S_{\varepsilon}(\mu_0,\mu_1)&=\mathsf S_{\varepsilon}(\mu_0,\mu_1)+\mathsf S^{2}_{\varepsilon}(\mu_0,\mu_1)
\\
&=\int\left|\|x-x'\|^2-\|y-y'\|^2\right|^2 +2\|x-x'\|^2\|y-y'\|^2 d\pi_{\bm{A}}\otimes\pi_{\bm{A}} (x,y,x',y')
\\&\quad -4\int \|x\|^2\|y\|^2d\mu_0\otimes \mu_1(x,y)-4\int \|x\|^2\|y\|^2d\pi_{\bm{A}}(x,y)
\\
&\quad -8\left\|\int xy^{\intercal} \,d\pi_{\bm{A}}(x,y) \right\|^2_F+\varepsilon\KL(\pi_{\bm{A}}||\mu_0\otimes \mu_1).
\end{aligned}
\end{equation}
As $\|x-x'\|^2\|y-y'\|^2=\left(\|x\|^2-2x^{\intercal} x'+\|x'\|^2\right)\left(\|y\|^2-2y^{\intercal} y'+\|y'\|^2\right)$, we have
\begin{align*}
 &\int\|x-x'\|^2\|y-y'\|^2 d\pi_{\bm{A}}\otimes\pi_{\bm{A}} (x,y,x',y')\\
 &=2\int\|x\|^2\|y\|^2d\mu_0\otimes \mu_1(x,y)+2\int\|x\|^2\|y\|^2d\pi_{\bm{A}}(x,y)   
\\
&\quad +4\int x^{\intercal}x'y^{\intercal}y'\,d\pi_{\bm{A}}\otimes\pi_{\bm{A}} (x,y,x',y'),
\end{align*}
which, together with \eqref{eq:interproofcorr} yields  
\[
\mathsf S_{\varepsilon}(\mu_0,\mu_1)=\int\left|\|x-x'\|^2-\|y-y'\|^2\right|^2 d\pi_{\bm{A}}\otimes\pi_{\bm{A}} (x,y,x',y')+\varepsilon\KL(\pi_{\bm{A}}||\mu_0\otimes \mu_1), 
\]
proving optimality of $\pi_{\bm{A}}$.

\underline{Item (iii)}. 
Suppose $\sS_{\varepsilon}$ admits a unique optimal coupling. If two matrices $\bm{A}$ and $\bm{B}$ minimize $\Phi$, then $\pi_{\bm{A}} = \pi_{\bm{B}}$ by uniqueness, so  $\bm{A} = \frac{1}{2} \int xy^\intercal \,d\pi_{\bm{A}}(x,y) = \frac{1}{2} \int xy^\intercal \,d\pi_{\bm{B}}(x,y) = \bm{B}$.  Conversely, suppose $\Phi$ admits a unique minimizer $\bm{A}^\star$. If $\pi$ is optimal for $\sS_{\varepsilon}$, then $\pi$ solves the EOT problem $\OT_{\bm{A}^\star,\varepsilon}(\mu_0,\mu_1)$,  so $\pi=\pi_{\bm{A}^\star}$. 
\qed

\subsection{Proof of \texorpdfstring{\cref{cor:HessianMinimalEigenvalue}}{Corollary 2}} 
\label{subsec:ProofCorMinEig}

We first prove Item (i). The minimal eigenvalue of $D^2\Phi_{[\bm A]}$ is given in variational form as 
    \begin{align*} 
        &\inf_{\|\bm C\|_F=1} D^2\Phi_{[\bm A]}(\bm C,\bm C)
        \\
        &
        = \inf_{\|\bm C\|_F=1}\left\{64\|\bm C\|_F^2+32\varepsilon^{-1}\int x^{\intercal}\bm C y \left(h_0^{\bm A,\bm C}(x)+h_1^{\bm A,\bm C}(y)-32x^{\intercal}\bm Cy\right)d\pi_{\bm A}(x,y)\right\} 
        \\
        &\ge 64+32\varepsilon^{-1}\inf_{\|\bm C\|_F=1}\left\{\int x^{\intercal}\bm C y \left(h_0^{\bm A,\bm C}(x)+h_1^{\bm A,\bm C}(y)-32x^{\intercal}\bm Cy\right)d\pi_{\bm A}(x,y)\right\},
    \end{align*}
    using the formula for $D^2\Phi_{[\bm A]}$ from \cref{prop:PhiDerivative}.
    Recall that $(h_0^{\bm A,\bm C},h_1^{\bm A,\bm C})$ satisfy
  \begin{align*}
        \int\left(h_0^{\bm A,\bm C}(x)+h_1^{\bm A,\bm C}(y)-32x^{\intercal}\bm Cy\right)e^{\frac{\varphi_0^{\bm A}(x)+\varphi_1^{\bm A}(y)-c_{\bm A}(x,y)}{\varepsilon}}d\mu_1(y)=0, \quad \forall\,x\in\supp(\mu_0),\\
    \int\left(h_0^{\bm A,\bm C}(x)+h_1^{\bm A,\bm C}(y)-32x^{\intercal}\bm Cy\right)e^{\frac{\varphi_0^{\bm A}(x)+\varphi_1^{\bm A}(y)-c_{\bm A}(x,y)}{\varepsilon}}d\mu_0(x)=0,  \quad\forall \,y\in\supp(\mu_1),
   \end{align*}
   such that, multiplying the top equation by $h_0^{\bm{A},\bm C}$ and integrating w.r.t. $\mu_0$ and performing the same operations on the lower equation with $h_1^{\bm A,\bm C}$ and $\mu_1$ respectively, 
   \begin{align*}
      \int\left[\left(h_0^{\bm A,\bm C}(x)\right)^2+h_1^{\bm A,\bm C}(y)h_0^{\bm A,\bm C}(x)-32x^{\intercal}\bm Cyh_0^{\bm A,\bm C}(x)\right]d\pi_{\bm A}(x,y)=0,\\
    \int\left[h_0^{\bm A,\bm C}(x)h_1^{\bm A,\bm C}(y)+\left(h_1^{\bm A,\bm C}(y)\right)^2-32x^{\intercal}\bm Cyh_1^{\bm A,\bm C}(y)\right]d\pi_{\bm A}(x,y)=0.
   \end{align*}
   Summing these equations gives 
   \[
   32\int x^{\intercal}\bm Cy\left(h_0^{\bm A,\bm C}(x)+h_1^{\bm A,\bm C}(y)\right)d\pi_{\bm A}(x,y)=\int \left(h_0^{\bm A,\bm C}(x)+h_1^{\bm A,\bm C}(y)\right)^2d\pi_{\bm A}(x,y),
   \]
   such that 
   \begin{align*}
       &32\int x^{\intercal}\bm C y \left(h_0^{\bm A,\bm C}(x)+h_1^{\bm A,\bm C}(y)-32x^{\intercal}\bm Cy\right)d\pi_{\bm A}(x,y)
       \\
       &= \int \left(h_0^{\bm A,\bm C}(x)+h_1^{\bm A,\bm C}(y)\right)^2d\pi_{\bm A}(x,y)-32^2\int (x^{\intercal}\bm C y)^2d\pi_{\bm A}(x,y),
   \end{align*}
which proves the first part of Item (i). %
 It remains to show that 
    \[
        \int \left(h_0^{\bm A,\bm C}(x)+h_1^{\bm A,\bm C}(y)\right)^2d\pi_{\bm A}(x,y)-32^2\int (x^{\intercal}\bm C y)^2d\pi_{\bm A}(x,y)\geq -32^2\Var_{\pi_{\bm A}}[X^{\intercal}\bm C Y]. 
    \]
   By Jensen's inequality, we have 
   \begin{align*}
        \int \left(h_0^{\bm A,\bm C}(x)+h_1^{\bm A,\bm C}(y)\right)^2d\pi_{\bm A}(x,y)&\geq\left( \int h_0^{\bm A,\bm C}(x)+h_1^{\bm A,\bm C}(y)d\pi_{\bm A}(x,y)\right)^2
        \\
        &= 32^2\left( \int x^{\intercal}\bm C y\,d\pi_{\bm A}(x,y)\right)^2, 
   \end{align*}
    where the equality follows from \eqref{eq:h0h1mean}, proving the desired inequality.    

    To prove the uniform bound on the variance in Item (i), observe that
    \begin{align*}
    \sup_{\|\bm C\|_F=1}\Var_{\pi_{\bm A}}[X^{\intercal}\bm C Y]
    &\leq\sup_{\|\bm C\|_F=1} \EE_{\pi_{\bm A}}[(X^{\intercal}\bm C Y)^2]
    \\
    &\leq  \sup_{\|\bm C\|_F=1}\|\bm C\|_F^2\int \|x\|^2\|y\|^2d\pi_{\bm A}(x,y),
    \\
    & \leq \sqrt{M_4(\mu_0)M_4(\mu_1)} 
   \end{align*}
    where the final two inequalities follow from the Cauchy-Schwarz inequality. 

    We now prove the upper bound on the maximum eigenvalue of $D^2 \Phi_{[\bm A]}$ from Item (ii) again using its variational characterization,  
\[
\lambda_{\max}\left(D^2\Phi_{[\bm A]}\right)=\sup_{\|\bm C\|_{F}=1}D^2\Phi_{[\bm A]}(\bm C,\bm C)=64+\lambda_{\max}\left(D^2\OT_{(\cdot),\varepsilon}(\mu_0,\mu_1)_{[\bm A]}\right).
\]
Observe that $\OT_{\bm A,\varepsilon}(\mu_0,\mu_1)=\inf_{\pi\in\Pi(\mu_0,\mu_1)}g(\bm A,\pi,\mu_0,\mu_1,\varepsilon)$, where $g$ depends on $\bm A$ only through the term $32\Tr(\bm A^{\intercal}\int xy^{\intercal}\,d\pi(x,y))$ which is linear in $\bm A$. It follows from, e.g.,  Proposition 2.1.2 in  \citet{hiriart2004fundamentals} that $\OT_{(\cdot),\varepsilon}(\mu_0,\mu_1)$ is concave. As such, $\lambda_{\max}\left(D^2\OT_{(\cdot),\varepsilon}(\mu_0,\mu_1)_{[\bm A]}\right)\leq 0$, so $\lambda_{\max}\left(D^2\Phi_{[\bm A]}\right)\leq 64$.
\qedsymbol

\subsection{Proof of \texorpdfstring{\cref{thm:ConvexitySmoothness}}{Theorem 1}}
\label{subsec:proofTheorem1}

    We first discuss the convexity properties of $\Phi$. By \cref{cor:HessianMinimalEigenvalue}, $\lambda_{\min}\left(D^2\Phi_{[\bm A]}+\frac{\rho}{2}\|\bm A\|_F^2\right)\geq 64-32^2\varepsilon^{-1}\sqrt{M_4(\mu_0)M_4(\mu_1)} +\rho$
     for any $\bm A\in\mathbb R^{d_0\times d_1}$ and $\rho\geq 0$. When this lower bound is nonnegative, $\Phi$ is $\rho$-weakly convex on $\mathbb R^{d_0\times d_1}$ by definition; recall \cref{footnote:weakConvLSmooth}. It follows that $\Phi$ is always $\rho$-weakly convex for $\rho=32^2\varepsilon^{-1}\sqrt{M_4(\mu_0)M_4(\mu_1)}-64$. Moreover, if $\sqrt{M_4(\mu_0)M_4(\mu_1)}<\frac{\varepsilon}{16}$, then $\lambda_{\min}\left(D^2\Phi_{[\bm A]}\right)>0$ such that $\Phi$ is strictly convex.    

    $L$-smoothness of $\Phi$ follows from the mean value inequality (see e.g. Example 2 p. 356 in \citealt{apostol1974mathematical}) 
    \begin{align*}    
       \|D\Phi_{[\bm A]}-D\Phi_{[\bm B]}\|_{F}&
       \leq\sup_{\bm C\in [\bm A,\bm B]}\sup_{\|\bm E\|_F=1} \left|D^2\Phi_{[\bm C]}\left(\bm A-\bm B,\bm E\right)\right|,
       \\
       &\leq \sup_{\bm C\in [\bm A,\bm B]}\left(\left|\lambda_{\min}\left(D^2\Phi_{[\bm C]}\right)\right|\vee\left|\lambda_{\max}\left(D^2\Phi_{[\bm C]}\right)\right|\right)\|\bm A-\bm B\|_{F},
    \end{align*} 
    for any $\bm A,\bm B\in\RR^{d_0\times d_1}$, where $[\bm A,\bm B]$ denotes the line segment connecting $\bm A$ and $\bm B$. The claimed result then follows by noting that, for any $\bm A,\bm B\in\mathcal D_M$, $[\bm A,\bm B]\subset \mathcal D_M$ by convexity and the supremum over $\mathcal D_M$ is achieved by compactness and the fact that the objective is continuous. Indeed, the maps $\lambda_{\max}(\cdot),\lambda_{\min}(\cdot)$ are continuous on the space of symmetric matrices, and $D^{2}\Phi_{[\cdot]}$ is continuous as $\Phi$ is smooth.    \qed

\subsection{Proof of \texorpdfstring{\cref{thm:fgmRates}}{Theorem 9}}
\label{sec:ProoffgmRates}

In this section, we show that Theorem 2.2 in \citet{Aspremont2008smooth} on the convergence rate of \cref{alg:fgmInexact} is applicable in our setting. We particularize their result to a fixed prox-function $d=\frac{1}{2}\|\cdot\|^2_{F}$ which is smooth and $1$-strongly convex.    

First, we justify the expressions for the iterates $\bm B_k,\bm C_k$ in \cref{alg:fgmInexact}, which are defined in  \citet{Aspremont2008smooth} as the  proximal operators   
\begin{gather*}
        \bm B_k=\argmin_{\bm V\in \cD_M}\left\{\Tr\left(\bm G_k^{\intercal}\bm V\right)+\frac{L}{2}\left\|\bm V-\bm A_k\right\|^2_F\right\},
        \\
        \bm C_k=\argmin_{\bm V\in \cD_M}\left\{\Tr\left(\bm W_k^{\intercal}\bm V\right)+\frac{L}{2}\left\|\bm V\right\|^2_F\right\}. 
    \end{gather*}
Rearranging terms, both problems can be written, equivalently, as 
    \begin{equation}
    \label{eq:genericProj}
        \argmin_{\bm V\in \cD_M}\left\{\left\|\bm V-\bm U\right\|_F^2\right\},
    \end{equation}
    for $\bm U=\bm A_k-L^{-1}\bm G_k$ 
 and $\bm U=-L^{-1}\bm W_k$ for the $\bm B_k$ and $\bm C_k$ iterations respectively. The solution of \eqref{eq:genericProj} is given by $\bm V=\bm U$ if $\bm V=\bm U\in\mathcal D_M$, and $\frac{M}{2}\bm U/\|\bm U\|_F$ otherwise.

Next, we show that our notion of $\delta$-oracle yields a $\delta'$-approximate gradient in the sense of Equation (2.3) in \citet{Aspremont2008smooth}. Precisely, we prove that  
\begin{equation}
\label{eq:deltapOracle}
    \left|\Tr\left( \left(\widetilde D\Phi_{[\bm A]}-D\Phi_{[\bm A]}\right)^{\intercal}\left(\bm B-\bm C\right)\right)\right|\leq \delta',
\end{equation}
for any $\bm A,\bm B,\bm C\in\mathcal D_M$. By the Cauchy-Schwarz inequality, 
\[
\left|\Tr\left( \left(\widetilde D\Phi_{[\bm A]}-D\Phi_{[\bm A]}\right)^{\intercal}\left(\bm B-\bm C\right)\right)\right|\leq M \left\|\widetilde D\Phi_{[\bm A]}-D\Phi_{[\bm A]}\right\|_{F}. 
\]
Recall that   
\[
 \widetilde D\Phi_{[\bm A]}-D\Phi_{[\bm A]}=32\sum_{\substack{1\leq i \leq N_0\\1\leq j\leq N_1}}x^{(i)}\left(y^{(j)}\right)^{\intercal}\left( \widetilde {\bm \Pi}_{ij}^{\bm A}-\bm \Pi_{ij}^{\bm A}\right),   
\]
where $\left\|\widetilde {\bm \Pi}^{\bm A}-\bm \Pi^{\bm A}\right\|_{\infty}<\delta$ uniformly in $\bm A\in\mathcal D_M$ by the $\delta$-oracle assumption such that 
\begin{equation}    \label{eq:AppGradientDiff}
\left\|\widetilde D\Phi_{[\bm A]}-D\Phi_{[\bm A]}\right\|_{F}\leq 32\left\|\widetilde {\bm \Pi}^{\bm A}-\bm \Pi^{\bm A}\right\|_{\infty}\sum_{\substack{1\leq i \leq N_0\\1\leq j\leq N_1}}\left\|x^{(i)}\left(y^{(j)}\right)^{\intercal}\right\|_F<32\delta \sum_{\substack{1\leq i \leq N_0\\1\leq j\leq N_1}}\left\|x^{(i)}\right\|\left\|y^{(j)}\right\|.
\end{equation}
 Combining the displayed equations yields 
\[
\left|\Tr\left( \left(\widetilde D\Phi_{[\bm A]}-D\Phi_{[\bm A]}\right)^{\intercal}\left(\bm B-\bm C\right)\right)\right|\leq 32M\delta\sum_{\substack{1\leq i \leq N_0\\1\leq j\leq N_1}}\left\|x^{(i)}\right\|\left\|y^{(j)}\right\|  
=\delta',
\]
proving \eqref{eq:deltapOracle}.

With these preparations \cref{thm:fgmRates} follows from Theorem 2.2 in \citet{Aspremont2008smooth} and the discussion following its proof, noting that $\sum_{i=0}^k\frac{i+1}{2}=\frac{(k+1)(k+2)}{4}$. 
\qed

\subsection{Proof of \cref{cor:adaptRatesGradient}}
\label{sec:proofCorAdaptRatesGradient}
As $\bm A_k,\bm B_k$ be iterates from \cref{alg:fgmNonConvexInexact} with $\bm B_k\in\interior(\mathcal D_M)$ such that $\bm B_k=\bm A_k-\beta_k \widetilde D \Phi_{[\bm A_k]}$  by definition. By the triangle inequality,
\[
    \|D\Phi_{[\bm A_k]}\|_{F}\leq \|D\Phi_{[\bm A_k]}-\widetilde D\Phi_{[\bm A_k]}\|_{F}+\|\widetilde D\Phi_{[\bm A_k]}\|_F=\|D\Phi_{[\bm A_k]}-\widetilde D\Phi_{[\bm A_k]}\|_{F}+\|\beta_k^{-1}\left(\bm B_k-\bm A_k\right)\|_F.
\] 
\eqref{eq:AppGradientDiff} further yields 
\[
\|D\Phi_{[\bm A_k]}-\widetilde D\Phi_{[\bm A_k]}\|_{F}<32 \delta\sum_{\substack{1\leq i \leq N_0\\1\leq j\leq N_1}}\|x^{(i)}\|\|y^{(j)}\|, 
\]
proving the claim.
\qed

\subsection{Proof of \texorpdfstring{\cref{prop:UnregularizedSolutions}}{Proposition 23} }
\label{sec:proofUnregularized}

\edit{
Let $S_i=\supp(\mu_i)$ for $i=0,1$. Recall that $S_0,S_1$ are compact by assumption. The proof of \cref{prop:UnregularizedSolutions} follows by verifying optimality conditions for minimizing locally Lipschitz functions 
using the Clarke subdifferential \citep{clarke1990optimization}.  To this end, we first verify that the objective $\Phi_0$ is locally Lipschitz and prove several auxiliary results to characterize the Clarke subdifferential $\partial \Phi_0$, recalling that the Clarke subdifferential of a locally Lipschitz function $f:\mathbb R^{d_0\times d_1}\to \mathbb R$ at a point $\bm A\in\mathbb R^{d_0\times d_1}$ is given by (cf. e.g. Theorem 2.5.1 in \citealt{clarke1990optimization}) 
\begin{equation}
\label{eq:subdifferentialCharacterization} 
\partial f(x)=\conv\left(\left\{ \lim_{n\to\infty}Df_{[\bm A_n]}: U \not\ni \bm A_n\to\bm A   
   \right\}\right),
\end{equation}
where $\conv(A)$ denotes the convex hull of the set $A$, $U\subset \mathbb R^{d_0\times d_1}$ denotes a set of full measure on which $f$ is differentiable, the existence of which is guaranteed by Rademacher's theorem, and we tacitly restrict this definition to convergent sequences of derivatives.

\begin{lemma}
    \label{lem:unregularizedLipschitz}
    The function $\bm A\in\mathbb R^{d_0\times d_1}\mapsto \Phi_0(\bm A)$ is locally Lipschitz continuous and coercive. 
\end{lemma}
\begin{proof}
    We start by proving that $\Phi_0$ is locally Lipschitz. Fix a compact set $\bm K\subset \mathbb R^{d_0\times d_1}$ and observe that, for any $\bm A,\bm A'\in \bm K$, 
\[
    \left| \|\bm A\|^2_F-\|\bm A'\|^2_F\right|=\left| \|\bm A\|_F-\|\bm A'\|_F\right|\left(\|\bm A\|_F+\|\bm A'\|_F\right)\leq 2\sup_{K}\|\cdot\|_F\|\bm A-\bm A'\|_F,
\]
due to the reverse triangle inequality. This shows that $\|\cdot\|^2_F$ is locally Lipschitz. As for $\OT_{(\cdot),0}(\mu_0,\mu_1)$, let $\pi_{\bm A},\pi_{\bm A'}$ be solutions of $\OT_{\bm A,0}(\mu_0,\mu_1),\OT_{\bm A',0}(\mu_0,\mu_1)$ respectively, then
\begin{equation}
\label{eq:unregularizedDiff}
    \OT_{\bm A,0}(\mu_0,\mu_1)-\OT_{\bm A',0}(\mu_0,\mu_1)\geq \int c_{\bm A}d\pi_{\bm A}-\int c_{\bm A'}d\pi_{\bm A}=-32\left\langle \bm A-\bm A',\int xy^{\intercal}d\pi_{\bm A}(x,y)\right\rangle_F, 
\end{equation}
and similarly $\OT_{\bm A,0}(\mu_0,\mu_1)-\OT_{\bm A',0}(\mu_0,\mu_1)\leq-32\langle \bm A-\bm A',\int xy^{\intercal}d\pi_{\bm A'}(x,y)\rangle_F$. Thus, 
\[
\left|\OT_{\bm A,0}(\mu_0,\mu_1)-\OT_{\bm A',0}(\mu_0,\mu_1)\right|\leq 16M\|\bm A-\bm A'\|_F,
\]
recalling that, for any $\pi\in\Pi(\mu_0,\mu_1)$, $\int xy^{\intercal}d\pi\in \mathcal D_M$. This proves that $\Phi_0$ is locally Lipschitz. Coercivity follows by adapting the proof of \cref{prop:PhiDerivative}.
\end{proof}

  By \cref{lem:unregularizedLipschitz} and Rademacher's theorem, $\Phi_0$ is differentiable almost everywhere on $\mathbb R^{d_0\times d_1}$. As the squared Frobenius norm is smooth, we study differentiability of $\OT_{\bm A,0}(\mu_0,\mu_1)$. To simplify notation, let $\Pi^{\star}_{\bm A,0}$ denote the set of optimal solutions to $\OT_{\bm A,0}(\mu_0,\mu_1)$. 

  \begin{lemma}
      \label{lem:convexityUnregularized}
      The function $\bm A\in\mathbb R^{d_0\times d_1}\mapsto -\OT_{\bm A,0}(\mu_0,\mu_1)$ is convex; its subdifferential\footnote{The subdifferential of a convex function $f:\mathbb R^{d_0\times d_1}\mapsto \mathbb R$ at $\bm A$ consists of all $\bm\Xi\in \mathbb R^{d_0\times d_1}$ for which $f(\bm A')-f(\bm A)\geq \langle \bm A'-\bm A,\bm\Xi\rangle_F$ for every $\bm A'\in\mathbb R^{d_0\times d_1}$. } at $\bm A$ contains $\left\{32\int xy^{\intercal}d\pi:\pi\in\Pi^{\star}_{\bm A,0}\right\}$. 
  \end{lemma}

    \begin{proof}
        As $\OT_{(\cdot),0}(\mu_0,\mu_1)$ is the infimum of a family of affine functions, it is concave (see Theorem 5.5 in \citealt{rockafellar1997convex}). The second claim follows directly from \eqref{eq:unregularizedDiff}.       
    \end{proof}

   With \cref{lem:convexityUnregularized}, it is easy to classify the points at which $\OT_{(\cdot),0}(\mu_0,\mu_1)$ is differentiable. 

   \begin{lemma}
      \label{lem:derivativeUnregularized} 
      $ -\OT_{(\cdot),0}(\mu_0,\mu_1)$ is differentiable at $\bm A$ with derivative $-D(\OT_{(\cdot),0}(\mu_0,\mu_1))_{[\bm A]}=32\int xy^{\intercal} d\pi(x,y)$ for any $\pi\in\Pi^{\star}_{\bm A,0}$ if and only if all couplings in $\Pi^{\star}_{\bm A,0}$ admit the same cross-correlation matrix. 
   \end{lemma}

   \begin{proof}
        We recall that a convex function is differentiable at $\bm A$ precisely when its subdifferential at $\bm A$ is a singleton, see Theorem 25.1 in \citet{rockafellar1997convex}. If the proposed condition on the cross-correlation matrices fails, \cref{lem:derivativeUnregularized} implies that the subdifferential is not a singleton, so differentiability fails. 

        To prove the other direction, assume that all couplings in $\Pi^{\star}_{\bm A,0}$ admit the same cross-correlation matrix. Fix a sequence $\bm H_n\in\mathbb R^{d_0\times d_1}\backslash\{\bm 0\}$ with $\|\bm H_n\|_F\downarrow 0$.  From \eqref{eq:unregularizedDiff}, we have, for any $\pi_{\bm A+\bm H_n}\in\Pi^{\star}_{\bm A+\bm H_n,0}$ and $\pi_{\bm A}\in\Pi^{\star}_{\bm A,0}$,
\[
\begin{aligned}
\|\bm H_n\|_F^{-1}{\left|\OT_{\bm A+\bm H_n,0}(\mu_0,\mu_1)-\OT_{\bm A,0}(\mu_0,\mu_1)+32\left\langle \bm H_n,\int xy^{\intercal}d\pi_{\bm A}(x,y)\right\rangle_F\right|}&
\\
&\hspace{-15em}\leq 32 \left\|\int xy^{\intercal} d\pi_{\bm A+\bm H_n}(x,y)-\int xy^{\intercal} d\pi_{\bm A}(x,y)\right\|_F, 
\end{aligned}
\]
As  
$c_{\bm A+\bm H_n}$ converges uniformly to $c_{\bm A}$ on $S_0\times S_1$, for any subsequence $n'$ of $n$ there exists a further subsequence $n''$ along which $\pi_{\bm A+\bm H_{n''}}\stackrel{w}{\to} \pi\in\Pi^{\star}_{\bm A,0}$ by Theorem 5.20 in \citet{villani2008optimal}. Since $(x,y)\in S_0\times S_1\mapsto x_iy_j$ is continuous and bounded for any $1\leq i\leq d_0,1\leq j\leq d_1$,
$\left\|\int xy^{\intercal} d\pi_{\bm A+\bm H_{n''}}(x,y)-\int xy^{\intercal} d\pi(x,y)\right\|_F=\left\|\int xy^{\intercal} d\pi_{\bm A+\bm H_{n''}}(x,y)-\int xy^{\intercal} d\pi_{\bm A}(x,y)\right\|_F\to 0$. As this limit is independent of the choice of subsequence by assumption, it holds along the original sequence $\bm H_n$ such that $D\left(\OT_{(\cdot),0}(\mu_0,\mu_1)\right)_{[\bm A]}=-32\int xy^{\intercal} d\pi_{\bm A}(x,y)$.
   \end{proof}

Those auxiliary results lead to finding the Clarke subdifferential $\partial \Phi_0$, as given below. 

    \begin{lemma}
       \label{lem:ClarkeSubdiff}
       The Clarke subdifferential of $\Phi_0$ at $\bm A\in\mathbb R^{d_0\times d_1}$ is given by 
       \[
            \partial \Phi_0(\bm A)=\left\{64\bm A-32\int xy^{\intercal}d\pi:\pi\in\Pi^{\star}_{\bm A,0}\right\}.
       \]
    \end{lemma}
    \begin{proof}
       We first characterize the Clarke subdifferential of $\OT_{(\cdot),0}(\mu_0,\mu_1)$ at $\bm A$.  
By Propositions 2.2.7 and 2.3.1  in \citet{clarke1990optimization} along with \cref{lem:convexityUnregularized},  
\begin{equation}
\label{eq:subdifferentialInclusion}
\partial \left(\OT_{(\cdot),0}(\mu_0,\mu_1)\right)(\bm A)\supset\left\{-32\int xy^{\intercal}d\pi(x,y): \pi\in\Pi^{\star}_{\bm A,0} \right\}.
\end{equation}

As \cref{lem:derivativeUnregularized} establishes the set $U\subset \mathbb R^{d_0\times d_1}$ on which $\Phi_0$ is differentiable, we study the Clarke subdifferential of $\Phi_0$ through the lens of  \eqref{eq:subdifferentialCharacterization}.  
 Observe that if $U\not\ni\bm A_n\to\bm A$ and $\pi_{\bm A_n}\in\Pi^{\star}_{\bm A_n,0}$, then $\pi_{\bm A_n}\stackrel{w}{\to}\pi_{\bm A}\in\Pi^{\star}_{\bm A,0}$ as in the proof of \cref{lem:derivativeUnregularized} (up to a subsequence). Moreover,    \[
\lim_{n\to\infty}D\left(\OT_{(\cdot),0}(\mu_0,\mu_1)\right)_{[\bm A_n]}=\lim_{n\to\infty}-32\int xy^{\intercal}d\pi_{\bm A_n}(x,y)=-32\int xy^{\intercal}d\pi_{\bm A}(x,y)
\]
along this subsequence. Thus, if $\bm A_n'\not\in U$ converges to $\bm A$ and $\lim_{n\to\infty}D\left(\OT_{(\cdot),0}(\mu_0,\mu_1)\right)_{[\bm A_n']}$ exists, then the limit is given by  $-32\int xy^{\intercal}d\pi(x,y)$ for some $\pi\in \Pi^{\star}_{\bm A,0}$. It is easy to see that $\Pi^{\star}_{\bm A,0}$ is convex, so  \eqref{eq:subdifferentialInclusion} and \eqref{eq:subdifferentialCharacterization} together imply that 
\[
    \partial \left(\OT_{(\cdot),0}(\mu_0,\mu_1)\right)(\bm A)= \left\{-32\int xy^{\intercal}d\pi(x,y):\pi \in \Pi_{\bm A,0}^\star \right\}.
\]
Conclude by applying the subdifferential sum rule (Corollary 1 on p. 39 of \citealt{clarke1990optimization}).    
    \end{proof}

\begin{proof}[Proof of \cref{prop:UnregularizedSolutions}]
   By Proposition 2.3.2 in \citet{clarke1990optimization}, if 
   $\Phi_0$ attains a local minimum at $\bm A^{\star}$, then $0\in\partial \Phi_0(\bm A^{\star})$, i.e., there exists $\pi\in\Pi^{\star}_{\bm A^{\star},0}$ for which $\bm A^{\star}=\frac{1}{2}\int xy^{\intercal} d\pi(x,y)$ by \cref{lem:ClarkeSubdiff}. The second result on optimality for the original GW problem follows directly from the proof of \cref{cor:StationaryPoint}.  
\end{proof}
}

\edit{
\subsection{Proof of \texorpdfstring{\cref{thm:approxConvergence}}{Theorem 24}}
\label{sec:thmapproxConvergence}

Throughout, let $\mathcal X_i\subset \mathbb R^{d_i}$ be a closed ball with finite radius $r>0$ for which $\supp(\mu_i)\subset \mathcal X_i$ for $i=0,1$. By the Bolzano-Weierstrass theorem, $\bm A^{\star}_{\varepsilon}$ admits a limit point $\bm A^{\star}_{0}$ as $\varepsilon\downarrow 0$. Let $\varepsilon_k\downarrow 0$ be a sequence along which $\bm A^{\star}_{\varepsilon_k}\to\bm A^{\star}_{0}$ and define 
\[
\begin{aligned}    
   F_k:\pi\in\mathcal P(\mathcal X_0\times \mathcal X_1)&\mapsto \begin{cases}
        \int c_{\bm A^{\star}_{\varepsilon_k}}d\pi+\varepsilon_k\mathsf D_{\mathsf{KL}}(\pi\|\mu_0\otimes \mu_1),&\text{if }\pi\in\Pi(\mu_0,\mu_1),
        \\
        +\infty,&\text{otherwise,}
        \end{cases}
       \\
    F:\pi\in\mathcal P(\mathcal X_0\times \mathcal X_1)&\mapsto \begin{cases}
        \int c_{\bm A^{\star}_0}d\pi,&\text{if }\pi\in\Pi(\mu_0,\mu_1),
        \\
        +\infty,&\text{otherwise.}
        \end{cases}   
        \end{aligned} 
\]
Throughout, we endow $\mathcal P(\mathcal X_0\times \mathcal X_1)$ with the topology induced by the weak convergence of probability measures which is metrized by the $2$-Wasserstein distance, $\mathsf W_2$, for instance; $(\mathcal P(\mathcal X_0\times \mathcal X_1),\mathsf W_2)$ is notably a separable metric space, see Theorem 6.18 in \citet{villani2008optimal}. 

We first show that $F_k$ $\Gamma$-converges to $F$ according to Definition 2.3 in \citet{braides2014local}. To this end, we show that 
  \begin{equation}
   \label{eq:GammaConv1}   
   F(\pi)\leq \liminf_{k\to \infty}F_k(\pi_k),\text{ for every } \pi_k\stackrel{w}{\to}\pi.
\end{equation} 
    and exhibit a sequence $\pi_k\stackrel{w}{\to}\pi$ satisfying 
  \begin{equation}
   \label{eq:GammaConv2}
F(\pi)\geq \limsup_{k\to \infty}F_k(\pi_k).
    \end{equation}
    
To prove \eqref{eq:GammaConv1}, first assume that $\pi\in\Pi(\mu_0,\mu_1)$. Then, if $\pi_k\stackrel{w}{\to} \pi$,
$
F_k(\pi_k)\geq \int c_{\bm A^{\star}_{\varepsilon_k}}d\pi_k, 
$ 
by nonnegativity of the KL divergence. As $c_{\bm A^{\star}_{\varepsilon_k}}$ converges to $c_{\bm A^{\star}_{0}}$ uniformly on $\mathcal X_0\times \mathcal X_1$ and $c_{\bm A^{\star}_{0}}$ is continuous and bounded, 
\begin{equation}
\label{eq:limofCost}
    \left|\int c_{\bm A^{\star}_{\varepsilon_k}}d\pi_k-\int c_{\bm A^{\star}_{0}}d\pi\right|\leq \left|\int c_{\bm A^{\star}_{\varepsilon_k}}-c_{\bm A^{\star}_{0}}d\pi_k\right|+\left|\int c_{\bm A^{\star}_{0}}d\pi-\int c_{\bm A^{\star}_{0}}d\pi_k\right|\to 0.  
\end{equation}
This shows that, if $\pi\in\Pi(\mu_0,\mu_1)$, then
\[
\liminf_{k\to\infty}F_k(\pi_k)\geq \lim_{k\to\infty}\int c_{\bm A^{\star}_{\varepsilon_k}}d\pi_k=F(\pi).
\]

If $\pi\not\in\Pi(\mu_0,\mu_1)$, then $\pi_k\not\in\Pi(\mu_0,\mu_1)$ for all $k$ sufficiently large as $\Pi(\mu_0,\mu_1)$ is compact for the weak convergence on $\mathcal P(\mathcal X_0\times \mathcal X_1)$ (cf. e.g. the proof of Theorem 1.4 in \citealt{santambrogio2010}) so the bound holds vacuously, proving \eqref{eq:GammaConv1}.

As for \eqref{eq:GammaConv2}, if $\pi\not\in\Pi(\mu_0,\mu_1)$, then there is nothing to show. For $\pi\in\Pi(\mu_0,\mu_1)$, we consider a block approximation; cf. \citet{carlier2017convergence,genevay2019sample}. First, for $i=0,1$, partition $\mathbb R^{d_i}$ by hypercubes of length $\ell>0$,
    \[
    \{H_{i,q,\ell}=[k_1\ell,(q_1+1)\ell)\times \dots \times [q_{d_i}\ell,(q_{d_i}+1)\ell):q=(q_1,\dots,q_{d_i})\in\mathbb Z^{d_i}\},
    \]
    and define the block approximation, $\pi_{\ell}$, of $\pi$ by   
    \[
    \begin{gathered}
    \pi_{\ell}=\sum_{(q,q')\in\mathbb Z^{d_0}\times \mathbb Z^{d_1}}\pi_{\ell}\vert_{H_{0,q,\ell}\times H_{1,q',\ell} },
    \\
    \pi_{\ell}\vert_{H_{0,q,\ell}\times H_{1,q',\ell}}=\frac{\pi(H_{0,q,\ell}\times H_{1,q',\ell})}{\mu_0\otimes \mu_1(H_{0,q,\ell}\times H_{1,q',\ell})}(\mu_0\vert_{H_{0,q,\ell}}\otimes \mu_1\vert_{H_{1,q',\ell}}), 
   \end{gathered} 
    \]
    for every $(q,q')\in\mathbb Z^{d_0}\times\mathbb Z^{d_1}$ with the convention $\frac{0}{0}=0$. Here, $\mu_0\vert_{H_{0,q,\ell}}(A)=\mu_0(A\cap H_{0,q,\ell})$ and similarly for the other restrictions.
    Of note is that $\pi_{\ell}\ll\mu_0\otimes \mu_1$ and that $\pi_{\ell}\in \Pi(\mu_0,\mu_1)$ (see the discussion surrounding Definition 1 in \citealt{genevay2019sample}). 

    \begin{lemma} 
        \label{lem:weakConvergenceBlock} 
    The block approximation
        $\pi_{\ell}$ of $\pi$ converges weakly to $\pi$ as $\ell\downarrow 0$. 
    \end{lemma}
    \begin{proof}
        Let $Q=\left\{ (q,q')\in\mathbb Z^{d_0}\times\mathbb Z^{d_1}:\pi(H_{0,q,\ell}\times H_{1,q',\ell})>0  \right\}$ and $\gamma=\sum_{(q,q')\in Q}\pi(H_{0,q,\ell}\times
    H_{1,q',\ell})\gamma_{q,q'}$ where $\gamma_{q,q'}$ is any coupling of the measures $\left(\frac{\pi\vert_{H_{0,q,\ell}\times
    H_{1,q',\ell}}}{\pi({H_{0,q,\ell}\times
    H_{1,q',\ell}})},\frac{\pi_{\ell}\vert_{H_{0,q,\ell}\times
    H_{1,q',\ell}}}{\pi_{\ell}({H_{0,q,\ell}\times
    H_{1,q',\ell}})} \right)$ with support in $\overline{H_{0,q,\ell}\times
    H_{1,q',\ell}}$ (the closure of $H_{0,q,\ell}\times
    H_{1,q',\ell}$). As $\pi_{\ell}({H_{0,q,\ell}\times
    H_{1,q',\ell}})=\pi({H_{0,q,\ell}\times
    H_{1,q',\ell}})$ by construction, it is readily seen that $\gamma\in\Pi(\pi,\pi_{\ell})$. Thus, 
    \[
    \begin{aligned} 
        \mathsf W_2^2(\pi,\pi_{\ell})\leq \int \|x-y\|^2d\gamma(x,y)&= \sum_{(q,q')\in Q}\pi(H_{0,q,\ell}\times
    H_{1,q',\ell})\int\|x-y\|^2d\gamma_{q,q'}(x,y)
    \\&
    \leq (d_0+d_1)\ell^2,
    \end{aligned} 
    \]
    noting that $\diam(H_{0,q,\ell}\times H_{1,q',\ell} )=\sqrt{d_0+d_1}\ell$ is the diameter of a hypercube in $\mathbb R^{d_0+d_1}$ of length $\ell$. Conclude that $\mathsf W_2(\pi,\pi_{\ell})\to 0$ as $\ell \downarrow 0$ such that $\pi_{\ell}\stackrel{w}{\to}\pi$.
    \end{proof}

     Setting $\ell_k=\varepsilon_k$,  
    \cref{lem:weakConvergenceBlock} yields 
    \[
        \lim_{k\to \infty} \int c_{\bm A^{\star}_{\varepsilon_k}} d\pi_{\varepsilon_k}=\int c_{\bm A^{\star}_{0}} d\pi,
    \]
    by a simple adaption of \eqref{eq:limofCost}. It remains to show that $\lim_{k\to\infty}\varepsilon_k\mathsf D_{\mathsf{KL}}(\pi_{\varepsilon_k}\|\mu_0\otimes\mu_1)=0$ such that 
    \[
        \lim_{k\to\infty} F_k(\pi_{\varepsilon_k})\to F(\pi).
    \]

    \begin{lemma}
     The block approximation
        $\pi_{\ell}$ of $\pi$ satisfies $\varepsilon_k\mathsf{D}_{\mathsf{KL}}(\pi_{\varepsilon_k}\|\mu_0\otimes \mu_1)\to 0$ as $k\to \infty$.   
    \end{lemma}
\begin{proof}
As $\pi(H_{0,q,\ell}\times H_{1,q',\ell})\leq 1$, 
    \begin{equation}
    \label{eq:entropyBound1}
        \begin{aligned}
    \mathsf D_{\mathsf{KL}}(\pi^{\ell}\|\mu_0\otimes \mu_1)&= \sum_{(q,q')\in \mathbb Z^{d_0}\times \mathbb Z^{d_1}}\int \log\left(\frac{\pi\left( H_{0,q,\ell}\times H_{1,q',\ell} \right)}{\mu_0\otimes \mu_1(H_{0,q,\ell}\times H_{1,q',\ell})}\right)d\pi_{\ell}\vert_{H_{0,q,\ell}\times H_{1,q',\ell}}
    \\
    &\leq \sum_{(q,q')\in \mathbb Z^{d_0}\times \mathbb Z^{d_1}}\int -\log(\mu_0(H_{0,q,\ell}))-\log(\mu_1(H_{1,q',\ell})) d\pi_{\ell}\vert_{H_{0,q,\ell}\times H_{1,q',\ell}}
    \\
    &=\sum_{(q,q')\in \mathbb Z^{d_0}\times \mathbb Z^{d_1}}\left( -\log(\mu_0(H_{0,q,\ell}))-\log(\mu_1(H_{1,q',\ell}))\right) \pi(H_{0,q,\ell}\times H_{1,q',\ell})
    \\
    &=-\sum_{q\in \mathbb Z^{d_0}}\log(\mu_0(H_{0,q,\ell}))\mu_0(H_{0,q,\ell})-\sum_{q'\in \mathbb Z^{d_1}}\log(\mu_1(H_{1,q',\ell}))\mu_1(H_{1,q',\ell}).
\end{aligned} 
    \end{equation}
    Observe that
    \[
        \label{eq:entropyBound2}
    \begin{aligned}
    &-\sum_{q\in \mathbb Z^{d_0}}\log(\mu_0(H_{0,q,\ell}))\mu_0(H_{0,q,\ell})
    \\
    &=-\int \log\left(\frac{\mu_0(H_{0,q,\ell})}{\ell^{d_0}}\mathbbm 1_{H_{0,q,\ell}}\right)\frac{\mu_0(H_{0,q,\ell})}{\ell^{d_0}}\mathbbm 1_{H_{0,q,\ell}}d\lambda-d_0\log(\ell), 
\end{aligned} 
    \]
    where $\lambda$ denotes the Lebesgue measure.
    The first term on the last line is the differential entropy of a probability distribution supported on $\mathcal X_{0,\ell}=\cup_{q\in I}H_{0,q,\ell}$, where $I=\left\{ q\in\mathbb Z^{d_0}:H_{0,q,\ell}\cap\mathcal X_0\neq \emptyset  \right\}$. This quantity is maximized among all probability distributions supported on $\mathcal X_{0,\ell}$ which are absolutely continuous w.r.t. the Lebesgue measure by the uniform distribution on $\mathcal X_{0,\ell}$ with
    value $d_0\log(\lambda(\mathcal X_{0,\ell}))$. With this and \eqref{eq:entropyBound1}, we obtain
    \[
    \mathsf{D}_{\mathsf{KL}}(\pi^{\ell}\|\mu_0\otimes \mu_1)\leq d_0\log\left(\frac{\lambda(\mathcal X_{0,\ell})}{\ell}  \right) +d_1\log\left(\frac{\lambda(\mathcal X_{1,\ell})}{\ell}  \right).
    \]
    Conclude that 
    \[
        0\leq \varepsilon_k\mathsf{D}_{\mathsf{KL}}(\pi_{\varepsilon_k}\|\mu_0\otimes \mu_1)\leq \varepsilon_k\left(d_0\log\left(\frac{\lambda(\mathcal X_{0,\varepsilon_k})}{\varepsilon_k}  \right) +d_1\log\left(\frac{\lambda(\mathcal X_{1,\varepsilon_k})}{\varepsilon_k}  \right)\right)\to 0,
    \]
    as $\lambda(\mathcal X_{i,\varepsilon_k})\geq \lambda(\mathcal X_{i})>0$ for $i=0,1$.
\end{proof}

With this, we have shown that $F_k$ $\Gamma$-converges to $F$. Now, let $\pi_{\varepsilon_k}$ minimize $F_k$. As $\Pi(\mu_0,\mu_1)$ is compact, $\pi_{\varepsilon_k}$ admits a cluster point $\pi_0$ which minimizes $F$ by Theorem 2.1 in \citet{braides2014local}. Thus, $\pi_0$ is an optimal solution of $\OT_{\bm A_0^{\star}}(\mu_0,\mu_1)$ and, along the subsequence where $\pi_{\varepsilon_k}\stackrel{w}{\to}\pi_0$, 
\[\left\|64\bm A^{\star}_{0}-32\int xy^{\intercal}d\pi_{{0}}\right\|_F=\lim_{k\to\infty}\left\|64\bm A^{\star}_{\varepsilon_k}-32\int xy^{\intercal}d\pi_{\varepsilon_{k}}\right\|_F\leq \delta.
\]
This concludes the proof of the first result.

   For the results pertaining to local/global optimality, let $\varepsilon_k\downarrow 0$ be arbitrary and define
   \[
        \begin{aligned}
            G_k:\bm A\in\mathbb R^{d_0\times d_1}&\mapsto \begin{cases} 
                \Phi_{\varepsilon_k},&\text{if }\bm A\in\mathcal D_M,
                \\
                +\infty,&\text{otherwise},
            \end{cases},\quad
            G:\bm A\in\mathbb R^{d_0\times d_1}&\mapsto \begin{cases} 
                \Phi_{0},&\text{if }\bm A\in\mathcal D_M,
                \\
                +\infty,&\text{otherwise}.
            \end{cases}
        \end{aligned} 
   \] We now show that $G_k$ $\Gamma$-converges to $G$.
   Observe that, for any $\bm A_k\to \bm A\in\mathcal D_M$, 
   \[   
        \Phi_{\varepsilon_k}(\bm A_k)=32\|\bm A_k\|^2_F+\OT_{\bm A_k,\varepsilon_k}(\mu_0,\mu_1)\geq32\|\bm A_k\|^2_F+\OT_{\bm A_k,0}(\mu_0,\mu_1),  
   \]
   where the inequality is due to nonnegativity of the KL divergence. As $c_{\bm A_k}\to c_{\bm A}$ uniformly, $32\|\bm A_k\|^2_F+\OT_{\bm A_k,0}(\mu_0,\mu_1)\to 32\|\bm A\|^2_F+\OT_{\bm A,0}(\mu_0,\mu_1)=\Phi_0(\bm A)$. If $\bm A\not\in\mathcal D_M$, then $\bm A_k\not\in\mathcal D_M$ for every $k$ sufficiently large,   so the $\liminf$ condition \eqref{eq:GammaConv1} holds.

   For the $\limsup$ condition \eqref{eq:GammaConv2}, if $\bm A\not\in\mathcal D_M$, then  the bound is vacuous. If $\bm A\in\mathcal D_M$, then let $\pi_{\bm A}$ be an optimal solution of $\OT_{\bm A,0}(\mu_0,\mu_1)$ and $\pi_{\varepsilon_k}$ be the block approximation of $\pi_{\bm A}$ with $\ell=\varepsilon_k$. Then, 
   \[
   0\leq \OT_{\bm A,\varepsilon_k}(\mu_0,\mu_1)-\OT_{\bm A,0}(\mu_0,\mu_1)\leq \int c_{\bm A}d\pi_{\varepsilon_k}+\varepsilon_k\mathsf D_{\mathsf{KL}}(\pi_{\varepsilon_k}\|\mu_0\otimes\mu_1)-\int c_{\bm A}d\pi_{\bm A}. 
   \]
   We have shown previously that the rightmost term converges to $0$ as $k\to\infty$, so 
   $
        \Phi_{\varepsilon_k}(\bm A) \to \Phi_{0}(\bm A),  
   $ 
   such that $G_k$ $\Gamma$-converges to $G$. By Theorem 2.1 in \citet{braides2014local}, any cluster point of a sequence of minizers of $G_k$ minimizes $G$, proving the claim. 

   Finally, consider the case where $(\bm A_{\varepsilon_k})_{k\in\mathbb N}$ satisfies the conditions of part 3 of \cref{thm:approxConvergence}, i.e., $\bm A_{\varepsilon_k}\to \bm A^{\star}$ and $\bm A_{\varepsilon_{k}}$ minimizes $\Phi_{\varepsilon_k}$ on a closed ball of fixed radius $r>0$ centred at $\bm A^{\star}$, say $B^{\star}_r$. A simple adaptation of the previous arguments yields that $G_k$ restricted to $B^{\star}_r$ $\Gamma$-converges to $G$ restricted to $B^{\star}_r$. As such, $\bm A^{\star}$ minimizes $\Phi_0$ over $B^{\star}_r$ and is thus locally minimal for $\Phi_0$.      
    }

\section{Concluding Remarks}\label{sec: summary}

\edit{%

This work studied efficient computation of the quadratic EGW alignment problem between Euclidean spaces subject to non-asymptotic convergence guarantees. Despite the availability of various heuristic methods for computing EGW, formal guarantees beyond asymptotic convergence to a stationary point %
were absent until now. To develop our algorithms, we leveraged the variational form of the EGW distance that ties it to the well-understood EOT problem with a certain parametrized cost function. By analyzing the stability of the variational problem, its convexity and smoothness properties were established, which led to two new efficient algorithms for computing the EGW distance. The complexity of our algorithms agree with the best known complexity of $O(N^2)$ for computing the quadratic EGW distance directly, but unlike previous approaches, our methods are subject to non-asymptotic convergence rate guarantees to global/local solutions in the convex/non-convex regime. As the first derivative of the objective function depends on the solution to an EOT problem which must be solved numerically, we quantify both the error incurred by Sinkhorn's algorithm and the resulting effect on the convergence of both algorithms. Moreover, we establish a suitable notion of convergence of solutions to the variational EGW problem to those of the variational GW problem in the limit of vanishing regularization. Below, we discuss possible extensions and future research direction stemming from this work.  
}

\medskip
\noindent\textbf{Algorithmic improvements.} The stability analysis of the variational problem lays the groundwork for solving the EGW problem via smooth optimization methods. Consequently, improvements or alternatives to the proposed accelerated gradient methods are of great practical interest. For instance, marked improvements can be attained by analyzing the tradeoff between the per iteration cost associated with updating the step size parameter and the resulting decrease in the number of iterations required for convergence. Similarly, establishing sharper bounds on the eigenvalues of the Hessian would improve our characterization of the smoothness and convexity properties of the objective.

\medskip
\noindent\textbf{Expanding duality theory.} To the best of our knowledge, the current duality theory for the EGW problem is limited to the quadratic and inner product costs over Euclidean spaces. As the present work makes heavy use of this duality theory, we anticipate that these results could be extended to the EGW problem with other costs and/or spaces once an adequate duality theory has been established. Non-Euclidean spaces are not only of theoretical, but also of practical interest under the GW paradigm as they allow comparing/aligning important examples such as manifold or graph data.

\acks{Z. Goldfeld is partially supported by NSF grants CCF-2046018, DMS-2210368, and CCF-2308446, and the IBM Academic Award.
K. Kato is partially supported by the NSF grants DMS-1952306, DMS-2014636, and DMS-2210368. G. Rioux is partially supported by the NSERC Postgraduate Fellowship PGSD-567921-2022.}

\vskip 0.2in
\bibliography{ref}

\newpage

\appendix

\section{Sharpness of variance bound from \texorpdfstring{\cref{cor:HessianMinimalEigenvalue}}{Corollary 2}} \label{app:VarianceBound}

   Let $\mu_0=\frac{1}{2}\left(\delta_{0}+\delta_{a}\right)$ and $\mu_1=\frac{1}{2}\left(\delta_{0}+\delta_{b}\right)$ for $a\in\RR^{d_0}$ and $b\in\RR^{d_1}$. In this case, any coupling $\pi\in \Pi(\mu_0,\mu_1)$ is of the form $\pi_{00}\delta_{(0,0)}+\pi_{0b}\delta_{(0,b)}+\pi_{a0}\delta_{(a,0)}+\pi_{ab}\delta_{(a,b)}$ with the constraint that $\pi_{00}=\pi_{ab}$ and $\pi_{0b}=\pi_{a0}=\frac{1}{2}-\pi_{ab}$. For any $\bm A\in\cD_M$, $\OT_{\bm A,\varepsilon}(\mu_0,\mu_1)$ is given by  
   \begin{align*}
     &   \inf_{\pi\in\Pi(\mu_0,\mu_1)}\left\{\int-4\|x\|^2\|y\|^2-32 x^{\intercal}\bm Ay\,d\pi(x,y)+\varepsilon\KL(\pi\|\mu_0\otimes \mu_1)\right\}
        \\
        &
     =\inf_{\pi_{ab}\in(0,1/2)}\left\{-\pi_{ab}(4\|a\|^2\|b\|^2+32a^{\intercal}\bm Ab)+2\varepsilon \pi_{ab}\log(4\pi_{ab})+\left(1-2\pi_{ab}\right)\varepsilon\log\left(2-4\pi_{ab}\right)\right\}, 
   \end{align*}
   the objective is a sum of convex functions and the first-order optimality condition reads
   \[
        4\|a\|^2\|b\|^2+32a^{\intercal}\bm Ab=2\varepsilon\log(4\pi_{ab})-2\varepsilon\log(2-4\pi_{ab})\iff {\pi_{ab}}= \frac{e^{z}}{2\left(1+e^{z}\right)},
   \]
   for $z=\left(2\|a\|^2\|b\|^2+16a^{\intercal}\bm Ab\right)/\varepsilon$. Let $\pi^{\star}$ be the corresponding EOT coupling for $\OT_{\bm A,\varepsilon}(\mu_0,\mu_1)$. For any $\bm C\in\RR^{d_0\times d_1}$, 
   \[
   \Var_{\pi^{\star}}[X^{\intercal}\bm C Y]=\pi^{\star}_{ab}(1-\pi^{\star}_{ab})(a^{\intercal}\bm Cb)^2\leq\pi^{\star}_{ab}(1-\pi^{\star}_{ab})\|\bm C\|_F^2\|a\|^2\|b\|^2,
   \]
   with equality for $\bm C=C ab^{\intercal}$ with $C\in \RR$. Hence, 
   \[
   \sup_{\|\bm C\|_F=1}\left\{\Var_{\pi^{\star}}[X^{\intercal}\bm C Y]\right\}=\pi^{\star}_{ab}(1-\pi^{\star}_{ab})\|a\|^2\|b\|^2,
   \]
   which can be made arbitrarily close to $\frac{1}{4}\|a\|^2\|b\|^2$ for fixed $a,b$ by choosing $\bm A\in\cD_M$ and $\varepsilon>0$ as to make $z$ sufficiently large. On the other hand, $\sqrt{M_4(\mu_0)M_4(\mu_1)}=\frac{1}{2}\|a\|^2\|b\|^2$, such that the variance bound in \cref{cor:HessianMinimalEigenvalue} is tight up to a constant factor.

\section{Sinkhorn's Algorithm as an inexact oracle}
\label{app:SinkhornInexactAlg}

Given $\mu_0=\sum_{i=1}^{N_0}{a_i}\delta_{x^{(i)}}\in\mathcal P(\mathbb R^{d_0})$ and $\mu_1=\sum_{j=1}^{N_1}{b_j}\delta_{y^{(j)}}\in\mathcal P(\mathbb R^{d_1})$, let $a,b$ denote the corresponding (positive) probability vectors.
Fix an underlying cost function $c:\mathbb R^{d_0}\times \mathbb R^{d_1}\to \mathbb R$ and $\varepsilon>0$, and let 
 $\bm K\in\mathbb R^{N_0\times N_1}$ with $\bm K_{ij}=e^{-\frac{c\left( x^{(i)},y^{(j)}\right)}{\varepsilon}}$.  
Consider the standard implementation of Sinkhorn's algorithm (cf. e.g. \citealt{cuturi2013sinkhorn,flamary2021pot}).

\begin{algorithm}
\caption{Sinkhorn Algorithm}
\label{alg:SinkhornAlg}
\begin{algorithmic}[1]
    \State Fix a threshold $\gamma$ and a maximum iteration number $k_{\max}$.
    \State $u_0\gets \mathds{1}_{N_0}/N_0$
\State $k\gets 1$
\Repeat
\State $v_k\gets b/(\bm K^{\intercal}u_{k-1})$
\State $u_k\gets a/(\bm K v_k)$
\State $\bm \Pi^k\gets \diag(u_k)\bm K\diag(v_k)$
\State $k\gets k+1$
\Until $\|(\bm \Pi^k)^{\intercal}\mathds{1}_{N_0} -b\|_2<\gamma$ or $k>k_{\max}$
\State \textbf{return} $\bm \Pi^k$
\end{algorithmic}
\end{algorithm}
In \cref{alg:SinkhornAlg} and the remainder of this section, the division of two vectors is understood componentwise. The stopping condition is based only on one of the marginal constraints, as $\bm \Pi^k\mathds{1}_{N_1}=a$ by construction. 

\medskip
The following definitions enable describing the convergence properties of \cref{alg:SinkhornAlg}; we follow the approach of \citet{franklin1989scaling} with only minor modifications.
Let $\mathbb R_+^d$ denote the set of vectors with positive entries and, for $x,y\in\mathbb R_+^d$, let  
\[
    \mathsf d_H(x,y)=\log \max_{1\leq i,j\leq d}\frac{x_iy_j}{y_ix_j},
\]
denote Hilbert's projective metric\footnote{$\mathsf d_H(x,y)=0$ if and only if $x=\alpha y$ for $\alpha>0$, 
 $\mathsf d_H$ is symmetric and satisfies the triangle inequality.} on $\mathbb R_+^d$. By definition, 
 \begin{equation}
     \label{eq:HilbertProjProp}
     \mathsf d_H(x,y)= \mathsf d_H(x/y,\mathds{1}_d),
 \end{equation}
 for any $x,y\in\mathbb R^{d}_+$ and, setting $x=e^w,y=e^z$ componentwise, 
 \begin{equation}
     \label{eq:projMetricAlt}
 \begin{aligned}
     \mathsf d_{H}(x,y)&=\log \max_{1\leq i,j\leq d}e^{w_i+z_j-w_j-z_i},
     \\
     &=\max_{1\leq i,j\leq d}w_i+z_j-w_j-z_i,
     \\
     &=\max_{1\leq i\leq d}(\log x_i-\log y_i)-\min_{1\leq i\leq d}(\log x_i-\log y_i),
     \\
     &=\max_{1\leq i\leq d}\log\left(\frac{x_i}{y_i}\right)-\min_{1\leq i\leq d}\log\left(\frac{x_i}{y_i}\right).
 \end{aligned}
 \end{equation}
 It was proved in \citet{birkhoff1957extensions,samelson1957Perron} that multiplication with a positive matrix is a strict contraction w.r.t. $\mathsf d_{H}$. Precisely, 
\begin{equation}
\label{eq:contraction}
    \mathsf d_{H}(\bm Ax,\bm Ay)\leq \lambda(\bm A)\mathsf d_{H}(x,y), 
\end{equation}
for any $\bm A\in \mathbb R^{d'\times d}_+$ and $x,y\in\mathbb R^{d}_+$, where 
\[
    \lambda(\bm A)=\frac{\sqrt{\eta(\bm A)}-1}{\sqrt{\eta(\bm A)}+1}<1,\quad \eta(\bm A)=\max_{\substack{1\leq i,j\leq d'\\1\leq k,l\leq d}}\frac{\bm A_{ik}\bm A_{jl}}{\bm A_{jk}\bm A_{il}}.
\]
Of note is that $\lambda(\bm A)=\lambda(\bm A^{\intercal})$. Let 
\[
    E= \{\bm A\in\mathbb R^{N_0\times N_1}_+:\bm A=\diag(x)\bm K\diag(y)\text{ for some } x\in\mathbb R_+^{N_0},y\in\mathbb R_+^{N_1} \},
\]
and observe that if $\bm A,\bm B\in E$, there exists $x_{\bm A,\bm B}\in\mathbb R_+^{N_0}, y_{\bm A,\bm B}\in\mathbb R_+^{N_1}$ for which $\bm A=\diag(x_{\bm A,\bm B})\bm B\diag(y_{\bm A,\bm B})$. In this setting, let $\mathsf d:E\times E\mapsto [0,\infty)$ be such that 
\[
    \mathsf d(\bm A,\bm B)=\mathsf d_{H}(x_{\bm A,\bm B},\mathds 1_{N_0})+\mathsf d_{H}(y_{\bm A,\bm B},\mathds 1_{N_1}),
\]
then $\mathsf d$ is a metric on $E$. As the EOT coupling $\bm \Pi^{\star}$ satisfies
\[
    \frac{\bm \Pi^{\star}_{ij}}{a_ib_j}=e^{\frac{\varphi\left(x^{(i)}\right)+\psi\left(y^{(j)}\right)-c\left(x^{(i)},y^{(j)}\right)}{\varepsilon}},
\]
where $(\varphi,\psi)$ is any pair of EOT potentials, $\bm \Pi^{\star}=\diag(u^{\star})\bm K\diag(v^{\star})\in E$ for 
\[
    u^{\star}_i=a_ie^{\frac{\varphi\left(x^{(i)}\right)}{\varepsilon}},\quad v^{\star}_j=b_je^{\frac{\psi\left(y^{(j)}\right)}{\varepsilon}}. 
\]
Note that $u^{\star}=a/(\bm Kv^{\star})$ and $v^{\star}=b/(\bm K^{\intercal}u^{\star})$. 

In the sequel, we analyze the convergence of $\bm \Pi^k$ to $\bm \Pi^{\star}$ under $\mathsf d$. The following result  
translates bounds on $\mathsf d(\bm \Pi^k,\bm \Pi^{\star})$ to bounds on $\|\bm \Pi^k-\bm \Pi^{\star}\|_{\infty}$.

\begin{lemma}
\label{lem:errEstimate}
    Fix $\delta>0$. If $\mathsf d(\bm \Pi^k,\bm \Pi^{\star})\leq \delta$, it follows that 
    $
    \|\bm \Pi^k-\bm \Pi^{\star}\|_{\infty}\leq e^{\delta}-1.
    $
\end{lemma}

\begin{proof}
\medskip
    By Lemma 3 in \citet{franklin1989scaling}, whenever $\mathsf d(\bm\Pi^k,\bm \Pi^{\star})\leq \delta$ it holds
    that
    \[
        e^{-\delta}-1\leq   \frac{\bm \Pi^{\star}_{ij}}{\bm \Pi^{k}_{ij}}-1\leq e^{\delta}-1,
    \]
    for every $1\leq i\leq N_0,1\leq j\leq N_1$. As such,
    \begin{align*}
        |\bm \Pi^{\star}_{ij}-\bm \Pi^{k}_{ij}|\leq \bm \Pi_{ij}^k\left((1-e^{-\delta})\vee(e^{\delta}-1) \right)\leq (1-e^{-\delta})\vee(e^{\delta}-1)=e^{\delta}-1, 
    \end{align*}
    yielding $\|\bm \Pi^{\star}-\bm \Pi^k\|_{\infty}\leq e^{\delta}-1$.
\end{proof}

The number of iterations required to achieve $\mathsf d(\bm \Pi^k,\bm \Pi^{\star})\leq \delta$ can be bounded as follows.

   \begin{proposition}
   \label{prop:Iter}
       Let $\bm \Pi^k$ be given by \cref{alg:SinkhornAlg} and fix $\delta>0$. Then, if 
       \[
       \begin{gathered}
       k\geq 1+\frac{1}{2\log\left(\lambda(\bm K)\right)}{\log\left(\frac{\delta(1-\lambda(\bm K))}{\mathsf d_H((\bm \Pi^1)^{\intercal}\mathds{1}_{N_0} ,b)}\right)},
      \end{gathered} 
       \] it follows that
       $\mathsf d(\bm \Pi^k,\bm \Pi^{\star})\leq \delta$. 
   \end{proposition}

The proof of \cref{prop:Iter} follows immediately from \cref{lem:iterProgPlan} ahead. We first prove an auxiliary lemma. 

\begin{lemma}
\label{lem:iterProgPotentials}
    For $k\geq 1$, the iterates $u_k,v_k$ of \cref{alg:SinkhornAlg} satisfy 
    \begin{gather*} 
       \mathsf d_H(u_{k+1},u^{\star})\leq \lambda(\bm K)^2\mathsf d_H(u_k,u^{\star}),\quad \mathsf d_H(v_{k+1},v^{\star})\leq \lambda(\bm K)^2\mathsf d_H(v_k,v^{\star}), 
    \\
        \mathsf d_H(u_k,u^{\star})\leq \frac{\mathsf d_{H}(u_k,u_{k+1})}{1-\lambda(\bm K)^2},\quad \mathsf d_H(v_k,v^{\star})\leq \frac{\mathsf d_{H}(v_k,v_{k+1})}{1-\lambda(\bm K)^2}.  
    \end{gather*} 
\end{lemma} 

\begin{proof}
To prove the first claim, we have by \eqref{eq:contraction} that 
   \[ 
    \mathsf d_H(u_{k+1},u^{\star})=\mathsf d_H\left(a/(\bm Kv_{k+1}),a/(\bm Kv^{\star})\right)=\mathsf d_H(\bm Kv^{\star},\bm Kv_{k+1})\leq \lambda(\bm K)\mathsf d_H(v_{k+1},v^{\star}).
    \]
    It follows similarly that $\mathsf d_H(v_{k+1},v^{\star})\leq \lambda(\bm K)\mathsf d_H(u_{k},u^{\star})$. Combining these bounds, 
    \[
    \mathsf d_H(u_{k+1},u^{\star})\leq \lambda(\bm K)^2\mathsf d_H(u_{k},u^{\star}),\quad \mathsf d_H(v_{k+1},v^{\star})\leq \lambda(\bm K)^2\mathsf d_H(v_{k},v^{\star}),
    \] 
    which proves the first claim. Applying 
     the triangle inequality yields
    \[
        \mathsf d_H(u_{k},u^{\star})\leq \mathsf d_H(u_{k+1},u^{\star})+\mathsf d_H(u_k,u_{k+1})\leq \lambda(\bm K)^2\mathsf d_H(u_{k},u^{\star})+\mathsf d_H(u_k,u_{k+1}),
    \]
   such that $(1-\lambda(\bm K)^2)\mathsf d_H(u_{k},u^{\star})\leq\mathsf d_H(u_k,u_{k+1})$ which proves the second claim; the same argument holds for the iterates $v_k$. 
\end{proof}

\cref{lem:iterProgPlan} translates the bound from
\cref{lem:iterProgPotentials} to a bound on $\mathsf d(\Pi^k,\Pi^{\star})$. We introduce the notation  $\odot$ to denote the componentwise product of vectors.

\begin{lemma}
    \label{lem:iterProgPlan}
    For $k\geq 2$, $\mathsf d(\bm \Pi^k,\bm \Pi^{\star})\leq \frac{\lambda(\bm K)^{2(k-1)}}{1-\lambda(\bm K)}\mathsf d_H((\bm \Pi^1)^{\intercal}\mathds{1}_{N_0} ,b)$. 
\end{lemma}

\begin{proof}
   As $\bm \Pi^k=\diag(u_k)\bm K\diag(v_k)$ and $\bm \Pi^{\star}=\diag(u^{\star})\bm K\diag(v^{\star})$,  
  \[ 
   \begin{aligned}
        \mathsf d(\bm \Pi^k,\bm \Pi^{\star})&=\mathsf d_H(u_k,u^{\star})+\mathsf d_H(v_k,v^{\star})
        \\
        &\leq \lambda(\bm K)^{2(k-1)}(1+\lambda(K))d_H(v_1,v^{\star})
        \\
        &\leq \frac{\lambda(\bm K)^{2(k-1)}}{1-\lambda(\bm K)}d_H(v_1,v_2), 
   \end{aligned}
  \] 
   where both inequalities follow from \cref{lem:iterProgPotentials} and its proof. Finally,
   \[
        \mathsf d_H(v_1,v_2)= \mathsf d_H\left(v_1,\frac{b}{\bm K^{\intercal}u_1}\right)=\mathsf d_H\left(v_1\odot \bm K^{\intercal}u_1,{b}\right)=\mathsf d_H((\bm \Pi^1)^{\intercal}\mathds{1}_{N_0} ,b).
   \]
\end{proof}

Now, we demonstrate why the termination condition based on the $2$-norm endows us with a $\delta'$-oracle approximation and provide a bound on the number of iterations required to achieve it. Theorem 1 in \citet{dvurechensky2018computational} proves that there exists $\bar k\leq 1+\frac{R}{\gamma}$ satisfying  
\[ \|u_{\bar k}\odot \bm K v_{\bar k+1} -a\|_1+\|(\bm \Pi^k)^{\intercal}\mathds{1}_{N_0} -b\|_1\leq \gamma, \]
for $R=-{2\log\left(e^{-\|\bm C\|_{\infty}/\varepsilon}\min_{\substack{1\leq i\leq N_0\\1\leq j\leq N_1}}a_i\wedge b_j\right)}$. This gives a bound on the maximal number of iterations to achieve the $2$-norm termination condition via the standard inequality $\|\cdot\|_2\leq \|\cdot\|_1$. We clarify that the analysis in \citet{dvurechensky2018computational} is for a slightly different implementation of Sinkhorn's algorithm. First, running  \cref{alg:SinkhornAlg} is tantamount to running their algorithm with reversed marginals. Next, one iteration of \cref{alg:SinkhornAlg} corresponds to two iterations in their analysis. Finally, their approach uses $\mathds 1_{N_0}$ rather than $\mathds 1_{N_0}/N_0$ for the initialization. This difference is innocuous for achieving the termination condition, as the iterates are identical up to multiplying $v_k$ by $N_0$ and dividing $u_k$ by $N_0$; $\bm \Pi^k$ is invariant this operation.   

We now bound $\mathsf d_{H}$ in terms of the Euclidean distance with the aim of controlling $\mathsf d(\bm \Pi^k,\bm \Pi^{\star})$ by $\|(\bm \Pi^k)^{\intercal}\mathds{1}_{N_0} -b\|_2$.   

\begin{lemma}
    \label{lem:normbounds}
    Let $r,s \in \mathbb R^d_+$ be arbitrary,
    then 
    \begin{align*}
        \mathsf d_{H}(s,r)&\leq (r_{i^{\star}}^{-1}+ s_{i_{\star}}^{-1} )\|r-s\|_2,
    \end{align*}
    where $i^*\in \argmax_{1\leq i \leq d} \frac{s_i}{r_i}$ and $i_*\in \argmin_{1\leq i \leq d} \frac{s_i}{r_i}$.
\end{lemma}

\begin{proof}
We have by \eqref{eq:projMetricAlt} that 
\[
    \mathsf d_{H}(s,r)=\max_{1\leq i\leq N_0}\log\left(\frac{s_i}{r_i}\right)-\min_{1\leq i\leq N_0}\log\left(\frac{s_i}{r_i}\right).
\]
Observe that 
$
    1-\frac{r_i}{s_i}\leq \log\left(\frac{s_i}{r_i}\right)\leq \frac{s_i}{r_i}-1,
$
as follows  the inequalities $\frac{x}{1+x}\leq \log(1+x)\leq x$ for $x>-1$. Whence, 
\[
    \begin{aligned}
    \mathsf d_{H}(s,r)&
\leq r_{i^{\star}}^{-1}\left(s_{i^{\star}}-r_{i^{\star}}\right) -s_{i_{\star}}^{-1}\left(s_{i_{\star}}-r_{i_{\star}}\right)
    \\
    &\leq  \left(r_{i^{\star}}^{-1}+ s_{i_{\star}}^{-1}\right) \left\| s-r\right\|_2.
\end{aligned}
\]
\end{proof}

Note that the bound in \cref{lem:normbounds} is symmetric in the sense that interchanging $s$ and $r$ does not modify the constant. This can be seen by noting that $\argmax_{1\leq i \leq d} \frac{s_i}{r_i}=\argmin_{1\leq i \leq d} \frac{r_i}{s_i}$ and $\argmin_{1\leq i \leq d} \frac{s_i}{r_i}=\argmax_{1\leq i \leq d} \frac{r_i}{s_i}$. 
By combining Lemmas \ref{lem:iterProgPlan} and \ref{lem:normbounds} we arrive at the desired result.  
\begin{proposition}
\label{prop:marginalViolation}
      Let $\underline{b}=\min_{1\leq i\leq N_1}b_i$ and set $0<\gamma < \underline{b}$. Then, for $k\geq 1$,
    \[\mathsf d(\bm \Pi^k,\bm \Pi^{\star})\leq \frac{\left((\bm \Pi^k)^{\intercal}\mathds{1}_{N_0}\right)_{i_{\star}}^{-1}+b_{i^{\star}}^{-1}}{1-\lambda(\bm K)}\|(\bm \Pi^k)^{\intercal}\mathds{1}_{N_0}-b\|
    \]
 where $i^*\in \argmax_{1\leq i \leq N_1} \frac{\left((\bm \Pi^k)^{\intercal}\mathds{1}_{N_0}\right)_i}{b_i}$ and $i_*\in \argmin_{1\leq i \leq N_1} \frac{\left((\bm \Pi^1)^{\intercal}\mathds{1}_{N_0}\right)_i}{b_i}$. Further, there exists $\bar k\leq 1+\frac{R}{\gamma}$ for which  $\|{\bm\Pi^{\bar k}}^{\intercal}\mathds{1}_{N_0}-b\|\leq \gamma$ and 
 $
\mathsf d({\bm\Pi^{\bar k}},\bm \Pi^{\star})\leq \frac{(\underline b-\gamma)^{-1}+\underline b^{-1}}{1-\lambda(\bm K)}\gamma.
$
In particular, setting 
\[
\gamma =\bar\alpha \underline b\text{ for } \bar\alpha=\frac{\delta(1-\lambda(\bm K))+2-\sqrt{\delta^2(1-\lambda(\bm K))^2+4}}{2}\in(0,1),
\]
it holds that $\mathsf d(\bm \Pi^k,\bm\Pi^{\star})=\delta$.

\end{proposition}

\begin{proof}
   From the proof of  \cref{lem:iterProgPlan}, 
   \[
   \mathsf d(\bm \Pi^k,\bm \Pi^{\star})\leq (1+\lambda(\bm K))\mathsf d_H(v_k,v^{\star})\leq \frac{\mathsf d_H(v_k,v_{k+1})}{1-\lambda(\bm K)}=\frac{\mathsf d_H((\bm \Pi^k)^{\intercal}\mathds 1_{N_0},b)}{1-\lambda(\bm K)},  
   \]
   where the final inequality and equality stem from \cref{lem:iterProgPotentials} 
 and its proof. Applying \cref{lem:normbounds} proves the first claim. 
 
 As for the second claim,  it is clear from the discussion preceding \cref{lem:normbounds} that $\|(\bm \Pi^{\bar k})^{\intercal} \mathds 1_{N_0}-b\|\leq \gamma$ for some $\bar k \leq 1+\frac{R}{\gamma}$. Now, let $w=(\bm \Pi^{\bar k})^{\intercal}\mathds 1_{N_0}$ and observe that $\|w-b\|_{\infty}\leq \|w-b\|\leq \gamma$ such that $b_i-\gamma\leq w_i\leq b_i+\gamma$ for $i=1,\dots,N_1$. Hence $w_i^{-1}\leq (b_i-\gamma)^{-1}\leq(\underline b-\gamma)^{-1}$ as $\gamma <\underline b$. Applying this bound to the previous inequality proves the claim. 

For the final claim, observe that $\bar{\alpha}$ solves the equation $\frac{2\alpha-\alpha^2}{1-\alpha}=\delta(1-\lambda(\bm K))$ for $\alpha\in(0,1)$ (indeed, $\delta(1-\lambda(\bm K))<\sqrt{\delta^2(1-\lambda(\bm K))^2+4} <{\delta(1-\lambda(\bm K))+2} $). Setting $\gamma=\bar{\alpha}\underline{b}$ in the previously derived bound on $\mathsf d(\bm \Pi^k,\bm \Pi^{\star})$ gives 
\[
\mathsf d(\bm \Pi^k,\bm \Pi^{\star})\leq \frac{(1-\bar\alpha)^{-1}+1}{1-\lambda(\bm K)}\bar \alpha=\frac{2\bar\alpha-\bar\alpha^2}{(1-\bar \alpha)(1-\lambda(\bm K))}=\delta.
\]

\end{proof}

The proof of \cref{rmk:InexactSinkhorn} follows by combining Propositions \ref{prop:Iter} and \ref{prop:marginalViolation}; the maximal number of iterations for \cref{alg:SinkhornAlg} to output a matrix $\tilde {\bm \Pi}$ satisfying $\mathsf d(\tilde {\bm \Pi},\bm \Pi^{\star})\leq \delta$ is  
\begin{equation}
\label{eq:maxIters}    
\begin{split}
\tilde k =&
\min \Bigg\{
1+\frac{1}{2\log\left(\lambda(\bm K)\right)}{\log\left(\frac{\delta(1-\lambda(\bm K))}{\mathsf d_H((\bm \Pi^1)^{\intercal}\mathds{1}_{N_0} ,b)}\right)}
,   
   \\
&\hspace{4em}1-\frac{4\underline b^{-1}}{\delta(1-\lambda(\bm K))+2-\sqrt{\delta^2(1-\lambda(\bm K))^2+4}}\log\Bigg(e^{-\|\bm C\|_{\infty}/\varepsilon}\min_{\substack{1\leq i\leq N_0\\1\leq j\leq N_1}}a_i\wedge b_j\Bigg)
\Bigg\},
\end{split}
\end{equation}
which corresponds to an $(e^{\delta}-1)$-oracle for the entropic transport plan in light of \cref{lem:errEstimate}.

\section{Convergence of \texorpdfstring{\cref{alg:fgmNonConvexInexact}}{Algorithm 2}}
\label{app:adaptRates}

In what follows, we slightly adapt the proof of Theorem $2$ in \citet{ghadimi2016accelerated} to conform to the inexact setting. We first clarify that they treat the composite problem 
\[
    \inf_{x\in \mathbb R^d} f(x)+g(x)+\mathcal Q(x),
\]
where $f$ is $L'$-smooth and non-convex, $g$ is  $L''$-smooth and convex, and $\mathcal Q$ is non-smooth and convex with a bounded domain. Hence $f+g$ is $L=L'+L''$ smooth and possibly non-convex. 

Our problem conforms to this setting (up to vectorization) with $f=\OT_{(\cdot),\varepsilon}(\mu_0,\mu_1)$, $g=32\|\cdot\|_{F}^2$, and $\mathcal Q=\mathcal I_{\mathcal D_M}$,  the indicator function of the set $\mathcal D_M$, defined by
\[\mathcal I_{\mathcal D_M}(\bm A)=
\begin{cases}
0,&\text{if }\bm A\in\mathcal D_M,
\\+\infty,&\text{otherwise.}
\end{cases}
\]
When $\Phi$ is convex, we set $f=0$ and $g=\Phi$ hence $L'=0$, $L=L''$.

As $\Phi$ is $L$-smooth, by Lemma $5$ in \citet{ghadimi2016accelerated},
\begin{equation}
\label{eq:LSmoothBound1}
    \Phi(\bm B_k)\leq \Phi(\bm A_{k})+\Tr\left(D\Phi_{[\bm A_{k}]}^{\intercal}(\bm B_k-\bm A_{k})\right)+\frac{L}{2}\left\|\bm B_k-\bm A_{k}\right\|^2_F, 
\end{equation}
and for any $\bm H\in \RR^{d_0\times d_1}$, letting $L'$ denote the Lipschitz constant of $\OT_{(\cdot),\varepsilon}(\mu_0,\mu_1)$, the same result yields 
\begin{equation}
\begin{aligned}
    \label{eq:LSmoothBound2}
    &\Phi(\bm A_{k})-\left((1-\tau_k)\Phi(\bm B_{k-1})+\tau_k\Phi(\bm H)\right)
    \\
    &=\tau_k\left(\Phi(\bm A_k)-\Phi(\bm H)\right)+(1-\tau_k)\left(\Phi(\bm A_k)-\Phi(\bm B_{k-1})\right)
    \\
    &\leq \tau_k\left(\Tr\left(D\Phi_{[\bm A_k]}^{\intercal}(\bm A_{k}-\bm H)\right)+\frac{L'}{2}\left\|\bm H-\bm A_{k}\right\|^2_F\right)
    \\
    &+(1-\tau_k)\left(\Tr\left(D\Phi_{[\bm A_k]}^{\intercal}(\bm A_{k}-\bm B_{k-1})\right)+\frac{L'}{2}\left\|\bm B_{k-1}-\bm A_{k}\right\|^2_F\right)
    \\
    &=\Tr\left(D\Phi_{[\bm A_{k}]}^{\intercal}(\bm A_{k}-\tau_k\bm H-(1-\tau_k)\bm B_{k-1})\right)
    \\&+\frac{L'\tau_k}{2}\left\|\bm H-\bm A_{k}\right\|^2_F+\frac{L'(1-\tau_k)}{2}\underbrace{\left\|\bm B_k-\bm A_{k}\right\|^2_F}_{\tau_k^2\left\|\bm B_{k-1}-\bm C_{k-1}\right\|^2_F},
\end{aligned}
\end{equation}
recalling the update $\bm A_{k}=\tau_k \bm C_{k-1}+(1-\tau_k)\bm B_{k-1}$.

 Denote the subdifferential of $\mathcal I_{\cD_M}$ at $\bm A\in\RR^{d_0\times d_1}$ by 
\[
\partial\mathcal I_{\cD_M}(\bm A)\mspace{-2mu}\coloneqq\mspace{-2mu} \left\{ \bm P\mspace{-2mu} \in\mspace{-2mu}\mathbb R^{d_0\times d_1}:\mathcal I_{\cD_M}(\bm X)\mspace{-2mu}-\mspace{-2mu}\mathcal I_{\cD_M}(\bm A)\geq \Tr\left(\bm P^{\intercal}(\bm X\mspace{-2mu}-\mspace{-2mu}\bm A)\right),\text{ for every }\bm X\mspace{-2mu}\in\mspace{-2mu}\RR^{d_0\times d_1}\right\}.\]
As $\bm C_k$ is optimal for the problem $\argmin_{\bm V\in \RR^{d_0\times d_1}}\left\{\frac{1}{2\gamma_k}\|\bm V-\left(\bm C_{k-1}-\gamma_k\bm G_k\right)\|_F^2+\mathcal I_{\cD_M}(\bm V)\right\}$, there exists $\bm P\in \partial \mathcal I_{\cD_M}(\bm C_k)$ for which $\bm G_k+\bm P+\frac{1}{\gamma_k}(\bm C_k-\bm C_{k+1})=0$ (see Theorem 23.8, Theorem 25.1, and p. 264 in \citealt{rockafellar1997convex}). Thus, for any $\bm U\in \RR^{d_0\times d_1}$,
\begin{align*}
    \Tr\left((\bm G_k+\bm P)^{\intercal}(\bm C_k-\bm U)\right)&=\frac{1}{\gamma_k}\Tr\left((\bm C_k-\bm C_{k-1})^{\intercal} (\bm U-\bm C_k)\right)
    \\
    &=\frac{1}{2\gamma_k}\left(
    \|\bm C_{k-1}-\bm U\|^2_F-\|\bm C_{k}-\bm U\|_F^2-\|\bm C_{k}-\bm C_{k-1}\|_F^2\right),
\end{align*}
where the final line follows from some simple algebra. As $\bm P\in\partial \mathcal I_{\cD_M}(\bm C_k)$, $\Tr(\bm P^{\intercal}(\bm C_k-\bm U))\geq \mathcal I_{\cD_M}(\bm C_k)-\mathcal I_{\cD_M}(\bm U)=-\mathcal I_{\cD_M}(\bm U)$, whence
    \begin{equation}
    \label{eq:traceInterStep}
    \Tr\left(\bm G_k^{\intercal}(\bm C_k-\bm U)\right)\leq \mathcal I_{\cD_M}(\bm U)+
    \frac{1}{2\gamma_k}\left(
    \|\bm C_{k-1}-\bm U\|^2_F-\|\bm C_{k}-\bm U\|_F^2-\|\bm C_{k}-\bm C_{k-1}\|_F^2\right).
\end{equation}    

By the same steps applied to the other subproblem with $\bm B_k$ and $\bm A_k$ taking the place of $\bm C_{k}$ and $\bm C_{k-1}$ respectively,  
\[
\Tr\left(\bm G_k^{\intercal}(\bm B_k-\bm U)\right)\leq \mathcal I_{\cD_M}(\bm U)
+
    \frac{1}{2\beta_k}\left(
    \|\bm A_k-\bm U\|^2_F-\|\bm B_{k}-\bm U\|_F^2-\|\bm B_{k}-\bm A_{k}\|_F^2\right).
\]
Setting $\bm U=\tau_k\bm C_k+(1-\tau_k)\bm B_{k-1}\in \cD_M$ (by convexity) in the previous display, bounding $-\|\bm B_{k}-\bm U\|_F^2$ above by $0$, and recalling that $\bm{A}_k=\tau_k\bm C_{k-1}+(1-\tau_k)\bm B_{k-1}$ such that $\bm A_k-\bm U=\tau_k(\bm C_{k-1}-\bm C_k)$,
\[
\Tr\left(\bm G_k^{\intercal}(\bm B_k-\tau_k\bm C_k+(1-\tau_k)\bm B_{k-1})\right)\leq 
    \frac{1}{2\beta_k}\left(\tau_k^2
    \|\bm C_{k}-\bm C_{k-1}\|^2_F-\|\bm B_{k}-\bm A_{k}\|_F^2\right).
\]
Combining with \eqref{eq:traceInterStep} upon scaling by $\tau_k$, 
\begin{align*}
\Tr\left(\bm G_k^{\intercal}(\bm B_k-\tau_k\bm U+(1-\tau_k)\bm B_{k-1})\right)&\leq\tau_k\mathcal I_{\cD_M}(\bm U)+ 
    \frac{1}{2\beta_k}\left(\tau_k^2
    \|\bm C_{k}-\bm C_{k-1}\|^2_F-\|\bm B_{k}-\bm A_{k}\|_F^2\right),
\\
&
+
    \frac{\tau_k}{2\gamma_k}\left(
    \|\bm C_{k-1}-\bm U\|^2_F-\|\bm C_{k}-\bm U\|_F^2-\|\bm C_{k}-\bm C_{k-1}\|_F^2\right), 
\end{align*}
by the choice of $\tau_k$, $\beta_k$, $\gamma_k$, we have that $\frac{\tau_k^2}{\beta_k}-\frac{\tau_k}{\gamma_k}\leq 0$ such that 
\begin{align*}
\Tr\left(\bm G_k^{\intercal}(\bm B_k-\tau_k\bm U+(1-\tau_k)\bm B_{k-1})\right)&\leq \tau_k\mathcal I_{\cD_M}(\bm U)+
    \frac{\tau_k}{2\gamma_k}\left(
    \|\bm C_{k-1}-\bm U\|^2_F-\|\bm C_{k}-\bm U\|_F^2\right)
    \\
    &-    \frac{1}{2\beta_k}\|\bm B_{k}-\bm A_{k}\|_F^2.
\end{align*}
Combining the equation above with \eqref{eq:LSmoothBound1} and \eqref{eq:LSmoothBound2} and setting $\bm H=\bm U\in \cD_M$ (otherwise the bound is vacuous),
\begin{align*}
   \Phi(\bm B_k)\mspace{-2mu}-\mspace{-2mu}\Phi(\bm H)&\mspace{-2mu}\leq\mspace{-2mu} (1\mspace{-2mu}-\mspace{-2mu}\tau_k)\left(\Phi(\bm B_{k-1})\mspace{-2mu}-\mspace{-2mu}\Phi(\bm H)\right)\mspace{-2mu}+\mspace{-2mu}\Tr\left(D\Phi_{[\bm A_{k}]}^{\intercal}(\bm B_{k}\mspace{-2mu}-\mspace{-2mu}\tau_k\bm H\mspace{-2mu}-\mspace{-2mu}(1\mspace{-2mu}-\mspace{-2mu}\tau_k)\bm B_{k-1})\right)
    \\&\mspace{-2mu}+\mspace{-2mu}\frac{L'\tau_k}{2}\left\|\bm H\mspace{-2mu}-\mspace{-2mu}\bm A_{k}\right\|^2_F\mspace{-2mu}+\mspace{-2mu}\frac{L'(1\mspace{-2mu}-\mspace{-2mu}\tau_k)}{2}{\tau_k^2\left\|\bm B_{k-1}\mspace{-2mu}-\mspace{-2mu}\bm C_{k-1}\right\|^2_F}+\frac{L}{2}\left\|\bm B_k\mspace{-2mu}-\mspace{-2mu}\bm A_{k}\right\|^2_F
    \\
    &\mspace{-2mu}\leq\mspace{-2mu} (1\mspace{-2mu}-\mspace{-2mu}\tau_k)\left(\Phi(\bm B_{k-1})-\Phi(\bm H)\right)\mspace{-2mu}+\mspace{-2mu}\delta'\mspace{-2mu}+\mspace{-2mu}
    \frac{\tau_k}{2\gamma_k}\left(
    \|\bm C_{k-1}\mspace{-2mu}-\mspace{-2mu}\bm H\|^2_F\mspace{-2mu}-\mspace{-2mu}\|\bm C_{k}\mspace{-2mu}-\mspace{-2mu}\bm H\|_F^2\right)
    \\&\mspace{-2mu}+\mspace{-2mu}\frac{L'\tau_k}{2}\left\|\bm H\mspace{-2mu}-\mspace{-2mu}\bm A_{k}\right\|^2_F\mspace{-2mu}+\mspace{-2mu}\frac{L'(1\mspace{-2mu}-\mspace{-2mu}\tau_k)}{2}{\tau_k^2\left\|\bm B_{k-1}\mspace{-2mu}-\mspace{-2mu}\bm C_{k-1}\right\|^2_F}\mspace{-2mu}+\mspace{-2mu}\left(\frac{L}{2}\mspace{-2mu}-\mspace{-3mu}    \frac{1}{2\beta_k}\right)\mspace{-2mu}\left\|\bm B_k\mspace{-2mu}-\mspace{-2mu}\bm A_{k}\right\|^2_F,
\end{align*}
where the inequality follows from the $\delta$-oracle which implies the bound (cf. \eqref{eq:deltapOracle})
\[
    \sup_{\bm Y,\bm Z\in\cD_M}\left\{\left|\Tr\left(\bm G_k-D\Phi_{[\bm A_k]})^{\intercal}( \bm Y-\bm Z)\right)\right|\right\}\leq \delta',
\] 
observing that $\bm B_k,\tau_k\bm H+(1-\tau_k)\bm B_{k-1}\in\mathcal D_M$ by convexity ($\tau_k\in(0,1]$).

Applying Lemma $1$ in \citet{ghadimi2016accelerated} yields, for $A_i=\frac{2}{i(i+1)}$,
\begin{align*}
    \frac{\Phi(\bm B_k)-\Phi(\bm H)}{A_k}&\leq \sum_{i=1}^kA_{i}^{-1}\left(\delta'+
\frac{\tau_i}{2\gamma_i}\left(
    \|\bm C_{i-1}-\bm H\|^2_F-\|\bm C_{i}-\bm H\|_F^2\right)
+\frac{L'\tau_i}{2}\left\|\bm H-\bm A_{i}\right\|^2
\right.
\\
&\left.
+\frac{L'(1-\tau_i)}{2}{\tau_i^2\left\|\bm B_{i-1}-\bm C_{i-1}\right\|^2_F}+\left(\frac{L}{2}-    \frac{1}{2\beta_i}\right)\left\|\bm B_i-\bm A_{i}\right\|^2
    \right)
\\
&\leq 
\frac{\|\bm C_0-\bm H\|_F^2}{2\gamma_1}+
\sum_{i=1}^kA_{i}^{-1}\left(\delta'
+\frac{L'\tau_i}{2}\left\|\bm H-\bm A_{i}\right\|^2
\right.
\\
&\left.
+\frac{L'(1-\tau_i)}{2}{\tau_i^2\left\|\bm B_{i-1}-\bm C_{i-1}\right\|^2_F}+\left(\frac{L}{2}-    \frac{1}{2\beta_i}\right)\left\|\bm B_i-\bm A_{i}\right\|^2
    \right).
\end{align*}
By convexity of $\|\cdot\|^2_F$,
\begin{align*}
    &\left\|\bm H-\bm A_{i}\right\|^2_F
+{\tau_i(1-\tau_i)}{\left\|\bm B_{i-1}-\bm C_{i-1}\right\|^2_F}
    \\
    &\leq 2\left(\|\bm H\|^2_F+\|\bm A_{i}\|^2_F
+\tau_i(1-\tau_i)\left(\|\bm B_{i-1}\|^2_F+\|\bm C_{i-1}\|^2_F\right)\right)
\\
&\leq 2\left(\|\bm H\|^2_F+(1-\tau_i)\|\bm B_{i-1}\|^2_F+\tau_i\|\bm C_{i-1}\|^2_F
+\tau_i(1-\tau_i)\left(\|\bm B_{i-1}\|^2_F+\|\bm C_{i-1}\|^2_F\right)\right)
\\
&\leq 2\left(\|\bm H\|^{2}_F+ (1+\tau_i(1-\tau_i))\max_{\mathcal D_M}\|\cdot\|^2_{F}\right)
\\&\leq 2\left(\|\bm H\|^2_F+\frac{5}{16}M^2\right),
\end{align*}
 observing that $\tau_i\in (0,1]$ hence $\tau_i(1-\tau_i)\leq \frac{1}{4}$. Thus, for $\bm H=\bm B^{\star}$, a global minimizer of $\Phi$,
 \[
\begin{aligned}
    \frac{\Phi(\bm B_k)-\Phi(\bm B^{\star})}{A_k}+\sum_{i=1}^k\frac{1-L\beta_i}{2A_i\beta_i}\left\|\bm B_i-\bm A_{i}\right\|^2 _F 
&\leq 
\frac{\|\bm C_0-\bm B^{\star}\|_F^2}{2\gamma_1}\\
&+
\sum_{i=1}^kA_{i}^{-1}\left(\delta'
\mspace{-2mu}+\mspace{-2mu}{L'\tau_i}\left(\|\bm B^{\star}\|^2_F+\frac{5}{16}M^2\right)
\right).
\end{aligned}
\]
By construction, $\sum_{i=1}^kA_{i}^{-1}{L'\tau_i}=\frac{L'}{A_k}$, and $\Phi(\bm B_k)-\Phi(\bm B^{\star})\geq 0$. It follows that 
\begin{align*}
&\min_{i=1}^k \left\|\beta_i^{-1}\left(\bm B_i-\bm A_i\right)\right\|^2_F\\&\leq 2\left(\sum_{i=1}^k\frac{\beta_i\left(1-L\beta_i\right)}{A_i}\right)^{-1}\left(
\frac{\|\bm C_0-\bm B^{\star}\|_F^2}{2\gamma_1}
+
\sum_{i=1}^kA_{i}^{-1}\delta'
\right.
\left.+\frac{L'}{A_k}\left(\|\bm B^{\star}\|^2_F+\frac{5}{16}M^2d_0^2d_1^2\right)\right).
\end{align*}
As $\beta_i=\frac{L}{2}$, $\gamma_1=\frac 1{4L}$, and $A_i=\frac{2}{i(i+1)}$, $\sum_{i=1}^k\frac{\beta_i\left(1-L\beta_i\right)}{A_i}=\frac{1}{4L}\sum_{i=1}^k{A_i^{-1}}=\frac{k(k+1)(k+2)}{24 L}$, so 
\begin{align*}
\min_{i=1}^k \left\|\beta_i^{-1}\left(\bm B_i\mspace{-2mu}-\mspace{-2mu}\bm A_i\right)\right\|^2_F\leq \frac{96L^2}{k(k\mspace{-2mu}+\mspace{-2mu}1)(k\mspace{-2mu}+\mspace{-2mu}2)}\|\bm C_0\mspace{-2mu}-\mspace{-2mu}\bm B^{\star}\|^2_F +8L\delta'+\frac{24 LL'}{N}\left(\|\bm B^{\star}\|^2_F+\frac{5M^2}{16}\right).  
\end{align*}
This proves the claimed result in the non-convex setting. 

In the convex regime, recall from the prior discussion that we may set $L'=0$ in the previous display, proving the claim.

\section{Additional Results}

\subsection{Proof of \texorpdfstring{\cref{lem:UpsilonDerivative}}{Lemma 3}}
\label{app:UpsilonDerivative}

The proof of \cref{lem:UpsilonDerivative} follows from the following lemma coupled with the chain rule for Fr{\'e}chet differentiable maps. 

\begin{lemma}
    \label{lem:DerivativeofIntegral}
    Let $\mu_i\in\mathcal P(\mathbb R^{d_i})$, for $i=0,1$, be compactly supported with $\supp(\mu_i)=S_i$. Then, the map $f\in \mathcal C(S_0\times S_1)\mapsto \left( \int e^{f(\cdot,y)}d\mu_1(y),\int e^{f(x,\cdot)}d\mu_0(x)\right)\in\mathcal C(S_0)\times \mathcal C(S_1)$ is smooth  with 
 first derivative  at $f\in\mathcal C(S_0\times S_1)$ given by
    \[
        h\in \mathcal C(S_0\times S_1)\mapsto \left(\int h(\cdot,y)e^{f(\cdot,y)}d\mu_1(y),\int h(x,\cdot)e^{f(x,\cdot)}d\mu_0(x)\right)\in\mathcal C(S_0)\times \mathcal C(S_1).
    \]
\end{lemma}
\begin{proof}
    First, we show that the map $f\in\mathcal C(S_0\times S_1)\mapsto e^f\in \mathcal C(S_0\times S_1)$ is Fr{\'e}chet differentiable with $D(e^{(\cdot)})_{[f]}(h)=he^f$. Fix $f\in\mathcal C(S_0\times S_1)$ and consider 
    \[
        \lim_{\substack{h\in\mathcal C(S_0\times S_1)\\\|h\|_{\infty,S_0\times S_1}\to 0}}\frac{\left\|e^{f+h}-e^f-he^f\right\|_{\infty,S_0\times S_1}}{\|h\|_{\infty,S_0\times S_1}}\leq \|e^f\|_{\infty,S_0\times S_1} \lim_{\substack{h\in\mathcal C(S_0\times S_1)\\\|h\|_{\infty,S_0\times S_1}\to 0}}\frac{\left\|e^{h}-1-h\right\|_{\infty,S_0\times S_1}}{\|h\|_{\infty,S_0\times S_1}}.
    \]
   Fix arbitrary $(x,y)\in S_0\times S_1$. By a Taylor expansion, 
   \[
        e^{h(x,y)}=1+h(x,y)+\frac 12e^{\xi(x,y)}h^2(x,y),
   \]
   where $|\xi(x,y)|\in [0,|h(x,y)|]$ i.e. $\|\xi\|_{\infty, S_0\times S_1}\leq \|h\|_{\infty, S_0\times S_1}$. That is, 
   \[\lim_{\substack{h\in\mathcal C(S_0\times S_1)\\\|h\|_{\infty,S_0\times S_1}\to 0}}
        \frac{\left\|e^{h}-1-h\right\|_{\infty,S_0\times S_1}}{\|h\|_{\infty,S_0\times S_1}}\leq\lim_{\substack{h\in\mathcal C(S_0\times S_1)\\\|h\|_{\infty,S_0\times S_1}\to 0}} \frac{1}{2}e^{\|\xi\|_{\infty,S_0\times S_1}}\|h\|_{\infty,S_0\times S_1}=0. 
   \]

On the other hand,
   the derivative of $f\in\mathcal C(S_0\times S_1)\mapsto \int f(x,y)d\mu_1(y)\in \mathcal C(S_0)$ at any point is given by $h\in\mathcal C(S_0\times S_1)\mapsto \int h(x,y)d\mu_1(y)\in \mathcal C(S_0)$. The claimed expression for the first derivative then follows by the chain rule.   The derivatives of this map can be computed to arbitrary order inductively by the prior argument.   
\end{proof}

\begin{proof}[Proof of \cref{lem:UpsilonDerivative}]
Observe that the map $(\bm A,\varphi_0,\varphi_1)\in \mathbb R^{d_0\times d_1}\times \mathfrak E\mapsto \varphi_0\oplus \varphi_1-c_{\bm A}\in\mathcal C(S_0\times S_1)$ is smooth with first derivative at $(\bm A,\varphi_0,\varphi_1)\in \mathbb R^{d_0\times d_1}\times \mathfrak E$ given by 
\[
(\bm B,h_0,h_1)\in \mathbb R^{d_0\times d_1}\times \mathfrak E\mapsto h_0\oplus h_1+32x^{\intercal}\bm By\in \mathcal C(S_0\times S_1).
\]
The result then follows from \cref{lem:DerivativeofIntegral} by applying the chain rule. 
\end{proof}

\subsection{Compactness of \texorpdfstring{$\mathcal L$}{L}}

\label{sec: compactness}

\begin{lemma}[Example 2 in \citealt{yosida1995functional}]
\label{lem:integralHS}
Let $\varepsilon>0$, $\mu_0\in \cP(\RR^{d_0})$, $\mu_1\in\cP(\RR^{d_1})$, and $\bm A\in\RR^{d_0\times d_1}$ be arbitrary and let $(\varphi_0^{\bm A},\varphi_1^{\bm A})$ be EOT potentials for $\OT_{\bm A,\varepsilon}(\mu_0,\mu_1)$. Then,
the map $\mathcal L:L^2(\mu_0)\times L^2(\mu_1)\mapsto L^2(\mu_0)\times L^2(\mu_1)$ defined by  
    \[
    \mathcal L(f_0,f_1)=\left(\int f_1(y)e^{\frac{\varphi_0^{\bm A}(x)+\varphi_1^{\bm A}(y)-c_{\bm A}(x,y)}{\varepsilon}}d\mu_1(y),\int f_0(x)e^{\frac{\varphi_0^{\bm A}(x)+\varphi_1^{\bm A}(y)-c_{\bm A}(x,y)}{\varepsilon}}d\mu_0(x)\right),
    \]
    is compact.
\end{lemma}
\begin{proof}
    For simplicity, we prove only that 
    \[
    \mathcal L_2:f\in L^2(\mu_0)\mapsto \int f(x)\xi(x,\cdot)d\mu_0(x)\in L^2(\mu_1),
    \]
    is a compact operator for $\xi:(x,y)\in\RR^{d_0}\times \RR^{d_1}\mapsto e^{\frac{\varphi_0^{\bm A}(x)+\varphi_1^{\bm A}(y)-c_{\bm A}(x,y)}{\varepsilon}}$. For any $y\in \RR^{d_1}$ and $f\in L^2(\mu_0)$, $\left|\mathcal L_2(f)(y)\right|^2\leq \|f\|_{L^2(\mu_0)}^2\int|\xi(\cdot,y)|^2d\mu_0$, as $\xi(\cdot,y)$ is bounded on $\supp(\mu_0)$ this operator is well-defined.

    Let $f_n$ be a bounded sequence in $L^2(\mu_0)$. By the Eberlein-\v Smulian theorem \citep[p. 141]{yosida1995functional}, up to passing to a subsequence, $f_n$ converges weakly to $f\in L^2(\mu_0)$. For fixed $y\in\RR^{d_1}$, $\xi(\cdot,y)\in L^2(\mu_0)$, hence $\mathcal L_2(f_n)(y)\to \mathcal L_2(f)(y)$ and it follows from the dominated convergence theorem that, for any $g\in L^2(\mu_1)$, $
        \int \mathcal L_2(f_n)gd\mu_1\to\int \mathcal L_2(f)gd\mu_1$ such that $\mathcal L_2(f_n)\to \mathcal L_2(f)$ weakly in $L^2(\mu_1)$. Also, by dominated convergence, 
        \[
        \|\mathcal L_2(f_n)\|_{L^2(\mu_1)}^2=\int \mathcal L_2(f_n)^2d\mu_1\to\int \mathcal L_2(f)^2d\mu_1=\|\mathcal L_2(f)\|_{L^2(\mu_1)}^2,
        \] 
         such that $\mathcal L_2(f_n)\to \mathcal L_2(f)$ strongly in $L^2(\mu_1)$. As $f_n$ was an arbitrary bounded sequence in $L^2(\mu_0)$ and $\mathcal L_2(f_n)\to \mathcal L_2(f)$ strongly in $L^2(\mu_1)$ up to a subsequence, $\mathcal L_2$ is a compact operator. 
\end{proof}

\end{document}